\documentclass[12pt,longbibliography]{article}

\usepackage[utf8]{inputenc}

\usepackage[T1]{fontenc}

\usepackage[a4paper, margin=2.7cm]{geometry}

\usepackage{amsmath, amsthm, amsfonts, amssymb}
\usepackage{mathrsfs}           
\usepackage{mathtools}

\usepackage[colorlinks=true,linkcolor=blue,citecolor=blue,urlcolor=blue,breaklinks]{hyperref}

\usepackage{bbm}

\usepackage{cleveref}

\usepackage{breakurl}
\usepackage{natbib}
\usepackage{url}
\def\R{\mathbb{R}}

\def\1{\mathbbm{1}}

\def\grad{\nabla}
\renewcommand{\div}{\operatorname{div}}
\DeclareMathOperator{\diam}{diam}
\DeclareMathOperator{\dist}{dist}

\def\d{\,\mathrm{d}}
\def\dx{\d x}

\def \ddt{\frac{\mathrm{d}}{\mathrm{d}t}}
\def \ddtau{\frac{\mathrm{d}}{\mathrm{d}\tau}}

\def\p{\partial}

\def\ird{\int_{\R^d}}

\def\:{\colon}

\def\lambdaL{\lambda_{\mathrm{L}}}
\def\lambdaP{\lambda_{\mathrm{P}}}

\newtheorem{thm}{Theorem}[section]
\newtheorem{cor}[thm]{Corollary}
\newtheorem{lem}[thm]{Lemma}
\newtheorem{prp}[thm]{Proposition}
\newtheorem{hyp}[thm]{Hypothesis}

\theoremstyle{definition}

\theoremstyle{remark}
\newtheorem{rem}[thm]{Remark}
\theoremstyle{example}

\numberwithin{equation}{section}

\def\thetitle{Large-time estimates for the Dirichlet heat equation in
  exterior domains}
\def\theauthor{José A. Cañizo, Alejandro Gárriz \& Fernando Quirós}

\hypersetup{pdftitle={\thetitle}}
\hypersetup{pdfauthor={\theauthor}}

\title{\thetitle}
\author{\theauthor}

\date{}

\AtEndDocument{\vspace{2cm}{\footnotesize
  \noindent
  \textsc{José A. Cañizo. Departamento de Matemática Aplicada \& IMAG,
    Universidad de Granada. Avda. Fuentenueva S/N, 18071 Granada, Spain.}
  \textit{E-mail address}: \texttt{\href{mailto:canizo@ugr.es}{canizo@ugr.es}}
  \par
  \addvspace{\medskipamount}
  \noindent
  \textsc{Alejandro Gárriz. Departamento de Matemática Aplicada \& IMAG,
    Universidad de Granada. Avda. Fuentenueva S/N, 18071 Granada, Spain.}
  \\
  \textit{E-mail address}:
  \texttt{\href{mailto:alejandro.garriz@ugr.es}{alejandro.garriz@ugr.es}}
    \par
  \addvspace{\medskipamount}
  \noindent
  \textsc{Fernando Quirós. Departamento de Matemáticas, Universidad Autónoma de Madrid \& Instituto de Ciencias Matem\'aticas ICMAT (CSIC-UAM-UCM-UC3M), Ciudad Universitaria de Cantoblanco, 28049 Madrid, Spain.}
  \textit{E-mail address}: \texttt{\href{mailto:fernando.quiros@uam.es}{fernando.quiros@uam.es}}
}}

\begin{document}

\maketitle

\begin{abstract}
  We give large-time asymptotic estimates, both in uniform and $L^1$
  norms, for solutions of the Dirichlet heat equation in the
  complement of a bounded open set of $\R^d$ satisfying certain
  technical assumptions. We always assume that the initial datum has suitable
  finite moments (often, finite first moment). All estimates include
  an explicit rate of approach to the asymptotic profiles at the
  different scales natural to the problem, in analogy with the
  Gaussian behaviour of the heat equation in the full space. As a
  consequence we obtain by an approximation procedure the asymptotic
  profile, with rates, for the Dirichlet heat kernel in these exterior
  domains. The estimates on the rates are new even when the domain is
  the complement of the unit ball in~$\R^d$, except for previous
  results by Uchiyama in dimension $2$, which we are able to improve
  in some scales. We obtain that the heat kernel behaves
  asymptotically as the heat kernel in the full space, with a factor
  that takes into account the shape of the domain through a harmonic
  profile, and a second factor which accounts for the loss of mass
  through the boundary. The main ideas we use come from entropy
  methods in PDE and probability, whose application seems to be new in
  the context of diffusion problems in exterior domains.
\end{abstract}

\tableofcontents

\section{Introduction and main results}
\label{sec:intro}

\subsection{Precedents and informal discussion of results}
\label{sec:setting}

We consider the heat equation with Dirichlet boundary conditions in an
exterior domain $\Omega = \R^d \setminus \overline{U}$, where
$U\subset \R^d$ is a bounded open set with $\mathcal{C}^2$ boundary
such that $\R^d \setminus \overline{U}$ is connected. That is, we
study solutions to
\begin{equation}
  \label{eq:heat-ext}
  \begin{aligned}
    \p_t u = \Delta u &\qquad \text{for $t > 0$, $x \in \Omega$},
    \\
    u(0,x) = u_0(x) &\qquad \text{for $x \in \Omega$},
    \\
    u(t,x) = 0 &\qquad \text{for $t > 0$, $x \in \p\Omega$},
  \end{aligned}
\end{equation}
where $u = u(t,x)$ depends on time $t \geq 0$ and space
$x \in \overline{\Omega}$, and $u_0$ is an integrable and nonnegative
initial condition on $\Omega$. Dimension $d=1$ is special, since the
complement of any bounded set is disconnected; in this case we take
the half-line $\Omega = (x_0,+\infty)$ for some
$x_0\in\mathbb{R}$. Problem~\eqref{eq:heat-ext} has a unique classical
solution such that $u\in C([0,+\infty);L^1(\Omega))$, which we will
call the \emph{standard solution}, or \emph{$L^1$ solution}.

Standard solutions to~\eqref{eq:heat-ext} decay to 0 as time goes
by. Giving a precise description of how they do so (the so-called
\emph{intermediate asymptotics}) is a well-known problem, which is
nontrivial even in the case in which $U$ is the open unit ball. The
main difficulty lies in finding good bounds on the way in which the
mass of the solution is ``lost'' through the boundary of $\Omega$:
indeed, by Gauss's divergence theorem and Hopf's Lemma~\citep{Friedman1958}, if
$u\not\equiv0$, then
\begin{equation*}
  \ddt \int_\Omega u(t,x) \d x = -\int_{\partial \Omega} \nabla u(t,x)
  \cdot \nu(x) \d S(x) < 0
  \qquad \text{for all $t > 0$},
\end{equation*}
where $\nu = \nu(x)$ is the inner normal to $\partial \Omega$ (i.e.,
pointing towards $\Omega$) and $\d S(x)$ denotes the surface integral
on $\partial \Omega$.

In this paper we use a modification of the well-known entropy method
to obtain the intermediate asymptotics for $L^1$ solutions with
suitable finite moments. This allows us to obtain the optimal decay
rate of solutions to 0, their asymptotic profile when properly scaled
to kill the decay, and also a rate of convergence of the scaled
solution to the limit profile. Up to our knowledge, entropy methods
have not been used before to deal with problems posed in domains with
holes, where mass is not conserved. They yield several new results,
notably rates of convergence to the asymptotic profiles which are new
in many cases.

The standard solution $u$ to problem~\eqref{eq:heat-ext} can be expressed as
\begin{equation}
  \label{eq:sol-via-kernel}
  u(t,x) = \int_\Omega p_\Omega(t,x,y) u_0(y) \d y
  \qquad \text{for $t > 0$, $x \in \overline{\Omega}$,}
\end{equation}
where $p_\Omega = p_\Omega(t,x,y)$ is the \emph{heat kernel} for the
domain $\Omega$, defined so that $(t,x) \mapsto p_\Omega(t,x,y)$ is
the only positive solution to~\eqref{eq:heat-ext} with
$u_0 = \delta_y$ and such that $x \mapsto p_\Omega(t,x,y)$ is
integrable for all $t > 0$.  As an important consequence of our
results we also obtain asymptotic estimates for the heat kernel,
improving for large times the ones available in the literature.

\paragraph{Known bounds for kernels.}

Many of the existing results are written in terms of bounds for the heat kernel $p_\Omega$, since they yield bounds for general solutions through formula \eqref{eq:sol-via-kernel}. From the maximum principle one readily obtains that
\begin{equation*}
  \label{eq:p-bounded-by heat}
  p_\Omega(t,x,y) \leq p(t,x,y)
  \qquad \text{for $t > 0$, $x, y \in \overline{\Omega}$,}
\end{equation*}
where $p(t,x,y)$ is the heat kernel in the full space:
\begin{equation*}
  p(t,x,y) = \Gamma(t,x-y) :=
  (4 \pi t)^{-\frac{d}{2}} \exp\left( -\frac{|x-y|^2}{4t} \right)
  = (2t)^{-d/2} G \left( \frac{x-y}{\sqrt{2 t}} \right).
\end{equation*}
Throughout the paper we denote by $G$ the standard Gaussian function
in $\R^d$,
\begin{equation*}
  G(y) := (2 \pi)^{-d/2} \exp \left( -\frac{|y|^2}{2} \right),
  \qquad y \in \R^d,
\end{equation*}
and we use the notation $\Gamma = \Gamma(t,x)$ for the fundamental
solution to the heat equation:
\begin{equation*}
  \Gamma(t,x) := p(t,x,0) =
  (4 \pi t)^{-\frac{d}{2}} \exp\left( -\frac{|x|^2}{4t} \right)
  = (2t)^{-d/2} G \left( \frac{x}{\sqrt{2 t}} \right).
\end{equation*}
Very general upper and lower bounds for $p_\Omega$ which improve on
the above ``trivial'' bound were given by \citet{Grigoryan2001} for
$x,y$ away from $\partial \Omega$, and later completed by
\citet{Zhang2003} up to $\partial \Omega$. Their estimates are valid
in exterior domains on manifolds under quite general conditions, and
have a different qualitative behavior depending on whether the
manifold in question is parabolic or not. In the case of the Euclidean
space, $\R^d$ is parabolic if and only if $d = 1,2$, so estimates have
important differences in dimensions $1$ and $2$. As an example, in
$\R^d$ for $d \geq 3$ (the nonparabolic case), they obtain that for
some $c_1, c_2 > 0$ one has
\begin{equation}
  \label{eq:sc1}
  \frac{1}{c_1} \left(\frac{\rho(x)}{\sqrt{t} \wedge 1} \wedge 1 \right)
  \left(\frac{\rho(y)}{\sqrt{t} \wedge 1} \wedge 1 \right)
  p \big(\frac{t}{c_2}, x, y\big)
  \leq
  p_\Omega(t,x,y)
\end{equation}
for all $t > 0$ and $x,y \in \Omega$, where
$\rho(x) := \dist(x, \partial \Omega)$. They also give an upper
bound of a similar form,
\begin{equation}
  \label{eq:sc2}
  p_\Omega(t,x,y) \leq
  c_1 \left(\frac{\rho(x)}{\sqrt{t} \wedge 1} \wedge 1 \right)
  \left(\frac{\rho(y)}{\sqrt{t} \wedge 1} \wedge 1 \right)
  p(c_2 t, x, y),
\end{equation}
valid in the same range (which improves on the trivial upper bound
when $x$ or $y$ are close to $\partial \Omega$). This bounds the
kernel by above and below with two different Gaussian functions, with
different positions of the diffusive scales $|x|\sim c\sqrt{t}$ for
fixed $c>0$, where the mass of the solution is predominantly
located. We will use these known estimates in our proofs; a more
complete summary of them is given in \Cref{sec:pre-kernel-estimates},
including the more involved $d=2$ case.

Throughout the paper $\phi \: \overline{\Omega} \to \R$, the
\emph{harmonic profile}, is a positive harmonic function on $\Omega$
with zero Dirichlet boundary condition,
\begin{equation}
  \label{eq:phi}
  \Delta \phi = 0 \quad \text{in }\Omega,
  \qquad
  \phi= 0 \quad \text{in }\p\Omega,
\end{equation}
with an appropriately set behaviour as $|x| \to +\infty$. In dimension
$d \geq 3$, we choose $\phi$ as the only such function with
$\lim_{|x| \to +\infty} \phi(x) = 1$; in dimension $2$,
$\phi(x) \sim \log(|x|)$ as $|x| \to +\infty$, and in dimension $d=1$
we simply take $\phi(x) = x-x_0$; see Section~\ref{sec:phi} for estimates
on $\phi$. The main observation where $\phi$ arises naturally is that it
defines a conserved quantity for the PDE \eqref{eq:heat-ext}: for any
standard solution $u$ we have
\begin{equation}
  \label{eq:conservation}
  \ddt \int_\Omega \phi u \d x = 0,
\end{equation}
as long as this quantity is initially finite. The function $\phi$
appears in a fundamental way when studying large-time asymptotics of
the heat kernel, and will be present throughout. It can already be
used to express estimates \eqref{eq:sc1}--\eqref{eq:sc2}, valid for
$d \geq 3$, in a concise way: since $\phi$ is positive and bounded in
$\Omega$, and behaves like the distance to the boundary close to it,
equations \eqref{eq:sc1}--\eqref{eq:sc2} are equivalent for $t \geq 1$ to
\begin{equation*}
  \label{eq:known_bound}
  \phi(x) \phi(y)
   p(t / c_2 ,x,y)
  \lesssim  p_\Omega(t,x,y) \lesssim
  \phi(x) \phi(y)
  p(c_2t,x,y),
  \qquad t \geq 1,
\end{equation*}
where $\lesssim$ denotes inequality up to a multiplicative constant.
In dimensions $d\ge 3$ the harmonic profile has a nice probabilistic
interpretation: $\phi(x)$ gives the probability that a particle that
is initially located at $x\in\Omega$ evolving with Brownian motion
never touches the complement of $\Omega$. If we think that particles
are killed when exiting $\Omega$, it can be regarded as a survival
probability.

\paragraph{Asymptotic results for kernels.}

On the other hand, there are several results in the literature on the
asymptotic behaviour of the kernel $p_\Omega$ as $t\to+\infty$. Again,
results strongly depend on whether $d =1,2$ or $d \geq 3$, so we will
discuss first the case $d \geq 3$. It was proved by
\citet{Collet2000a} that in dimension $d \geq 3$ and for fixed
$x, y \in \Omega$,
\begin{equation}
  \label{eq:asymp-kernel-estimate}
  p_\Omega(t,x,y) \sim \phi(x) \phi(y) p(t,x,y) \sim \phi(x) \phi(y)(4\pi t)^{-d/2}\qquad\text{as $t \to +\infty$},
\end{equation}
where $f \sim g$ denotes that the limit $f/g$ is
$1$ in the asymptotic regime considered. We give a new proof of this
in our main kernel bound, Corollary~\ref{cor:heat_kernel_bounds}. Our
bound includes an explicit rate of the above asymptotic approach,
which seems to be a new result in dimensions $d \geq 3$. To be more
precise, we show that in $d \geq 3$, for some
$\sigma = \sigma(y) > 0$,
\begin{equation}
  \label{eq:main-kernel-estimate}
  |p_\Omega(t,x,y) - \phi(x) \phi(y) p(t,x,y)|
  \lesssim \phi(x) \phi(y) t^{-\frac{d}{2} - \sigma},
\end{equation}
for all $t > 0$ and all $x,y \in \Omega$. (Since the kernels
$p_\Omega$ and $p$ are symmetric in $x$, $y$, one may just as well
write $\sigma(x)$.) For $x$, $y$ in any fixed compact set, this gives
an additional asymptotic term in~\eqref{eq:asymp-kernel-estimate},
showing that
\begin{equation*}
  p_\Omega(t,x,y) = \phi(x) \phi(y) p(t,x,y) (1 + O(t^{-\sigma}))
  \qquad \text{as $t \to +\infty$},
\end{equation*}
a result that gives a rate of convergence in relative error. This can
be compared to Theorem 1.4 of \citet{Collet2000a} or Theorem~4 in
Section~2.4 of \citet{Uchiyama-JTP-2018}, which give $o(1)$ instead of
$O(t^{-\sigma})$. Our estimate~\eqref{eq:main-kernel-estimate} also
gives information on the diffusive scale in which $x = z \sqrt{2t}$:
\begin{equation*}
  p_\Omega (t, z \sqrt{2t}, y) =
  (2t)^{-\frac{d}{2}} \phi(y)G(z) (1 + O(t^{-\sigma}))
  \qquad \text{as $t \to +\infty$},
\end{equation*}
uniformly for $y, z$ in a compact set. An interpretation of this is
that the Dirichlet fundamental solution starting at $y$ outside a
domain in dimensions $d \geq 3$ has a self-similar behavior comparable
to that of the heat equation on $\R^d$: it converges to $0$, but after
a diffusive rescaling it approaches a multiple of the Gaussian,
corrected by a factor $\phi(y)$ which accounts for the fact that a
mass $1 - \phi(y)$ is asymptotically lost through the boundary.

There's still another interesting asymptotic regime contained
in~\eqref{eq:main-kernel-estimate}: when both $x$ and $y$ depend on
$t$, and both $|x|$ and $|y|$ diverge. First, we need to assume
$|x-y| \lesssim \sqrt{t}$ (or perhaps a slightly weaker assumption)
for \eqref{eq:main-kernel-estimate} to be useful: if its right-hand
side decays slower than $p(t,x,y)$ then it does not contain any
asymptotic information. Also, in order to obtain useful information
from our estimate we need to know something about the dependence of
$\sigma(y)$ on $y$. The constant $\sigma$ we give depends on the
properties of certain rather explicit logarithmic Sobolev
inequalities, and we believe that $\sigma$ can actually be taken to be
a constant on $\Omega$. We have not been able to prove this, and it is
an interesting question which can be studied independently; however,
if one accepts this for a moment, then \eqref{eq:main-kernel-estimate}
implies, since $\phi(x), \phi(y) \to 1$ at known rates,
\begin{equation*}
  p_\Omega(t,x,y) = p(t,x,y) (1 + O(t^{-\sigma}))
\end{equation*}
when $|x|, |y| \gtrsim \sqrt{t}$ and $|x-y| \lesssim \sqrt{t}$. That is: if both $x$ and $y$ move to infinity at any (fast enough) speed but stay within ``diffusive distance'' of each other, then asymptotically the effect of the hole $U$ is not seen in the kernel. However, as remarked, we cannot completely prove this with our methods since we have not proved whether one may take $\sigma$ independent of $y$. This kind of convergence result, without a rate, is also contained in \citet[Theorem 4]{Uchiyama-JTP-2018}, with a different set of restrictions: Uchiyama requires $|x|, |y| \lesssim t$ but does not place further restrictions on $|x-y|$.

Let us also discuss briefly our results on the kernel in dimension $d =2$. For this, assume that $0 \in U$. In Corollary~\ref{cor:heat_kernel_bounds} we show that in $d=2$,
\begin{equation}
  \label{eq:intro-d2-kernel-bound}
  \left|p_\Omega(t,x,y)- \frac{4 \phi(x) \phi(y)}{(\log t)^2} p(t,x,y)\right|
  \lesssim
  \frac{\phi(x)\phi(y)}{t (\log t)^2}\left( \frac{1}{\log t}+\frac{|x| \wedge |y|}{t^{\sigma}}\right),
\end{equation}
for all $t \geq 2$ and all $x,y \in \Omega$ with
$|x|,|y| \lesssim \sqrt{t}$. We recall that in dimension $2$ we have
$\phi(x) \sim \log|x|$ as $|x| \to +\infty$. Here,
$\sigma = \sigma(y) > 0$ is still a quantity depending on $y$, which
we conjecture can be taken independent of $y$. As
before,~\eqref{eq:intro-d2-kernel-bound} shows that for $x$, $y$ in
any fixed compact set,
\begin{equation*}
  p_\Omega(t,x,y) = \frac{4}{(\log t)^2}\phi(x) \phi(y) p(t,x,y)\left(1 + O \left( \frac{1}{\log t} \right) \right).
\end{equation*}
This is an improvement over \citet[Theorem 1.2]{Collet2000a}, which gives $o(1)$ instead; and over \citet[Theorem 3]{Uchiyama-JTP-2018},
which gives $O(\log \log t / \log t)$ and requires additional error terms. Similarly, in the diffusive scale,
\begin{equation*}
  p_\Omega (t, z \sqrt{2t}, y) = \frac{4 \phi(y)}{\log t}(2t)^{-\frac{d}{2}} G(z)\left(1 + O \left( \frac{1}{\log t} \right) \right)
\end{equation*}
as $t \to +\infty$, for $y$, $z$ in any fixed compact set. On the
other hand, if both $x$ and $y$ depend on time and diverge,
now~\eqref{eq:intro-d2-kernel-bound} is only useful as long as
$|x| \wedge |y| = o(t^{\sigma})$ and $|x| \vee |y| = O(\sqrt{t})$;
this is more restrictive than the mentioned result
by~\citet{Uchiyama-JTP-2018}, who only requires
$|x| \vee |y| = O(\sqrt{t})$, since our constant $\sigma$ is always
less than $1/2$ (possibly equal to $1/2$ in some cases or
asymptotically, but this is an open question).

If any of $|x|$ and $|y|$ diverge faster than $\sqrt{t}$,
then~\eqref{eq:intro-d2-kernel-bound} does not give information on
asymptotics: for the right-hand side to decay faster than
$(\log t)^{-2} \phi(x) \phi(y) p(t,x,y)$ we essentially need
$|x-y| \lesssim \sqrt{t}$; but for the term
$(|x| \wedge |y|) t^{-\sigma}$ to decay to $0$, at least one of $|x|$
or $|y|$ has to be $o(t^{\sigma})$, and $\sigma < 1/2$ in our
results. This is different from the case $d \geq 3$, where the same
estimate can give asymptotics in scales faster than the diffusive
one. The only result we know in that case is Theorem~1
in~\citet{Uchiyama-JTP-2018}, which states that
\begin{align}
  \label{eq:uchi1}
  p_\Omega(t,x,y)
  &=\frac{\phi(x)}{\log
    \frac{t}{|y|}}
    p(t,x,y)\left(1+o(1)\right)
    \quad&&\text{if $\sqrt{t} < |y|$, $|y| = o(t)$, $|x| |y| \lesssim t$,}
  \\
  \label{eq:uchi2}
  p_\Omega(t,x,y)&=p(t,x,y)\left(1+o(1)\right)\quad &&\text{if $|x| |y| \gtrsim t$ and $|x| \to +\infty$.}
\end{align}
We believe that the strategy we follow in this paper can also give
estimates of the rate in the above approximation. However, we leave
this for a future work to avoid adding to an already long paper.

\paragraph{Asymptotic results for solutions}

Using the above kernel estimates one can obtain results for general
integrable initial data via \eqref{eq:sol-via-kernel}. Some previous
papers have also tried the approach of obtaining bounds directly on
solutions $u$. We cite the interesting paper by \citet{Herraiz1998},
who gives large-time estimates for the Dirichlet heat equation on
exterior domains, without convergence rates. His results also include
asymptotics for solutions which do not have a finite first moment, and
are proved mainly by comparison arguments with super and subsolutions;
see also~\citet{DominguezTena-RodrigueBernal-2024}, dealing as well
with Neumann and Robin homogeneous boundary conditions.
Similar asymptotic bounds have also been studied for a linear nonlocal
heat equation in the series of papers by
\citet*{CEQW2012,CEQW2015-DCDS,CEQW2016-JMAA,CEQW2016-SIMA} and for the (local)
porous medium equation in
\citet*{BQV2007,GG2007,CQW2017,CQW2018}. These bounds have been an
important initial motivation for our work.

We also mention recent results (including Neumann and Robin boundary
conditions, besides Dirichlet ones) on the asymptotic behaviour of the
mass of standard solutions by
\citet{DominguezTena-RodrigueBernal-2023} using a different
approach. In the Dirichlet case and with suitable finite moments, they
are improved in our Corollary~\ref{cor:mass}, after which we make some
further comments.

In the present paper we are also able to give large-time asymptotics in the $L^1$ sense, which imply asymptotics on the mass and seem to be new. We also highlight that our strategy gives a unified approach in all dimensions, and we hope it can lend itself to generalisations in
other contexts.

In the rest of this introduction we describe our results in more detail.

\subsection{Main uniform estimates}

Let us state our results more precisely. For dimensions $d \geq 2$ we always assume that
\begin{equation}
  \label{eq:hypOmega1}
  \begin{gathered}
    \Omega := \R^d \setminus \overline{U} \ \text{is connected,}
    \\
    \text{$U \subseteq \R^d$ is nonempty, bounded, open, and with $\mathcal{C}^2$ boundary}.
  \end{gathered}
\end{equation}
As remarked before, in $d=1$ we just take $\Omega := (x_0,+\infty)$, $x_0\in\mathbb{R}$. In order to estimate the constant $\sigma$ mentioned in
\eqref{eq:main-kernel-estimate} and \eqref{eq:intro-d2-kernel-bound} we need to assume that there is a positive lower bound on the logarithmic Sobolev constant of a family of densities related to the asymptotic limit:
\begin{hyp}
  \label{hyp:logsob}
  There is a constant $\lambda>0$ such that \emph{all} the probability
  densities $F_\tau \: \Omega_\tau \to (0,+\infty)$, $\tau \geq 0$,
  defined by
  \begin{equation*}
    F_\tau(x) := K_\tau \phi^2(x e^{\tau}) G(x), \qquad x \in\Omega_\tau:= e^{-\tau} \Omega,
  \end{equation*}
  where $K_\tau$ is a normalisation constant such that $F_\tau$ has
  integral $1$, satisfy the logarithmic Sobolev inequality
  \begin{equation}\label{eq:log_sob-hyp.1}
    \lambda \int_{\Omega_\tau} g \log \frac{g}{F_\tau}
    \leq \int_{\Omega_\tau}g\left|\nabla \log\frac{g}{F_\tau}\right|^2
  \end{equation}
  for all positive, integrable $g \: \Omega_\tau \to (0,+\infty)$ such
  that $\int_{\Omega_\tau} g(x) \d x = 1$ and such that the right hand
  side is finite. Without loss of generality, in dimensions $d \geq 3$
  we always assume $\lambda < d-2$, and in general we always assume
  $\lambda < 2$.
\end{hyp}
The functions $F_\tau$ are a sort of ``transient equilibria''
motivated by the change of variables that we will use. We refer to
Section~\ref{sec:entropy} for a better explanation of the significance
of this, but we will make a few remarks:
\begin{enumerate}
\item The above hypothesis is unnecessary in dimension $d=1$ with $\Omega = (x_0,+\infty)$, $x_0\in\mathbb{R}$ (that is, it is always true), and we may take $\lambda = 2$ in that case (see Lemma \ref{lem:logsob-1d}).

\item We show in Section \ref{sec:logsob} that this hypothesis holds for a large family of domains, namely those for which the ``hole'' $U$ is a smooth deformation of the unit ball. We do not know the exact family of domains for which Hypothesis~\ref{hyp:logsob} is valid, nor a way to obtain good bounds on the value of $\lambda$. We expect however that the logarithmic Sobolev constant $\lambda_\tau$   corresponding to $F_\tau$ approaches the logarithmic Sobolev constant for the standard Gaussian function on $\R^d$ as $\tau \to +\infty$; hence we expect $\lim_{\tau \to +\infty} \lambda_\tau = 2$. The constant $\lambda$ which we use throughout is just a uniform lower bound of all constants $\lambda_\tau$ for $\tau \geq 0$.

\item We have not been able to show that one can also take $\lambda$ to be invariant by translations of the domain $\Omega$, but we believe this to be true.

\item Finally, the condition $\lambda < d-2$ is a technical one to
  simplify the exposition (and we can always satisfy it by taking a
  smaller $\lambda$ if needed). As remarked in the previous point, if
  it is true that the best possible $\lambda$ satisfies
  $\lambda \leq 2$ as we expect, then $\lambda < d-2$ does not add any
  restriction when $d \geq 5$. If one wants to optimise the rates of
  convergence (and assuming we have better information on $\lambda$)
  one might be able to take the best possible $\lambda$ (ignoring
  $\lambda < d-2$) and obtain slightly improved rates. As remarked
  above, we believe the optimal strategy would be to estimate
  $\lambda_\tau$ as best as possible, and use $\lambda_\tau$
  throughout. See Remark \ref{rem:lambda-value} and
  equation~\eqref{eq:46} for more on this, and Remark
  \ref{rem:optimal_rates} for the optimal decay rates we expect to
  hold.
\end{enumerate}

\paragraph{Some notation.}

Before beginning the exposition of our results, let us define for the
rest of the article the following quantities, which will be relevant
throughout all of our study.  We recall that we choose $\phi$ as the
unique positive harmonic function with Dirichlet boundary conditions
on $\Omega$ such that
  \begin{align}
    \label{eq:phi-limd2}
    \lim_{x \to +\infty} \frac{\phi(x)}{\log|x|} = 1
    & \qquad \text{ in dimension $d = 2$,}
    \\
    \label{eq:phi-limd3}
    \lim_{|x| \to +\infty} \phi(x) = 1
    & \qquad \text{ in dimensions $d \geq 3$.}
  \end{align}
  In dimension $d=1$ we will consider the domain
  $\Omega=(x_0,\infty)$, $x_0\in\mathbb{R}$, and $\phi(x)=x-x_0$.
  The existence and uniqueness of this function $\phi$ is classical
  and is outlined in Lemma~\ref{lem:phi}.

  We denote by
\[
  m_\phi:= \int_\Omega  u_0(x)\phi(x) \d x
\]
the preserved quantity, which we may call \textit{harmonic mass}, since the initial datum is weighted against the harmonic function $\phi$. We also define
\begin{equation*}\label{eq:Mk}
  \begin{aligned}
    &m_k :=\int_\Omega u_0(x) |x|^k \d x,
    &&\qquad M_k :=\int_\Omega u_0(x) (1 + |x|^k) \d x,
    \\
    &m_{k,\phi} := \int_\Omega  u_0(x)|x|^k\phi(x) \d x,
    &&\qquad M_{k,\phi} :=\int_\Omega u_0(x) (1 + |x|^k)\phi(x) \d x,
  \end{aligned}
\end{equation*}
for any $k \geq 0$ (so $m_\phi \equiv m_{0,\phi}$). These quantities
may be $+\infty$ but are always well defined since $u_0$ is
nonnegative. The quantity $m_0$ is the initial mass, the $m_k$,
$k\ge1$, are moments of the initial data, and the $M_k$ are equal to
$m_0 + m_k$. The quantities $m_{k,\phi}$ and $M_{k,\phi} = m_\phi +
m_{k,\phi}$ are weighted
``harmonic'' versions of these.

Analogously, sometimes we write $M_k(t)$ to denote the corresponding quantity at time $t$,
\[
M_k(t) :=\int_\Omega u(t,x) (1 + |x|^k) \d x,
\]
and similarly for the other moments.

\medskip

Our main result is the following:
\begin{thm}[Uniform estimates of solutions]
  \label{thm:main-uniform}
  In dimension $d=1$ take $\Omega = (x_0,+\infty)$,
  $x_0\in\mathbb{R}$; in dimension $d \geq 2$, assume
  $\Omega \subseteq \R^d$ is an exterior domain
  satisfying~\eqref{eq:hypOmega1} and Hypothesis~\ref{hyp:logsob}. Let
  $u$ be the standard solution to the heat
  equation~\eqref{eq:heat-ext} in $\Omega$ with nonnegative initial
  condition $u_0 \in L^1(\Omega;(1+\phi(x))\dx)$. We define the
  normalisation function $k_t$ (which depends only on $\Omega$) by
  \begin{equation}
    \label{eq:kt-def-intro}
    k_t \int_\Omega \phi(x)^2 \Gamma(t,x) \d x = 1,
    \qquad t > 0.
  \end{equation}
  Then there exists a constant $C > 0$ which depends only on $d$ and
  the domain $\Omega$ such that for all $t \geq 2$ and $x\in \Omega$ we
  have:
  \begin{enumerate}
  \item[\rm (i)] In dimension $d \geq 3$
    \begin{equation*}
      |u(t,x) - m_\phi \phi(x) \Gamma(t,x)|
      \leq
      C  \phi(x) M_{1,\phi}\, t^{-\frac{d}{2}- \frac{\lambda}{4}}.
    \end{equation*}

  \item[\rm (ii)] In dimension $d=2$, choose
    $x_0\in\R^2\setminus\Omega$. Then,
    \begin{equation*}
      \left|u(t,x)-k_tm_\phi\phi(x)\Gamma(t,x)\right|
      \leq
      \frac{C \phi(x)}{t (\log t)^2}
      \left( \frac{m_{\phi} }{\log t}
        + \frac{M_{1,\phi} + m_\phi |x_0|}
        {t^{\lambda/4}}
      \right).
    \end{equation*}

  \item[\rm (iii)] In dimension $d=1$ we consider
    $\Omega = (x_0,+\infty)$. Take $M > 0$. Then, for all $t \geq 2$
    such that $M \sqrt{t} \geq |x_0|$,
    \begin{equation*}
      \left|u(t,x) - k_t m_\phi \phi(x) \Gamma(t,x) \right|
      \leq
      \frac{C\phi(x)}{t^2}
      (M_{1,\phi} + m_\phi |x_0|).
    \end{equation*}
\end{enumerate}
The constant $C$ in the inequalities above is invariant by
translations of $\Omega$ in all dimensions $d \geq 1$.
\end{thm}

\begin{rem}[Initial data]
  Unless $M_{1,\phi}$ is finite, Theorem~\ref{thm:main-uniform} gives
  no information. The behaviour of the heat equation in the full space
  suggests that this is not a technical restriction. Indeed, in that
  case no convergence speed can be found without further information
  on the data other than integrability of $u_0$, as shown in the
  counterexample constructed in~\cite[Section
  4.1]{vazquez2017asymptotic}. In general, in all of $\R^d$, the speed
  of approach to the fundamental solution can be slow if the tail of
  the initial data is integrable but ``thick'' enough; see the
  explicit spectrum of the Fokker-Planck operator in spaces with
  different power weights by \citet[Appendix A]{Gallay2001}.

  In dimension $d \geq 3$, it is clearly enough to require
  $L^1(\Omega;(1+|x|)\dx)$ since $\phi$ is bounded. The condition
  $u_0 \in L^1(\Omega;1+\phi(x)\dx)$ imposes some restriction to the
  size of the solution at infinity in dimension $d=1,2$.
\end{rem}

\begin{rem}[Optimal rates]
  \label{rem:optimal_rates}
  As remarked just after Hypothesis \ref{hyp:logsob}, we expect
  $\lim\limits_{t \to \infty} \lambda_t = 2$, which would mean that the
  optimal rates in the previous theorem should be obtained when taking
  $\lambda=2$ in our proofs (and ignoring our condition that
  $\lambda < \min\{2, d-2\}$ in dimensions $d \geq 3$, which we make
  just for convenience). This leads us to conjecture the following
  behaviour:
  \begin{itemize}
  \item In dimension $d=3$, Theorem~\ref{thm:main-uniform} (i) should hold with
    $t^{-\frac{3}{2}-\frac14} = t^{-\frac74}$ on the right-hand side.
    This is in contrast to $t^{-2}$, which is the optimal rate in the
    full space~$\R^3$.
  \item In dimension $d=4$, Theorem~\ref{thm:main-uniform} (i) should hold with $t^{-\frac52} \log (2+t)$
    on the right-hand side. For comparison, $t^{-\frac52}$ is the
    optimal decay rate for general solutions in the full space $\R^4$.
  \item In dimension $d \geq 5$, Theorem~\ref{thm:main-uniform} (i) should hold with $t^{-\frac{d}{2}-\frac12}$
    on the right-hand side, which matches the decay rate in the full
    space $\R^d$.
  \end{itemize}
  All this depends on whether it actually holds that
  $\lim_{t \to +\infty} \lambda_t = 2$ (and even on the rate at which
  this convergence takes place). This is an interesting problem, but
  requires a better understanding than the one currently available on
  perturbation results for log-Sobolev inequalities.
\end{rem}

\begin{rem}
  We prove in Section \ref{sec:Kt} that in dimension $d=2$ the
  normalisation function $k_t$ satisfies
  \begin{equation*}
    k_t \sim \frac{4}{(\log t)^2} \quad \text{as $t \to +\infty$.}
  \end{equation*}
  This decay of $k_t$ is related to the decay of the mass of the
  solution $u$ in dimension~$d=2$; see
  Section~\ref{sec:total_mass}. The constant implicit in this approach
  is not invariant by translations of the domain, which is why we have
  not written it in Theorem \ref{thm:main-uniform}; if one is not
  worried about translation invariance, one may substitute $k_t$ by
  $4 / (\log t)^2$ in the $d=2$ case of the theorem. Also giving up
  translation invariance, we can let the constant $C$ depend on
  $|x_0|$, getting
  \begin{equation*}
    \left|u(t,x)- \frac{4 m_\phi}{(\log t)^2} \phi(x)\Gamma(t,x)\right|
    \lesssim
    \frac{\phi(x)}{t (\log t)^2}
    \left( \frac{m_{\phi} }{\log t}
      + \frac{M_{1,\phi}}
      {t^{\lambda/4}}
    \right)
    \lesssim
    \frac{\phi(x) M_{1,\phi}}{t (\log t)^3}
  \end{equation*}
  for all $x \in \Omega$ and $t \geq 1$.
\end{rem}

\begin{rem}
  When $d=1$ we have $k_t = 1/t$, $\phi(x)=x-x_0$. Therefore, in the
  case $x_0=0$,
$$
k_t\phi(x)\Gamma(t,x)=\frac{x}{t} \Gamma(t,x)=2 D(t,x),
$$
where $D(t,x)=-\partial_x \Gamma(t,x)$ is the so-called \emph{dipole
  solution} of the heat equation, which has $-\delta_0'$ as initial
datum. The name comes from electromagnetism, where $\delta_0'$
represents a dipole. Hence,
\begin{equation*}
  \left|u(t,x) - 2 m_\phi D(t,x) \right|
  \leq C  M_{1,\phi} \, (x\wedge 1)t^{-\frac{3}{2}}
\end{equation*}
in the half-line $\Omega = (0,+\infty)$. Naturally, when $x_0>0$ and $\Omega = (x_0,+\infty)$ we simply have to consider $\phi(x)=x-x_0$, obtaining a translation of the dipole.
\end{rem}

\begin{rem}
  Our results are written in a way which is not invariant by
  translations, since we have chosen a fundamental solution
  $\Gamma(t,x)$ centred at the origin. This is for simplicity in the
  later proofs, but one can easily use the translation invariance of
  solutions of the heat equation to write corresponding ``translated''
  statements if the reader prefers.
\end{rem}

\begin{rem}[Sign-changing solutions]
  Since the problem is linear and positivity preserving, the sign
  assumption on the initial data can be removed, dealing separately
  with the positive and negative parts of the solution, thus obtaining
  the same statement in Theorem \ref{thm:main-uniform} (keeping the
  same $m_\phi$ on the left hand side, but using moments of $|u_0|$ on
  the right-hand side instead of moments of $u_0$). However, if sign
  changes are allowed, it may happen that $m_\phi=0$ for a given
  nontrivial initial data. If this is the case,
  Theorem~\ref{thm:main-uniform} is still true, and shows that
  solutions with $m_\phi = 0$ decay faster to~$0$ than positive
  solutions, in analogy with the heat equation on $\R^d$ and solutions
  with zero integral. That is: assume the conditions of
  Theorem~\ref{thm:main-uniform}, but allow $u_0 \in L^1(\Omega)$ to
  have any sign and assume $m_\phi = \int_\Omega u_0 \phi = 0$. Then
  in $d \geq 3$ we have, for $t \geq 2$,
  \begin{equation*}
    |u(t,x)|
    \lesssim
    M_{1,\phi}[|u_0|] \,\phi(x)\, t^{-\frac{d}{2}- \frac{\lambda}{4}}.
  \end{equation*}
  One can write the corresponding results in dimensions $1$ and $2$ by
  substituting the absolute moments $M_{1,\phi}[|u_0|]$ and
  $m_\phi[|u_0|]$ on the right hand side of the inequalities instead
  of $M_{1,\phi}$ and $m_\phi$.
\end{rem}

\medskip As a consequence of Theorem \ref{thm:main-uniform} we have
the following asymptotic bounds on the Dirichlet heat kernel. Notice
that the heat kernel is explicitly given in dimension 1 by
\begin{equation}
  \label{eq:kernel-dim_1_general}
  p_\Omega(t,x,y) = \Gamma(t,x-y)-\Gamma(t,x+y-2x_0)
\end{equation}
when $\Omega=(x_0,\infty)$ for some $x_0\in\R$. We include however the
result also in this case; even with the explicit kernel it is not
straightforward to obtain the bound we give.

\begin{cor}[Uniform estimates of the heat kernel]
  \label{cor:heat_kernel_bounds}
  Assume $\Omega \subseteq \R^d$ is an exterior domain satisfying
  \eqref{eq:hypOmega1} and Hypothesis \ref{hyp:logsob}.
  Given $y \in \Omega$, take $\lambda = \lambda(y) > 0$ the constant
  from Theorem \ref{thm:main-uniform} corresponding to the domain
  $\Omega_y := \Omega - y$
  \begin{enumerate}
  \item[\rm (i)] In dimension $d \geq 3$, there exists a constant $C > 0$
    depending only on $d$ and $\Omega$ such that
    \begin{equation*}
      |p_\Omega(t,x,y) - \phi(x) \phi(y) p(t,x,y)|
      \leq C \phi(x) \phi(y) t^{-\frac{d}{2} - \frac{\lambda}{4}}
      \qquad \text{for all $t \geq 2$, \ $x,y \in \Omega$.}
    \end{equation*}

  \item[\rm (ii)] In dimension $d=2$, for any $M > 0$ there exists a
    constant $C > 0$ depending only on $d$, $\Omega$ and $M$ such that
    \begin{equation*}
      \left|p_\Omega(t,x,y) - \frac{4 \phi(x) \phi(y)}{(\log t)^2} p(t,x,y)\right|
      \leq \frac{C \phi(x)\phi(y)}{t (\log t)^2}
      \left( \frac{1}{\log t}
      +
      \frac{|x| \wedge |y|}{t^{\lambda/4}}
      \right).
    \end{equation*}
    for all $t \geq 2$ and all $|x|,|y| \in \Omega$ with $|x| \wedge |y| \leq M
    \sqrt{t}$.

  \item[\rm (iii)] In dimension $d=1$  we consider
    $\Omega = (0,+\infty)$. Take $M > 0$. There exists a
    constant $C > 0$ depending only on $d$ and $M$ such that
    \begin{equation*}
      \left|p_\Omega(t,x,y) - \frac{\phi(x) \phi(y)}{t} p(t,x,y)\right|
      \leq \frac{C \phi(x)\phi(y)}{t^2}
      \left( 1
      +
      |x| \wedge |y|
      \right)
    \end{equation*}
    for all $t\geq 2$ and all $|x|, |y| \in \Omega$ with
    $|x| \wedge |y| \leq M \sqrt{t}$.
  \end{enumerate}
\end{cor}

We give a straightforward proof of
Corollary~\ref{cor:heat_kernel_bounds} as a consequence of Theorem
\ref{thm:main-uniform}, which highlights the importance of the
translation invariance of the constants:

\begin{proof}[Proof of Corollary \ref{cor:heat_kernel_bounds}]
  We give the proof first in dimension $d \geq 3$. Consider
  $\Omega_y := \Omega - y$ the translation of the domain $\Omega$ by a
  vector $-y \in \R^d$. The positive function $\phi_\Omega$ satisfying
  $\Delta \phi_\Omega = 0$ and
  $\phi_\Omega \big\vert_{\partial \Omega_y} = 0$ associated to this
  domain $\Omega_y$ is clearly $\phi_y(x) = \phi(x+y)$. Hence,
  applying Theorem \ref{thm:main-uniform} on $\Omega_y$ gives (with
  $\lambda = \lambda(y)$)
  \begin{equation*}
    \label{eq:26}
    |u(t,x) - m_{\phi_y} \phi_y(x) p(t,x,0)| \leq C
    M_{1,\phi_y}\ \phi_y(x) t^{-\frac{d}{2}-\frac{\lambda}{4}}
    \qquad \text{($x \in \Omega_y$, $y \in \R^d$),}
  \end{equation*}
  for all $t \geq 2$ and any standard solution $u$ to the Dirichlet heat
  equation on $\Omega_y$ with integrable initial data $u_0$ such that
  the quantities $M_1$ and $m_\phi$ are finite.  The constant $C$ in
  \eqref{eq:26} is invariant by translations as proved in Theorem
  \ref{thm:main-uniform}, so it is the same for all~$y$. Assume now
  that $0 \in \Omega_y$ (that is, $y \in \Omega$), and take a sequence
  of initial conditions $u_0$ which approximate $\delta_0$ in an
  appropriate way (for example, take $u_{0,n}(x) := n^d \varphi(n x)$
  for a smooth, compactly supported probability density $\varphi$). It
  is well known that the corresponding solutions converge uniformly
  for $x$ in compact sets of $\Omega_y$, for all fixed $t >
  0$. Passing to the limit we obtain from \eqref{eq:26} that
  \begin{equation*}
    |p_{\Omega_y}(t,x,0) - \phi(y) \phi(x+y) p(t,x,0)|
    \leq C  \phi(x+y)\phi(y) t^{-\frac{d}{2}-\sigma}
    \qquad \text{($x \in \Omega_y$, $y \in \Omega$),}
  \end{equation*}
  for all $t \geq 2$. (Observe that in this approximation,
  $M_{1,\phi_y} \to \phi(y)$ and $m_\phi \to \phi(y)$.) Using now that
  $p_{\Omega_y}(t,x,0) = p_{\Omega}(t,x+y,y)$ and
  $p(t,x,0) = p(t,x+y,y)$ we get
  \begin{equation*}
    |p_{\Omega}(t,x+y,y) - \phi(y) \phi(x+y) p(t,x+y,y)|
    \leq C \phi(x+y)\phi(y) t^{-\frac{d}{2}-\sigma}.
  \end{equation*}
  Finally, applying this to $x \equiv x-y$ we obtain the result. The
  proof in dimensions $d=2$ is obtained by the same argument, using the
  corresponding case of Theorem~\ref{thm:main-uniform}.
\end{proof}

Notice that in dimension $2$ the above result does not give
information if $|x| \wedge |y| \sim t^{\lambda/4}$, and in dimension
$1$ it gives no information if $|x| \wedge |y| \sim \sqrt{t}$. In
dimension $2$ this gives a similar restriction as
\citet{Uchiyama-JTP-2018} (which would coincide if $\lambda=2$). The
results mentioned in \eqref{eq:uchi1}, \eqref{eq:uchi2} suggest that
$|x|\wedge|y|=O(\sqrt{t})$ is sharp if we want to obtain the factor
$4/(\log t)^2$. The strategy in our proof suggests a way to obtain a
different behaviour in other scales by keeping the factor $k_t$, but
we have not pursued this.

Regarding the case $d\geq 3$, we think that the dependence of
$\lambda$ on $y$ can be removed in
Corollary~\ref{cor:heat_kernel_bounds}, but we have not been able to
prove it. Whether this can be done or not depends on whether \emph{all
  translations of the domain $\Omega$ satisfy Hypothesis
  \ref{hyp:logsob}, with a constant $\lambda$ which is independent of
  the translation.} Contrary to Hypothesis \ref{hyp:logsob}, we have
not been able to prove this for any domain, but we believe that it
holds at least for the same family of domains for which we show
Hypothesis \ref{hyp:logsob}. One can directly check that if this holds
then one can take the constant $\lambda$ in Corollary
\ref{cor:heat_kernel_bounds} to be independent of $y$.

\subsection{Strategy and $L^1$ estimates}

Our strategy to prove Theorem \ref{thm:main-uniform} mimics the
well-known \emph{entropy method}, which has been very successful in
kinetic theory (see \cite{Cercignani1982,Carrillo2001, Arnold2001} and
the references therein), and which has also been used to study the
asymptotic behaviour of the heat equation in the full space;
see~\cite{Toscani1996} and the review by \cite{vazquez2017asymptotic}.
This method is based on the study of functionals which are decreasing
in time along solutions of the PDE, and usually yield $L^1$ or $L^2$
convergence results. Its application is not straightforward in our
setting, since the equation has no exact scale invariance due to the
presence of the hole in its domain, and is not conservative due to the
Dirichlet boundary condition. However, the main strategy can be
summarised in trying to view the equation as a perturbation of the
equation in the full space. Naturally, this becomes harder in lower
dimensions, where the effect of the boundary condition is more
pronounced. Similar ideas were used by one of the authors to study the
heat equation with an added nonlinear term in \citet*{Canizo2012}.

\medskip The central result of this paper, obtained through these
ideas, is the following weighted $L^1$ convergence result, where the
weight is given by the harmonic profile $\phi$. From it we
can, step by step, extract the necessary information to get results on
$L^1$ convergence, decay of the mass, and global uniform convergence.

\begin{thm}[Weighted $L^1$ estimates]
  \label{thm:main-L1}
  Assume the hypotheses of Theorem \ref{thm:main-uniform}. There
  exists a constant $C > 0$ depending only on the dimension $d$ and
  the domain $\Omega$, and invariant by translations of $\Omega$, such
  that:
  \begin{enumerate}
  \item[\rm (i)] In dimensions $d \geq 3$, for all $t\geq 2$,
    \begin{equation}
      \label{eq:main-L1-d3}
      \int_\Omega \phi(x) \left|
        u(t,x) - m_\phi \phi(x)\Gamma(t,x)
      \right| \d x
      \leq \frac{C M_{1,\phi}}{t^{\lambda/4}}.
    \end{equation}

  \item[\rm (ii)] In dimension $d=2$, let
    $x_0\in \R^2\setminus\Omega$. For all $t \geq 2$,
    \begin{equation*}
      \label{eq:main-L1-d2}
      \int_\Omega \phi(x) \left|u(t,x) -
        k_t m_\phi \phi(x) \Gamma(t,x) \right| \d x
      \leq
      \frac{C m_\phi}{\log{t}}
      +
      \frac{C(M_{1,\phi} + m_\phi |x_0|)}{t^{\lambda/4}}.
    \end{equation*}

  \item[\rm (iii)] In dimension $d=1$ we consider
    $\Omega = (x_0,+\infty)$. Take $M > 0$. For all $t\geq 2$ and all
    $|x_0| \leq M \sqrt{t}$,
    \begin{equation*}
      \int_\Omega \phi(x) \left|
        u(t,x) - k_t m_\phi \phi(x) \Gamma(t,x)
      \right| \d x
      \leq
      \frac{C(M_{1,\phi} + m_{\phi}|x_0|)}{\sqrt{t}}.
    \end{equation*}
    (In this case the constant $C$ depends also on $M$.)
  \end{enumerate}
\end{thm}

We notice that this result can be stated in a unified way for all
dimensions (with rates depending on the dimension): $u$ always
approaches $k_t m_\phi \phi \Gamma$ at an appropriate rate. Since
$\lim_{t \to \infty} k_t = 1$ in dimensions $d \geq 3$, we have chosen
to give a slightly simplified statement which does not involve $k_t$
in dimensions $d \geq 3$.

After proving this result in Section~\ref{sec:l1}, we also show a
similar global $L^1$ convergence result without weights, which in the
case of dimension $d\geq 3$ is almost immediate, but requires some
more thought in the other cases; see Section~\ref{sec:l1-pure}.

\begin{thm}[$L^1$ estimates]
  \label{thm:global-L1}
  Assume the hypotheses of Theorem \ref{thm:main-uniform}. There exists a
  constant $C > 0$ depending only on the dimension $d$ and the domain
  $\Omega$, and invariant by translations of~$\Omega$, such that:
  \begin{enumerate}
  \item[\rm (i)] In dimensions $d \geq 3$,
    we have, for all $t\geq 2$,
    \begin{equation}
      \label{eq:global-L1-d3-phi}
      \int_\Omega \left|u(t,x) - m_\phi \phi(x) \Gamma(t,x) \right| \d x
      \leq \frac{C M_{1,\phi}}{t^{\lambda/4}}.
    \end{equation}
    Alternatively, we may remove $\phi$ and obtain
    (possibly for a different constant $C$):
    \begin{equation}
      \label{eq:global-L1-d3-no_phi}
      \int_\Omega \left|u(t,x) - m_\phi \Gamma(t,x) \right| \d x
      \leq \frac{C M_{1,\phi}}{t^{\lambda/4}}.
    \end{equation}
  \item[\rm (ii)] In dimension $d=2$, let
    $x_0\in \R^2\setminus\Omega$. Then, for all $t \geq 2$,
    \begin{equation*}
      \int_\Omega \left|u(t,x) -
        k_t m_\phi \phi(x) \Gamma(t,x) \right| \d x
      \leq
      \frac{C}{ \log t}
      \left(
        \frac{m_\phi}{\log{t}}
        +
        \frac{M_{1,\phi} + m_\phi |x_0|}{t^{\lambda/4}}
      \right).
    \end{equation*}	
  \item[\rm (iii)] In dimension $d=1$ we consider
    $\Omega = (x_0,+\infty)$. Take $M > 0$. Then, for all $t\geq 2$
    and all $|x_0| \leq M \sqrt{t}$,
    \begin{equation*}
      \int_\Omega
      \left|
        u(t,x) - k_t m_\phi \phi(x) \Gamma(t,x)
      \right| \d x
      \leq
      \frac{C(M_{1,\phi} + m_{\phi}|x_0|)}{t}.
    \end{equation*}
\end{enumerate}
\end{thm}

We will devote Section~\ref{sec:total_mass} to obtain explicit decay rates of the mass of the solutions; see Corollary \ref{cor:mass}. As a
consequence of Theorem \ref{thm:global-L1} we will show that, as $t \to +\infty$,
$$
\begin{array}{ll}
\displaystyle
  \int_\Omega u(t,x) \d x
  = m_\phi + K m_\phi t^{-\frac{d-2}2} + o(t^{-\frac{d-2}2-\frac{2\sigma}{d}}),\qquad
    &d \geq 3,
  \\[10pt]
  \displaystyle
  \int_\Omega u(t,x) \d x
  = \frac{2 m_\phi}{\log t} + O((\log t)^{-2}),\qquad
    &d = 2,
  \\[10pt]
  \displaystyle
  \int_\Omega u(t,x) \d x
  = \frac{m_\phi \sqrt{\pi}}{\sqrt{t}} + O(t^{-1}),\qquad
    &d = 1,
\end{array}
$$
where $K= C^* \int_{\mathbb{R}^N}G(y)|y|^{2-d} \d y$ and
$\displaystyle C^*=\lim_{|x|\to\infty}(1-\phi(x))|x|^{d-2}$. The
existence of this limit is proved in Lemma~\ref{lem:d3-phi-limit}; see
also~\cite[Lemma 4.5]{Quiros-Vazquez-2001}.

\begin{rem}
  In dimensions $d\ge 3$ the amount of mass lost along the time evolution is
  $$
  \int_{\Omega} u_0(x)\d x - \lim_{t\to\infty}\int_{\Omega} u(x,t)\d x=\int_{\Omega}  (1-\phi(x))u_0(x)\d x;
  $$
  thus, it is given by the projection of the initial data onto $\psi:=1-\phi$, which represents in this way the ``dissipation capacity'' of $U$.
  The function $\psi$ is the harmonic function defined in $\Omega$ that takes value 1 on $\partial\Omega=\partial U$ and 0 at infinity. Hence, it is the function measuring the \emph{capacity} of $U$, by means of the formula
  \begin{equation*}
    \label{capacidad}
    \mathop{\rm cap}(U)=\inf_{\{u\ge1\text{ on }U\}}\int_\Omega |\nabla u(x)|^2\d x.
  \end{equation*}
\end{rem}

\subsection{Organisation of the paper}

In Section \ref{sec:entropy} we describe our strategy in more detail,
giving a summary of the outcome for the heat equation on all of
$\R^d$, and then applying similar ideas to the equation on an exterior
domain $\Omega$. Section \ref{sec:estimates} gathers several necessary
estimates on solutions of the heat equation, the function $\phi$, and
related quantities, and in Section~\ref{sec:logsob} we show some
specific logarithmic Sobolev inequalities by applying existing results
in the literature. In Sections~\ref{sec:l1} and \ref{sec:l1-pure} we
prove our weighted and pure $L^1$ estimates, and as a consequence we
obtain the uniform estimates from Theorem~\ref{thm:main-uniform} in
Section~\ref{sec:uniform_convergence}.

\section{Change of variables and entropy}
\label{sec:entropy}

The aim of this section is to describe in detail our use of the entropy
approach to obtain information on large-time behaviour. We start by
recalling this strategy when applied to the heat equation in the whole
space $\R^d$, a computation that was first performed
in~\cite{Toscani1996}, and then explain how to adapt it to an exterior
domain.

\subsection{The heat equation in the full space}
\label{sec:heat-full-space}

If $u = u(t,x)$ is a classical, $L^1$ solution to $\p_t u = \Delta u$
on $\R^d$, then the function
\begin{equation}
  \label{eq:change-fullRd}
  g(\tau, y) := e^{d\tau} u\Big( \frac12 (e^{2\tau}-1), e^\tau y\Big), \qquad \tau \geq 0, \ y \in \R^d
\end{equation}
satisfies the Fokker-Planck equation
\begin{equation}
\label{eq:g.whole}
  \p_t g = \Delta g + \div(xg) \qquad \text{for $t > 0$, $x \in \R^d$.}
\end{equation}
Notice that the mass of $g$ is preserved by the evolution:
$$
\ird g(\tau,y)\d y=\ird u\Big( \frac12 (e^{2\tau}-1), x\Big)\d x=\ird u(0,x)\d x.
$$

\medskip

\noindent\emph{Notation. } We will often regard $g$ as a curve taking values in $L^p(\Omega)$ for some $p\in[1,+\infty]$. In accordance to this, we will use the notation $g(\tau)(y):=g(\tau,y)$.

\medskip

The only equilibrium of~\eqref{eq:g.whole} with integral $1$ is the standard Gaussian $G$, and all solutions with integral $1$ converge exponentially to $G$. This can be proved by the following argument: assume that $g$ is a nonnegative, integrable solution with integral $1$ and finite second moment; that is, with
\begin{equation*}
  \ird u(0,y) (1+|y|^2) \d y < +\infty, \qquad \ird u(0,y) \d y = 1.
\end{equation*}
We define for $\tau \geq 0$ the \emph{relative entropy}
\begin{equation*}
  \label{eq:def-relative-entropy}
  H(g(\tau) \,|\, G)= \ird g(\tau) \log \frac{g(\tau)}{G}.
\end{equation*}
By Jensen's inequality,
$$
H(g(\tau) \,|\, G) := -\ird g(\tau) \log \frac{G}{g(\tau)}\ge -\log\Big(\ird G\Big)=0,
$$
with equality if and only if $g(\tau)=G$.  Notice that if $g\sim G$, then $H\sim0$. Hence, $H$ is expected to give a measure of how far $g$ is from $G$. This is indeed the case, as
shown by the well-known Csiszár-Kullback's inequality
\begin{equation}
  \label{eq:csiszar-kullback}
  \|g - F\|_1^2 \leq 2 H(g \,|\, F),
\end{equation}
true for any $F, g$ nonnegative functions in $L^1(\Omega)$, $\Omega\subseteq\mathbb{R}^d$, with $F$ positive and $\|F\|_1 = \|g\|_1$ \citep{Csiszar1967,Pinsker1964,Kullback1967,Unterreiter2000}. Therefore, an estimate of the decay rate of the relative entropy $H(g(\tau) \,|\, G)$ gives information on the rate of convergence of $g$ towards $G$ in the $L^1$ norm.

A direct calculation that uses the conservation of mass shows that
\begin{equation}
\label{eq:diff.inequality}
  \ddtau H(g (\tau)\,|\, G)  = -\ird g \left| \nabla \log \frac{g(\tau)}{G} \right|^2 \leq - 2 H(g(\tau) \,|\, G),
\end{equation}
where the latter inequality stems from the well-known Gaussian logarithmic Sobolev inequality (\cite{Gross1975}; see also Section~\ref{sec:logsob}). From~\eqref{eq:diff.inequality} it is immediate that
\begin{equation}
  \label{eq:24}
  H(g(\tau) \,|\, G) \leq H(g_0 \,|\, G) e^{-2 \tau},
\end{equation}
where $g_0=g_0(y)=g(0,y)=u(0,y)$ for $y \in \R^d$. Combining~\eqref{eq:24} with Csiszár-Kullback's inequality~\eqref{eq:csiszar-kullback} we get
\begin{equation*}
  \|g(\tau) - G\|_1 \leq \sqrt{2 H(g_0 \,|\, G)} \, e^{-\tau}.
\end{equation*}
This in turn can be translated to information on $u$ unravelling the change of variables we performed at the beginning:
\begin{equation}
  \label{eq:23}
  \ird \Big| u(t,x) - \Gamma\Big(t+\frac12, x\Big) \Big| \dx
  \leq \frac{\sqrt{H(u_0 \,|\, G)}}{\sqrt{t + \frac 12}}
  \qquad \text{for $t \geq 0$,}
\end{equation}
where $\Gamma$ is the fundamental solution of the heat equation on $\R^d$. Notice that as $t \to \infty$, this estimate contains some new
information, since both $\ird u(t,x)\d x$ and $\ird \Gamma (t,x)\d x$ are equal to $1$, while their difference decays as $t \to +\infty$.

\smallskip

Estimate \eqref{eq:23} is the main outcome of this method. However, one may wish to transform this to a perhaps simpler form by using the
well known regularisation property $H(u(1/2, \cdot) \,|\, G)  \lesssim   M_2$: starting at time $t = 1/2$ we obtain that
\begin{equation*}
  \label{eq:25}
  \ird \left| u(t,x) - \Gamma(t, x) \right| \dx
  \leq
  \frac{\sqrt{H(u(1/2, \cdot) \,|\, G)}}{\sqrt{t}}
  \lesssim
  \frac{\sqrt{M_2}}{\sqrt{t}} \qquad \text{for $t > 1/2$.}
\end{equation*}
Since for all $t \geq 0$ the left-hand side is easily bounded by $2$,
we finally obtain the following: for any nonnegative initial data
$u_0$ which is a probability distribution on $\R^d$ with finite second
moment, the standard solution $u$ to the heat equation on $\R^d$ with
this initial data satisfies
\begin{equation*}
  \ird \left| u(t,x) - \Gamma\left(t, x\right) \right| \dx
  \lesssim
  \frac{\sqrt{M_2}}{\sqrt{t+1}}
  \qquad \text{for $t \geq 0$.}
\end{equation*}
If we allow $u_0$ to have integral $m_0$, not necessarily equal to
one, then a simple scaling shows
\begin{equation*}
  \ird \left| u(t,x) - m_0\Gamma\left(t, x\right) \right| \dx
  \lesssim
  \frac{\sqrt{m_0M_2}}{\sqrt{t+1}}
  \qquad \text{for $t \geq 0$.}
\end{equation*}

This result is the analogue of Theorem \ref{thm:main-L1} in the full space, and contains information about the $L^1$ behavior of $u$, since both $u(t,\cdot)$ and $\Gamma(t,\cdot)$ have integral equal to $1$. It is then not too hard to obtain a result in $L^\infty$ norm: since $u - \Gamma$ solves the heat equation for positive times, standard regularisation properties give
\begin{equation*}
    t^{\frac{d}{2}} |u(t,x) - \Gamma(t,x)| \lesssim \|u(t/2,\cdot) - \Gamma(t/2,\cdot)\|_1 \lesssim \frac{\sqrt{m_0M_2}}{\sqrt{\frac{t}{2}+1}}.
\end{equation*}
This estimate is useful for large enough $t$, and provides a rate of convergence of $u$ to the fundamental solution. It is the analogous result to Theorem~\ref{thm:main-uniform} in the full space. With some more work, one can actually change the dependence on $M_2$ by a dependence on~$M_1$ only (see Section \ref{sec:l1}).

\subsection{The heat equation in an exterior domain}
\label{sec:heat-in-exterior-domain}

Our strategy in an exterior domain is to try to mimic the proof in the previous section as closely as possible. We will first rewrite equation~\eqref{eq:heat-ext} in a form which is more convenient for the calculations to be carried out later. The calculations in this section are valid in any dimension $d \geq 1$. The conservation law \eqref{eq:conservation} suggests defining
\begin{equation}
  \label{eq:change1}
  v := \phi u,
\end{equation}
which satisfies the mass-conserving equation
\begin{equation*}
  \label{eq:v}
  \p_t v = \Delta v - 2 \div (v X)\quad\text{in }(0,\infty)\times\Omega,\quad \text{with
 } X(x) := \frac{\nabla \phi(x)}{\phi(x)},
  \quad x \in \Omega.
\end{equation*}
In order to study the asymptotic behaviour of $v$ it is natural to carry out the same (mass preserving) change of variables~\eqref{eq:change-fullRd} which we would consider for the heat equation in all of $\R^d$. Hence we define
\begin{equation}
  \label{eq:change2}
  g(\tau, y) := e^{d\tau} v\big(  (e^{2\tau}-1)/2, e^\tau y\big),
  \qquad \tau \geq 0, \ y \in \Omega_\tau,
\end{equation}
where $\Omega_\tau$ is the moving domain
\begin{equation*}
  \Omega_\tau := e^{-\tau} \Omega.
\end{equation*}
Again we begin by assuming that $m_\phi= \ird u_0 = 1$; a change of scale will give the result for general $m_\phi > 0$. To sum up, this amounts to the following change of variables, which we record here for later reference: we are setting
\begin{equation}
  \label{eq:change-g}
  g(\tau, y) = e^{d\tau} v(t, x), \qquad v(t,x) = \phi(x) u(t,x),
\end{equation}
with
\begin{equation}
  \label{eq:change-tauy}
  t = \frac12 (e^{2\tau}-1), \qquad x = e^\tau y,
\end{equation}
or equivalently
\begin{equation}
  \label{eq:change-tx}
  \tau = \frac12 \log (2 t + 1), \qquad  y = \frac{x}{\sqrt{2 t + 1}}.
\end{equation}
The function $g$ preserves its mass along the evolution, and satisfies
\begin{equation}
	\label{eq:dtg}
	\begin{aligned}
		\p_\tau g = \Delta g + \div(yg) - 2 \div (Z g),  &\qquad \tau > 0,\ y \in \Omega_\tau,
		\\
		g(0,y) = v\left(0, y\right), &\qquad y \in \Omega_\tau,
		\\
		g(\tau,y) = 0, &\qquad\tau > 0,\ y \in \p\Omega_\tau,
        \\
        \text{where }Z = Z(\tau, y) = e^\tau X(e^\tau y)   = e^\tau \frac{\nabla \phi(e^\tau y)}{\phi(e^\tau y)},&\qquad \tau > 0,\ y \in \Omega_\tau.
	\end{aligned}
\end{equation}
The point of this is that we expect \eqref{eq:dtg} to be easier to study than the original equation~\eqref{eq:heat-ext} directly. For
$d \geq 3$ we expect the term $\div(Z g)$ to be small in some sense, so the behaviour will be dominated by the Fokker-Planck equation. For
$d \leq 2$ its effect cannot be asymptotically small, but this form will be easier to work with.

In order to study \eqref{eq:dtg} we try to use relative entropy arguments, similar to the ones that can be used for the Fokker-Planck
equation. First, we define a ``transient equilibrium'' $F_{\tau}$ by solving the equation
\begin{equation*}
  0 = \Delta g + \div(y g) - 2 \div (Z g),
\end{equation*}
which can be seen to give
\begin{equation}
  \label{eq:Ft}
  F_\tau(y) = K_\tau\, \phi(e^\tau y)^2 G(y)
  = (2 \pi)^{-d/2}  K_\tau\, \phi(e^\tau y)^2 e^{-\frac{|y|^2}{2}},
\end{equation}
where $K_\tau$ is a normalisation constant chosen so that
\begin{equation*}
  \int_{\Omega_\tau} F_\tau(y) \d y = 1,\quad\text{that is,}\quad K_\tau\int_{\Omega_\tau}\phi(e^\tau y)^2 G(y)\d y=1.
\end{equation*}
Observe that the link between $K_\tau$ and the constant $k_t$ defined in Theorem \ref{thm:main-L1} is
\begin{equation*}
  k_t = K_{\frac{1}{2} \log(2t)}
  \text{ for $t \geq \frac12$},
  \qquad \text{or equivalently} \qquad
  K_\tau = k_{\frac12 e^{2\tau}}
  \text{ for $\tau \geq 0$}.
\end{equation*}
Equation \eqref{eq:dtg} can be rewritten as
\begin{equation*}
  \p_t g = \div \left( g \nabla \log \frac{g}{F_{\tau}} \right),
\end{equation*}
which makes clearer the parallel with the usual Fokker-Planck
equation. We consider the relative entropy with respect to the
transient equilibrium,
\begin{equation*}
  H(g(\tau) \,|\, F_\tau) := \int_{\Omega_\tau} g(\tau) \log \frac{g(\tau)}{F_\tau}.
\end{equation*}
In order to calculate its time derivative we have to take into account that the domain $\Omega_\tau = e^{-\tau} \Omega$ is moving. By a change of
variables one easily sees that, for any smooth function $f$ with enough decay as $|y| \to +\infty$,
\begin{equation}
  \label{eq:moving-domain}
  \ddtau \int_{\Omega_\tau} f(\tau,x) \d x = \int_{\Omega_\tau} \p_\tau f(\tau,x) \d x
  - \int_{\partial \Omega_\tau} f(\tau,x)\, x \cdot \eta(x) \d S(x),
\end{equation}
where $\eta$ denotes the unit normal to $\partial \Omega_\tau$ pointing towards $\Omega_\tau$. Using this and taking into account the Dirichlet
boundary condition satisfied by $g$,
\begin{equation}
  \label{eq:derivada_entropia}
  \ddtau H(g(\tau) \,|\, F_\tau)
  = -\int_{\Omega_\tau} g(\tau) \left| \nabla \log \frac{g(\tau)}{F_\tau} \right|^2
  - \int_{\Omega_\tau} g(\tau) \, \frac{\p_\tau F_\tau}{F_\tau}.
\end{equation}
Let us also define for notational simplicity
\begin{equation*}
g_0 := \phi u_0, \qquad h_0:= H(g_0|F_0) = \int_\Omega \phi u_0 \log \frac{u_0}{\phi G}.
\end{equation*}
In a similar way as for the usual Fokker-Planck equation, we expect to have a logarithmic Sobolev inequality of the form
\begin{equation}\label{eq:log_sob}
  \lambda H(g\,|\, F_\tau)\leq \int_{\Omega_\tau}g\left|\nabla \log\frac{g}{F_\tau}\right|^2
\end{equation}
holding for some $\lambda > 0$ independent of $\tau$ and all nonnegative $g \in L^1(\Omega_\tau)$ with $\int_{\Omega_\tau} g = 1$. This will allow us to write
\begin{equation}
  \label{eq:H-inequality-main}
  \ddtau H(g(\tau) \,|\, F_\tau)
  \leq
  -\lambda H(g(\tau) \,|\, F_\tau) - \int_{\Omega_\tau} g(\tau) \, \frac{\p_\tau F_\tau}{F_\tau}.
\end{equation}
If we can additionally show that the last term on the right-hand side decays as $\tau \to +\infty$, then Gronwall's lemma gives us a decay rate of the form
\begin{equation}
  \label{eq:Hdecay}
  H(g(\tau)\,|\,F_\tau) \leq \delta(\tau),
\end{equation}
where $\delta = \delta(\tau)$ is an explicit function which tends to $0$ as $\tau \to +\infty$ and depends only on the dimension $d$ and the initial entropy $h_0$. Combining~\eqref{eq:Hdecay} with Csiszár-Kullback's inequality~\eqref{eq:csiszar-kullback} we obtain the decay rate
\begin{equation}
  \label{eq:L1decay}
  \|g(\tau) - F_\tau\|_1 \leq \sqrt{2 \delta(\tau)}.
\end{equation}
From this point on, obtaining information on the original solution $u$ to equation \eqref{eq:heat-ext} is a matter of changing back to the
original variables and rewriting the resulting expression in convenient ways. Assuming that we have managed to prove~\eqref{eq:L1decay}, the change of variables \eqref{eq:change-g}--\eqref{eq:change-tx} readily gives
\begin{equation*}
\int_\Omega  \phi(x) \left|u (t, x)  - (2 \pi)^{-d/2} k_{t+\frac12}   \phi(x) (2t + 1)^{-d/2} e^{-\frac{|x|^2}{2(2t+1)}} \right|
  \d x \leq \alpha(t),
\end{equation*}
where $\alpha(t) := \sqrt{2\delta(\tau)} =\sqrt{2 \delta \big( \frac12 \log(2 t + 1) \big)}$, since $K_\tau=k_{t+\frac12}$. This can be written in terms of $\Gamma$ as
\begin{equation*}
  \label{eq:2}
\int_\Omega  \phi(x) \left|u (t, x) - k_{t+\frac12} \phi(x) \Gamma(t + \frac12, x) \right| \d x \leq \alpha(t).
\end{equation*}
Applying this estimate to the solution with initial data
$\tilde{u}_0(x) := u(\frac12,x)$ we obtain
\begin{equation*}
  \int_\Omega  \phi(x) \left|u (t, x) - k_{t} \phi(x) \Gamma(t, x)
  \right| \d x \leq \alpha \big(t-\frac12\big)
  \qquad
  \text{for $t > \frac12$}.
\end{equation*}
In principle, the function $\alpha (t)$ depends on the initial
relative entropy $h_0$. However, one can further use regularisation
estimates for the heat equation to substitute it for a dependence only
on moments of the initial condition $u_0$, much as we did at the end
of Section \ref{sec:heat-full-space}. The above equation
\eqref{eq:H-inequality-main} is a central step in the paper, and we
use it repeatedly to obtain the rest of our results.

\subsection{A useful expression for the remainder term}

Equation \eqref{eq:H-inequality-main} reads
\begin{equation}
  \label{eq:H-inequality-main2}
  \begin{aligned}
    &\ddtau H(g(\tau) \,|\, F_\tau) \leq -\lambda H(g(\tau) \,|\, F_\tau) - R(\tau),\\
    &R(\tau) := \int_{\Omega_\tau} g(\tau) \, \frac{\p_\tau F_\tau}{F_\tau}
    = \int_{\Omega_\tau} g(\tau) \, \p_\tau \log F_\tau.
  \end{aligned}
\end{equation}
The remainder term $R(\tau)$ can be equivalently written in a more
convenient form as follows. Assume that $m_\phi = 1$. From the
expression of $F_\tau$ in \eqref{eq:Ft},
\begin{equation*}
  \p_\tau \log F_\tau
  = \ddtau \log K_\tau + 2 \p_\tau \log \phi(e^\tau y).
\end{equation*}
Since the integral of $g$ on $\Omega_\tau$ is $m_\phi = 1$
(independently of $\tau$) we have
\begin{equation*}
  R(\tau)
  = \ddtau \log K_\tau
  + 2 \int_{\Omega_\tau} g(\tau,y) \,\p_\tau \log \phi(e^\tau y) \d y.
\end{equation*}
Now, using~\eqref{eq:moving-domain} we get
\begin{align*}
  \ddtau \log K_\tau
  &= - K_\tau \ddtau \int_{\Omega_\tau} \phi(e^\tau y)^2
  G(y) \d y
  \\
  &= - 2K_\tau \int_{\Omega_\tau} \phi(e^\tau y) \nabla \phi(e^\tau y)\cdot (e^\tau y)\, G(y) \d y
\\
  &= - 2 \int_{\Omega_\tau} \frac{ \nabla
  \phi(e^\tau y)\cdot(e^\tau y)}{\phi(e^\tau y)} F_\tau(y) \d y.
\end{align*}
On the other hand,
\begin{equation*}
  \p_\tau \log \phi(e^\tau y)
  = \frac{\nabla \phi(e^\tau y)\cdot(e^\tau y)}{\phi(e^\tau y)}.
\end{equation*}
Hence we have
\begin{equation}
  \label{eq:remainder-rewritten}
  R(\tau)
  = 2 \int_{\Omega_\tau} \frac{ \nabla \phi(e^\tau y)\cdot( e^\tau y)}{\phi(e^\tau y)}
  (g(\tau,y) - F_\tau(y)) \d y.
\end{equation}
This yields a useful estimate. By the Cauchy-Schwartz inequality,
\begin{equation*}
  |R(\tau)|^2
  \leq  4
  \left( \int_{\Omega_\tau}
    \frac{ |\nabla \phi(e^\tau y)|^2 |e^\tau y|^2}{\phi(e^\tau y)^2}
    (g(\tau,y) + F_\tau(y))
    \d y
  \right)
  \left( \int_{\Omega_\tau}  \frac{(g(\tau) - F_\tau)^2}{g(\tau) + F_\tau}\right).
\end{equation*}
The second parenthesis can be estimated as follows, using a standard
strategy in proving the Csiszár-Kullback inequality: since
\begin{equation*}
  z \log z - z + 1 \gtrsim \frac{(z-1)^2}{z+1}
  \qquad \text{for all }z>0,
\end{equation*}
and both $g$ and $F_\tau$ have integral 1 in $\Omega_\tau$, we have
\begin{align*}
  \int_{\Omega_\tau}  \frac{(g(\tau) - F_\tau)^2}{g(\tau) + F_\tau}
  &=
  \int_{\Omega_\tau}  F_\tau \dfrac{\Big(\frac{g(\tau)}{F_\tau} - 1\Big)^2}{\frac{g(\tau)}{F_\tau} + 1}
  \\
  &\lesssim
  \int_{\Omega_\tau}  F_\tau \left(
    \frac{g(\tau)}{F_\tau} \log \frac{g(\tau)}{F_\tau} - \frac{g(\tau)}{F_\tau} + 1
  \right)
  = H(g(\tau) \,|\, F_\tau).
\end{align*}
Hence,
\begin{gather}
  \label{eq:remainder-estimate}
  |R(\tau)|^2\lesssim H(g(\tau) \,|\, F_\tau)\, Q_g(\tau),\quad\text{where}\\
  \notag
  Q_g(\tau):= \int_{\Omega_\tau}\frac{ |\nabla \phi(e^\tau y)|^2 |e^\tau y|^2}{\phi(e^\tau y)^2}(g(\tau,y) + F_\tau(y))    \d y.
\end{gather}
This estimate will be useful later.

\section{Some preliminary estimates}
\label{sec:estimates}

We collect here several estimates on quantities involving solutions of
the heat equation or the harmonic profile $\phi$ that are
required in further proofs. Some of them are well-known and some are
new.

\subsection{Estimates on the harmonic profile $\phi$}
\label{sec:phi}

We gather here some well known results on the solutions $\phi$ to equation \eqref{eq:phi}. The
case of dimension $d=1$ is easily reduced to solving the ordinary
differential equation $\phi'' = 0$ on $(x_0,+\infty)$ with $\phi(x_0)=0$,
and in that case we will choose $\phi(x) = x-x_0$. In dimensions
$d \geq 2$ we obtain the following from the classical theory, which
the reader can find for example in \citet[Chapter II, \S
4.3]{DLbook}:
\begin{lem}
	\label{lem:phi}
	Let $\Omega \subseteq \R^d$ be an exterior domain in dimension
        $d \geq 2$ satisfying \eqref{eq:hypOmega1}. There exists a
        unique classical solution $\phi$ of equation \eqref{eq:phi}
        which satisfies
        \eqref{eq:phi-limd2}--\eqref{eq:phi-limd3}. This classical
        solution $\phi$ is positive everywhere on $\Omega$.
        Additionally, for any
        $x_0 \in \R^d \setminus \overline{\Omega}$ there exist
        $C > 0$, $0 < C_1 < C_2$ such that
	\begin{align}
		\label{eq:phi-bounds-d2}
		\big| \phi(x) - \log |x - x_0| \big| \leq C
		& \qquad \text{for all $x \in \Omega$, in $d = 2$.}
		\\
		\label{eq:phi-bounds-d3}
		C_1 |x-x_0|^{2-d} \leq  1 - \phi(x) \leq C_2 |x-x_0|^{2-d}
		& \qquad \text{for all $x \in \Omega$,  in $d \geq 3$.}
	\end{align}
\end{lem}

\begin{rem}
	These constants are obviously invariant by translations of $\Omega$:
	if instead of~$\Omega$ we consider $\Omega + w$, where $w \in \R^d$
	is any vector, then the same estimates are still true for the
	translated domain if we take $x_0 + w$ instead of $x_0$.
\end{rem}

\begin{proof}[Proof of Lemma \ref{lem:phi}]
	Uniqueness is given by \citet[Chapter~II, \S~4.3, Proposition 9 \& Corollary 3]{DLbook}. Existence is given by
	Theorem 2 in the same section (notice that the existence of a solution satisfying the null condition at infinity easily implies
	the existence of a solution to \eqref{eq:phi} satisfying~\eqref{eq:phi-limd2} or \eqref{eq:phi-limd3}). We remark that
	$\Omega$ having $\mathcal{C}^1$ boundary implies in particular that every point in the boundary is regular.

	The bound \eqref{eq:phi-bounds-d2} is already contained in
	\citet[Lemma 2.1]{GG2007} and \citet[Proposition 2.1 and Remark
	2.1]{CQW2018}, and \eqref{eq:phi-bounds-d3} can be obtained by very
	similar arguments. We recall them here for completeness.

    In dimension $d \geq 3$, and since $\dist(x_0, \Omega) > 0$, there is
	$C_1 > 0$ such that
	\begin{equation}
		\label{eq:4a}
		\phi(x) \leq 1 - C_1 |x-x_0|^{2-d}
		\qquad \text{for all $x \in \partial \Omega$.}
	\end{equation}
	Now fix $\varepsilon > 0$. Since $\lim_{|x| \to +\infty} \phi(x) = 1$
	we can find $R > 0$ such that
	\begin{equation}
		\label{eq:4}
		\phi(x) \leq 1 - C_1 |x-x_0|^{2-d} + \varepsilon
		\qquad \text{for all $x \in \Omega$ with $|x| > R$.}
	\end{equation}
	Since the function $1 - C_1 |x-x_0|^{2-d} + \varepsilon$ is harmonic on
	$B_R \cap \Omega$ and $\phi$ satisfies the inequality \eqref{eq:4}
	on the boundary of $B_R \cap \Omega$ (due to \eqref{eq:4a} and
	\eqref{eq:4}), we deduce that also
	\begin{equation}
		\label{eq:5}
		\phi(x) \leq 1 - C_1 |x-x_0|^{2-d} + \varepsilon
		\qquad \text{for all $x \in \Omega$ with $|x| \leq R$.}
	\end{equation}
	From \eqref{eq:4} and \eqref{eq:5} we see that in fact $\phi(x) \leq
	1 - C_1 |x-x_0|^{2-d} + \varepsilon$ in all of $\Omega$, and then we
	may pass to the limit as $\varepsilon \to 0$ to obtain the lower bound
	in \eqref{eq:phi-bounds-d3}. The upper bound is obtained in an
	analogous way.
	
	The inequalities in \eqref{eq:phi-bounds-d2} can also be obtained by
	a very similar argument: there is $C > 0$ such that
	\begin{equation}
		\label{eq:6}
		\phi(x) < \log|x-x_0| + C
		\qquad \text{for all $x \in \partial \Omega$.}
	\end{equation}
	Fixing now $\varepsilon > 0$, since $\lim_{|x| \to +\infty} \phi(x)/\log|x-x_0| = 1$
	we can find $R > 0$ such that
	\begin{equation}
		\label{eq:7}
		\phi(x) \leq (1+\varepsilon) \log|x-x_0|
		\qquad \text{for all $x \in \Omega$ with $|x| > R$.}
	\end{equation}
	For $\varepsilon > 0$ small enough, \eqref{eq:6} implies that
	\begin{equation*}
		\phi(x) \leq (1+\varepsilon) \log|x-x_0| + C
		\qquad \text{for all $x \in \partial \Omega$.}
	\end{equation*}
	Since the function $(1+\varepsilon) \log|x-x_0| + C$ is harmonic on
	$B_R \cap \Omega$ and $\phi$ satisfies the above inequality
	on the boundary of $B_R \cap \Omega$, we deduce that also
	\begin{equation*}
		\phi(x) \leq (1+\varepsilon) \log|x-x_0| + C
		\qquad \text{for all $x \in \Omega$ with $|x| \leq R$.}
	\end{equation*}
	This and \eqref{eq:7} show that this inequality is in fact satisfied
	in all of $\Omega$, and we may pass to the limit as $\varepsilon \to 0$
	to obtain that $\phi(x) \leq  \log|x-x_0| + C$ on all of
	$\Omega$. The inequality $\phi(x) \geq  \log|x-x_0| - C$ (for a
	possibly different $C > 0$) is obtained in a similar way.
\end{proof}

In this paper we always consider $\phi$ to be the solution whose
existence and uniqueness is given by Lemma \ref{lem:phi}.

\begin{lem}[Linear behavior of $\phi$ at $\p \Omega$]
	\label{lem:phi2}
	Let $\Omega \subseteq \R^d$ be an exterior domain in dimension
	$d \geq 2$ satisfying \eqref{eq:hypOmega1}. For any $R > 0$ there
	exist constants $C_1 > 0$ and $C_2 > 0$ (depending on $R$ and
	$\Omega$) such that $\phi$ satisfies
	\begin{equation*}
		C_1 \dist(x, \partial \Omega)
		\leq \phi(x)
		\leq C_2 \dist(x, \partial \Omega)
		\qquad \text{for all $x \in \Omega \cap B_R$,}
	\end{equation*}
	where $B_R$ is the open ball of radius $R$ in $\R^d$, centered at $0$.
\end{lem}

\begin{proof}
	For $x \in \Omega \cap B_R$ take $y \in \partial\Omega$ such that
	$|x-y| = \dist(x, \partial \Omega)$. To obtain the upper bound we
	just use that $\nabla \phi$ is bounded above by some constant $C_2$
	in $\Omega \cap B_R$: since $\phi(y)=0$,
	\begin{equation*}
		\phi(x) \leq |x - y| \sup_{z \in [x,y]} |\nabla \phi(z)|
		\leq C_2 |x-y| = C_2 \dist(x, \partial \Omega).
	\end{equation*}
	For the lower bound, write the Taylor expansion
	\begin{equation*}
		\phi(x) = \nabla \phi(y) \cdot (x-y) + O(|x-y|^2),
	\end{equation*}
	where the constant implicit in the $O$ notation can be taken to be
	independent of the point $x$. Since $\dist(x, \p \Omega)$ is
	attained at $y$, it must happen that $x-y$ is a multiple of the
	normal vector to $\p \Omega$ at $y$. Then Hopf's Lemma~\citep{Friedman1958} ensures that
	$\nabla\phi(y) \cdot (x-y) \geq C_1 |x-y|$ for some $C_1 > 0$ which
	does not depend on $x \in \p \Omega$.  Then
	\begin{equation*}
		\phi(x) \geq C_1 |x-y| + O(|x-y|^2),
	\end{equation*}
	which shows the lower bound in a neighbourhood $V$ of $\p
	\Omega$. The lower bound on the rest of $\Omega \cap B_R$ is just a
	consequence of the fact that $\phi$ is uniformly bounded below by
	some positive constant on $(\Omega \setminus V) \cap B_R$, and
	$\dist(x, \partial \Omega)$ is bounded above by some constant on
	$\Omega \cap B_R$.
\end{proof}

We also need some estimates on the gradient of $\phi$:

\begin{lem}
	\label{lem:phi-grad}
	With the same hypotheses as Lemma \ref{lem:phi2}, for all
	$x_0 \in \R^d \setminus \overline{\Omega}$ there exists a constant
	$C > 0$ such that we have, in all dimensions
	$d \geq 2$,
	\begin{equation}
		\label{eq:gradphi-d3}
		|\nabla \phi(x) | \leq C |x-x_0|^{1-d}
		\quad \text{for all }x \in \Omega.
	\end{equation}
\end{lem}

\begin{proof}
	The case $d=2$ can be found in \citet[Proposition
	2.2]{CQW2018}. Though we believe that the result for $d\ge3$ is
	well-known, we provide an argument for this case,  since we have not
	found a reference.
	
	Thanks to a translation and a rescaling we may assume without loss
	of generality that $x_0=0$ and also that the hole is inside the ball
	of radius 1 centered at the origin; that is,
	$\R^d\setminus \overline\Omega\subset B_1(0)$.
	
	A uniform upper bound for $|\nabla \phi(x)|$ if
	$x\in \Omega\cap B_1(0)$ can be obtained thanks to Hopf's Lemma and
	the fact that $\Omega\cap B_1(0)$ is compact, so we will provide the
	bound in $\R^d\setminus B_1(0)$. We consider the Kelvin transform of
	$1-\phi$,
	\[
	\zeta(x):=|x|^{2-d}\Big(1-\phi\Big(\frac{x}{|x|^{2}}\Big)\Big),
	\]
	of $1-\phi$, which is defined in $B_1(0) \setminus \{0\}$, is harmonic
	and satisfies, thanks to~\eqref{eq:phi-bounds-d3},
	\[
	|\zeta(x)|\leq C,
	\qquad x \in B_1(0)\setminus \{0\}.
	\]
	The function $\zeta$ can then be extended to a harmonic function
	defined also at the origin, that we still call $\zeta$ for
	convenience. Now define, for a sequence of $\varepsilon\in(0,1)$
	converging to 0, the sequence of functions
	\[
	\zeta_\varepsilon(x):=\zeta(\varepsilon x),\qquad
	x \in B_{1/\varepsilon}(0),
	\]
	which are harmonic and uniformly bounded in
	$B_{1/\varepsilon}(0)$. This means that, locally, and up to a
	subsequence, the sequence $\zeta_\varepsilon$ converges uniformly in
	compact sets of $\R^d$ to a function $\zeta_0$ that is harmonic and
	bounded in $\R^d$; so, by Liouville's Theorem, $\zeta_0$ must be a
	constant. This implies, in particular, that
	$\grad \zeta_\varepsilon(x)\to 0$, uniformly in compact sets of~$\R^d$. We can calculate
	\begin{equation*}
		\nabla \zeta (x) = (2-d) x |x|^{-d}
		\Big(1 - \phi\Big(\frac{x}{|x|^{2}}\Big)\Big)
		- |x|^{-d}\Big( \nabla \phi \Big(\frac{x}{|x|^{2}}\Big)
		- 2\left[
		\nabla\phi\Big(\frac{x}{|x|^{2}}\Big)\cdot x
		\right]
		\frac x{|x|^2}\Big).
	\end{equation*}
	Hence
	\begin{align*}
		\nabla \zeta_\varepsilon (x)
		=
		&\varepsilon^{2-d} (2-d) x |x|^{-d}
		\Big(1 - \phi\Big(\frac{x}{\varepsilon|x|^{2}}\Big)\Big)
		\\
		&- \varepsilon^{1-d} |x|^{-d}
		\Big( \nabla \phi \Big(\frac{x}{\varepsilon |x|^{2}}\Big)
		-2 \left[
		\nabla\phi\Big(\frac{x}{\varepsilon|x|^{2}}\Big)\cdot\frac
		x{\varepsilon|x|^2}
		\right]
		\varepsilon x\Big).
	\end{align*}
	Since $\nabla \zeta_{\varepsilon} \to 0$ uniformly in compact sets,
	there must exist a constant $\delta > 0$ such that
	\begin{equation*}
		|\nabla \zeta_{\varepsilon}(x) | \leq \delta
		\qquad \text{for all $\varepsilon \in (0,1]$ and all $|x|=1$.}
	\end{equation*}
	Calling $y = x / (\varepsilon |x|^2)$ (so $y = x / \varepsilon^2$
	for $|x|=1$), this implies that for all $|y|\geq 1$,
	\[
	|y|^{d-1} \left|
	2 \left[
	\nabla\phi(y)\cdot y
	\right]
	\frac{y}{|y|^2} -\nabla  \phi(y)
	\right|
	\leq (d-2)|y|^{d-2}|1-\phi(y)|+\delta.
	\]
	It is easily checked that $\big| 2 \left[ \nabla\phi(y)\cdot y \right] \frac{y}{|y|^2} -\nabla
	\phi(y) \big|=|\nabla \phi(y)|$.  Then, using once
	more~\eqref{eq:phi-bounds-d3}, there is a constant $C$, depending
	only on $d$, such that
	\[
	|\nabla \phi(y)|\, |y|^{d-1}\leq C
	\]
	for all $y$ such that $|y| \geq 1$. This is the bound for the
	gradient in $\R^d\setminus B_1(0)$ that we were looking for. As
	discussed, together with a bound in $\Omega\cap B_1(0)$ we obtain
	the result.
\end{proof}

We mention that as a consequence of the above proof we get the
following stronger version of the estimate \eqref{eq:phi-bounds-d3}:

\begin{lem}
	\label{lem:d3-phi-limit}
	Let $\Omega \subseteq \R^d$ be an exterior domain in dimension
	$d \geq 2$ satisfying \eqref{eq:hypOmega1}, and
	$x_0 \in \R^d \setminus \overline{\Omega}$. In dimension $d \geq 3$
	there exist $C, C^* > 0$ such that
	\begin{equation*}
		\Big|
		(1 - \phi(x)) |x|^{d-2} - C^*
		\Big|
		\leq \frac{C}{1+|x-x_0|}
		\qquad
		\text{for all $x \in \Omega$.}
	\end{equation*}
\end{lem}

\begin{proof}
	By translating the domain if needed, it is enough to prove it when
	$x_0 = 0 \in \R^d \setminus \overline{\Omega}$. Since the function
	$\zeta$ in the proof of Lemma \ref{lem:phi-grad} can be extended to
	a harmonic function on $B_1(0)$, there must exist $C > 0$ such that
	$|\zeta(x) - \zeta(0)| \leq C |x|$ for all $|x|\leq 1/2$. Writing
	$y = x / |x|^2$ and noticing that $\zeta(0) = C^*$ we obtain
	precisely the statement in the lemma.
\end{proof}


\subsection{Preliminary estimates for kernels and solutions in
  exterior domains}
\label{sec:pre-kernel-estimates}

We gather here some known estimates for the heat kernel in exterior
domains from \citet{Grigoryan2001} and \citet[Theorem
1.1]{Zhang2003}. Though the estimates in these papers are valid for
exterior domains in noncompact manifolds with nonnegative Ricci
curvature, for simplicity we state them only for the case we are
dealing with, in which the manifold is an exterior domain
$\Omega\subset\R^d$. As a consequence of these results, we will obtain
some estimates for solutions of the Cauchy-Dirichlet
problem~\eqref{eq:heat-ext}. The behaviour of the kernels, and hence
of solutions, changes drastically across the critical dimension
$d=2$. Hence, we consider separately the cases $d\ge 3$ and $d=2$.


For later use we give a simple result on the convolution of two
functions. It is mainly used in later estimates to ensure that the
constants we find are invariant by translations of the domain. We give it without proof; point (i) was given in \citet[Lemma
2.2(i)]{Lieb1983}, point (ii) is an easy consequence of point (i):

\begin{lem}
	\label{lem:convolution}
	Let $d \geq 1$ and $f, g \: \R^d \to [0,+\infty)$ be nonnegative,
	radially symmetric functions for which the convolution $f * g$ is
	well defined for all $x \in \R^d$. Then $f * g$ is radially
	symmetric and furthermore:
	\begin{enumerate}
		\item[\rm (i)] If both $f$ and $g$ are radially nonincreasing, then $f*g$ is
		radially nonincreasing.
		\item[\rm (ii)] If $f$ is radially nondecreasing and $g$ is radially
		nonincreasing, then $f * g$ is radially nondecreasing.
	\end{enumerate}
\end{lem}

We use the previous result in several estimates. One example is the
following simple estimate on negative moments of solutions of the heat
equation on all of $\R^d$ (which can be applied to solutions in a
domain with Dirichlet boundary conditions, since the solution is then
bounded above by the solution on all of $\R^d$):
\begin{lem}
  \label{lem:negative-moments}
  In dimension $d \geq 1$, take $0 \leq k < d$ and any $x_0 \in
  \R^d$. Let $u$ be the standard solution to the heat equation in
  $\R^d$ with a nonnegative initial condition $u_0 \in
  L^1(\R^d)$. Then, for some $C = C(d,k) \geq 1$,
  \begin{equation*}
    \ird u(t, x) (1+|x-x_0|)^{-k} \d x
    \leq
    C m_0 (1 + t)^{-k/2},
    \qquad
    t \geq 0.
  \end{equation*}
  More generally, for any $p \geq 0$, if $u_0$ is such that
  $M_p < +\infty$,
  \begin{equation}
    \label{eq:36}
    \ird u(t, x) (1+|x-x_0|)^{-k} |x|^p \d x
    \leq
    C M_p (1 + t)^{-k/2} + C m_0 (1 + t)^{-(k-p)/2},
    \quad
    t \geq 0.
  \end{equation}
  In particular, for some (other) $C > 0$,
  \begin{equation}
    \label{eq:37}
    \ird u(t, x) (1+|x-x_0|)^{-k} |x|^p \d x
    \leq  C M_p (1 + t)^{-(k-p)/2},
    \qquad t \geq 0.
  \end{equation}
  We emphasise that all constants $C$ above are independent of $x_0$.
\end{lem}

\begin{proof}
  We give first the proof in the case $p=0$, which can be obtained
  easily from the expression $u(\cdot,t) = u_0 * \Gamma(t,\cdot)$ and the
  convolution Lemma \ref{lem:convolution}:
  \begin{multline*}
    \ird u(t, x) (1+ |x-x_0|)^{-k} \d x
    =
    \ird u_0(z) \ird \Gamma(t, x-z) (1+ |x-x_0|)^{-k} \d x \d z
    \\
    \leq
    \ird u_0(z) \ird \Gamma(t, x) (1+ |x|)^{-k} \d x \d z
    \leq
    \| u_0 \|_1 \ird \Gamma(t, x) |x|^{-k} \d x
    \lesssim
    \| u_0 \|_1 t^{-k/2}.
  \end{multline*}
  We use the above calculation for $t \geq 1$, while for $0 \leq
  t \leq 1$ we simply use that
  \begin{equation*}
    \int_\Omega u(t, x) (1+ |x-x_0|)^{-k} \d x
    \leq \| u(t,\cdot) \|_1 \leq \|u_0 \|_1.
  \end{equation*}
  Both estimates together give the estimate in the statement in the
  case $p=0$. For $p > 0$, using that $|x|^p \lesssim |x-z|^p + |z|^p$ we have
  \begin{align*}
    \ird u(t, x) &(1+ |x-x_0|)^{-k} |x|^p \d x \lesssim T_1 + T_2,\quad\text{where}
    \\
    T_1&:=\ird u_0(z) |z|^p \ird \Gamma(t, x-z) (1+ |x-x_0|)^{-k}
    \d x\d z,
    \\
    T_2&:=\ird u_0(z)
    \ird \Gamma(t, x-z) (1+ |x-x_0|)^{-k} |x-z|^p \d x \d z.
  \end{align*}
  The first term can be bounded as in the case $p=0$ to get $T_1 \lesssim   t^{-k/2}  \ird u_0(x) |x|^p \d x$.

  As for the second term, we use that
  \begin{equation*}
    \Gamma(t, y) |y|^p \lesssim \Gamma(2t, y) t^{p/2}
  \end{equation*}
  to get
  \begin{equation*}
    T_2 \lesssim
    t^{p/2}
    \ird u_0(z)
    \ird \Gamma(2t, x-z) (1+ |x-x_0|)^{-k} \d x \d z
    \lesssim
    \| u_0 \|_1 t^{-(k-p)/2},
  \end{equation*}
  also with a similar bound as in the case $p=0$. These bounds for
  $T_1$ and $T_2$ are useful for $t \geq 1$; for $0 \leq t\leq 1$ we
  use that
  \begin{equation*}
    \ird u(t, x) (1+ |x-x_0|)^{-k} |x|^p \d x
    \leq
    \ird u(t, x) |x|^p \d x
    \lesssim
    \ird u_0(x) |x|^p \d x.
  \end{equation*}
  This completes the proof of the bound \eqref{eq:36}, and
  \eqref{eq:37} is an immediate consequence.
\end{proof}

\subsubsection{Dimension $d\geq 3$}

We recall that $\rho(x)$ denotes $\dist(x,\partial\Omega)$. We start
with the somewhat simpler nonparabolic case $d\ge 3$.
\begin{thm}[\citet{Zhang2003}]
  \label{thm:Zhang}
  Let $\Omega \subseteq \R^d$, $d \geq 3$, satisfy
  \eqref{eq:hypOmega1}. Let $p_\Omega$ be the Dirichlet heat
  kernel in $\Omega$. There exist constants
  $c_1, c_2 > 0$ depending on $\Omega$ such that
  \begin{multline*}
    \left( \frac{\rho(x)}{\sqrt{t} \wedge 1} \wedge 1 \right)
    \left( \frac{\rho(y)}{\sqrt{t} \wedge 1} \wedge 1 \right)
    \frac{1}{c_1} \Gamma\Big(\frac{t}{c_2}, x-y\Big)
    \leq
    p_\Omega(t,x,y)
    \\
    \leq
    \left( \frac{\rho(x)}{\sqrt{t} \wedge 1} \wedge 1 \right)
    \left( \frac{\rho(y)}{\sqrt{t} \wedge 1} \wedge 1 \right)
    c_1 \Gamma\Big( c_2 t, x-y \Big)
  \end{multline*}
  for all $x,y \in \Omega$ and all $t > 0$.
\end{thm}

As a consequence of Lemma \ref{lem:phi2} and the fact that
$\phi(x) \to 1$ as $|x| \to +\infty$, we may bound $\rho(x) \wedge 1$
above and below by a multiple of $\phi$. Also, for $t \leq 1$,
$\frac{\rho(x)}{\sqrt{t}} \wedge 1 \lesssim
\frac{\phi(x)}{\sqrt{t}}$. Hence from Theorem \ref{thm:Zhang} we
obtain the following:
\begin{cor}
  \label{thm:Zhang_t>1}
  Under the assumptions of Theorem~\ref{thm:Zhang}, there exist
  constants $c_1, c_2 > 0$ such that the following short-time bound
  holds:
  \begin{equation}
    \label{eq:Zhang2-tsmall}
    p_\Omega(t,x,y)
    \leq
    c_1 \varphi(t,x) \varphi(t,y) \Gamma(c_2 t, x-y)
    \qquad \text{for all $0 < t \leq 1$, $x, y \in \Omega$,}
  \end{equation}
  where
  \begin{equation*}
    \varphi(t,x) := \min\left\{ 1, \frac{\phi(x)}{\sqrt{t}} \right\}.
  \end{equation*}
  Also, there exist positive constants
  $c_1, c_2 > 0$ depending on $\Omega$ such that:
  \begin{equation}
    \label{eq:Zhang2}
    \frac{1}{c_1} \phi(x) \phi(y)
    \Gamma\Big(\frac{t}{c_2}, x-y\Big)
    \leq
    p_\Omega(t,x,y)
    \leq
    c_1  \phi(x) \phi(y) \Gamma\Big( c_2 t, x-y \Big)
  \end{equation}
  for all $x,y \in \Omega$ and all $t \geq 1/4$.
\end{cor}

The lower bound $t > 1/4$ is not special in any way, and we write it
for simplicity; any strictly positive lower bound is fine. In this
paper we only use the upper bounds of the above result. With them, we
obtain the following estimates which improve the ``trivial'' ones for
the heat equation on the full space by a factor $\phi$:

\begin{cor}[Kernel $L^p$ and moment estimates]
  \label{cor:kernel_Lp}
  Assume the hypotheses of Theorem~\ref{thm:Zhang}.
  \begin{enumerate}
  \item[\rm (i)] Let $1 \leq p \leq \infty$. There exists
    $C = C(p, \Omega)$ such that
    \begin{equation*}
      \| p_\Omega(t, x, \cdot)/\phi\|_{L^p(\Omega)}
      \leq C \phi(x) t^{-\frac{d}{2p'}}
      \quad
      \text{for all $t \geq 1/4$ and all $x \in \Omega$,}
    \end{equation*}
    where $\frac{1}{p} + \frac{1}{p'}=1$ (with the usual agreement
    that $1/\infty = 0$).

  \item[\rm (ii)] Let $k \geq 0$. There exists $C = C(k, \Omega)$ such that
    \begin{equation*}
      \int_\Omega |y|^k p_\Omega(t,x,y) \d y
      \leq C \phi(x) \left( t^{\frac{k}{2}} + |x|^k \right)
      \quad
      \text{for all $t \geq 1/4$ and all $x \in \Omega$.}
    \end{equation*}

  \end{enumerate}
\end{cor}

\begin{proof}
    (i) The case $p = \infty$ follows directly
    from~\eqref{eq:Zhang2}, since $0\le \phi\le 1$. When
    $1 \leq p < +\infty$ we use the upper bound in \eqref{eq:Zhang2}
    to get, for all $t \geq 1/4$,
    \begin{align*}
      \int_\Omega \frac{1}{(\phi(y))^p} \left( p_\Omega(t,x,y) \right)^p \d y
      &\lesssim
        (\phi(x))^p \int_\Omega
        \big(\Gamma(c_2 t, x-y)\big)^p
        \d y
      \\
      &\lesssim
        t^{-\frac{dp}{2}}(\phi(x))^p \int_\Omega
        \exp \left( - \frac{p |x-y|^2}{4 c_2 t } \right)
        \d y
        \lesssim
        t^{-\frac{dp}{2}+ \frac{d}{2}} (\phi(x))^p.
    \end{align*}
    Raising this to the power $1/p$ we obtain the result for the $L^p$
    norms.
	
    \smallskip

    \noindent (ii) Regarding the moments of order $k$, we use again
    the bound \eqref{eq:Zhang2} to get
    \begin{align*}
      \int_\Omega |y|^k p_\Omega(t,x,y) \d y &\lesssim \phi(x)
      \int_\Omega \phi(y) |y|^k \Gamma\Big( c_2 t, x-y \Big) \d y
      \\
      &\lesssim \phi(x) t^{-\frac{d}{2}} \ird \left( \left|x-y\right|^k
        + |x|^k \right) \exp \Big( \frac{|x-y|^2}{4 c_2 t} \Big) \d y
      \\
      &\lesssim \phi(x) \left( t^{\frac{k}{2}} + |x|^k \right).\qedhere
    \end{align*}
\end{proof}

\begin{cor}[$L^p$-$L^\infty$ regularisation with weight $\phi$]
  \label{cor:Linfty_regularisation_d3}
	In dimension $d \geq 3$, assume $\Omega \subseteq \R^d$ satisfies
	\eqref{eq:hypOmega1}. Let $1 \leq p \leq +\infty$, and take
	$u_0$ such that $\phi u_0 \in L^p(\Omega)$. The unique solution
	$u\in C([0,\infty);L^p(\mathbb{R}^d))$ to
	problem~\eqref{eq:heat-ext} in $\Omega$ with initial condition $u_0$
	satisfies
	\begin{equation*}
		\label{eq:heat-Lp-Lq-ext-phi}
		| u(t,x) | \leq C \phi(x)
		t^{-\frac{d}{2p} } \|\phi u_0\|_p
		\qquad \text{for all $x \in \Omega$ and all $t \geq 1/4$,}
	\end{equation*}
	for some $C > 0$ depending only on $p$ and $\Omega$, with the usual
	convention $1/\infty = 0$.
\end{cor}

\begin{proof}
	For any $1 \leq p \leq +\infty$, we may write the solution as an
	integral against the kernel $p_\Omega$, and then use Corollary
	\ref{cor:kernel_Lp} for any $t \geq 1/4$:
	\begin{equation*}
		|u(t,x)| \leq \int_\Omega p_\Omega(t,x,y) |u_0(y)| \d y
		\leq
		\Big \| \frac{p_\Omega(t,x, \cdot)}{\phi} \Big \|_{p'}
		\|\phi u_0\|_p
		\lesssim
		\phi(x) t^{-\frac{d}{2p}} \|\phi u_0\|_p.
	\end{equation*}
	Notice that the kernel $p_\Omega(t,x,y)$ is symmetric in $x$,
        $y$, so we may take the $L^{p'}$ norm in $x$ or $y$
        indistinctly.
\end{proof}

\begin{cor}[Regularisation of moments with weight $\phi$]
  \label{cor:moment-regularisation-d3}
  Let $\Omega \subseteq \R^d$, $d\ge3$,
  satisfy~\eqref{eq:hypOmega1}. Take $k\ge0$ and any
  $x_0 \in \R^d\setminus\overline\Omega$. There is a constant
  $C = C(d,k)>0$ (independent of $x_0$ and $\Omega$) such that
  standard solutions to problem~\eqref{eq:heat-ext} with initial data
  $u_0 \in L^1(\Omega;(1+|x|^k)\phi(x)\d x)$ satisfy
  \begin{align*}
    m_{k,\phi}(t) \leq m_k(t)
    \leq
    C (m_{k,\phi} +  m_\phi t^{k/2})
    \qquad \text{for all $t \geq 1/4$.}
  \end{align*}
  As a consequence,
  \begin{equation*}
    \label{eq:heat-Mk-ext}
    M_{k,\phi}(t)\le M_k(t)
    \le C t^{k/2} M_{k,\phi}
    \qquad \text{for all }t \geq 1/4.
  \end{equation*}
  Moreover, for all $j>0$, we have estimates of the negative moments of the form
  \begin{align*}
    \int_\Omega u(t, x) |x-x_0|^{-j} |x|^k \d x
    &
      \le C m_{k,\phi} t^{-j/2} +  m_\phi t^{-(j-k)/2},
    \\
    &\le C m_{k,\phi}t^{-(j-k)/2},
      \qquad
      t \geq 1/4.
  \end{align*}
\end{cor}

\begin{proof}
  Expressing $u$ in terms of the heat kernel, and using the upper
  bound in~\eqref{eq:Zhang2},
  \begin{align*}
    m_k(t)=\int_\Omega u(t, x) |x|^k \d x
    &=
      \int_\Omega u_0(z) \int_\Omega p(t, x,z)  |x|^k \d x \d z
    \\
    &\lesssim
      \int_\Omega u_0(z)\phi(z) \int_\Omega \phi(x)\Gamma( c_2 t, z-x ) |x|^k \d x \d z
    \\
    &\lesssim \int_\Omega u_0(z)\phi(z) \ird \Gamma\Big( c_2 t, z-x \Big) |x|^k \d x \d z.
  \end{align*}
  Since $|x|^k \lesssim |x-z|^k + |z|^k$, we get
  \begin{align*}
    \int_\Omega u(t, x) &|x|^k \d x \lesssim t^{k/2}m_\phi + \int_\Omega u_0(z)\phi(z)|z|^k \d z =t^{k/2}m_\phi +m_{k,\phi}.
  \end{align*}
  A similar argument, using Lemma~\ref{lem:convolution}, yields the
  last result about the negative moments: expressing $u$ in terms of
  the heat kernel, and using the upper bound in~\eqref{eq:Zhang2},
  \begin{align*}
    \int_\Omega u(t, x) |x-x_0|^{-j}|x|^k \d x
    &=
    \int_\Omega u_0(z) \int_\Omega p(t, x,z)  |x-x_0|^{-j}|x|^k \d x \d z
    \\
    &\lesssim
      \int_\Omega u_0(z)\phi(z) \int_\Omega
      \phi(x)\Gamma\Big( c_2 t, z-x \Big) |x-x_0|^{-j}|x|^k \d x \d z
    \\
    &\lesssim \int_\Omega u_0(z)\phi(z)
      \ird \Gamma\Big( c_2 t, z-x \Big) |x-x_0|^{-j}|x|^k \d x \d z.
  \end{align*}
  Since $|x|^k \lesssim |x-z|^k + |z|^k$, we get
  \begin{align*}
    \int_\Omega u(t, x) &|x-x_0|^{-j} |x|^k \d x \lesssim T_1 + T_2,\quad\text{where}
    \\
    T_1&:=\int_\Omega u_0(z)\phi(z) |z|^k \ird \Gamma\Big( c_2 t, z-x \Big) |x-x_0|^{-j}
         \d x\d z,
    \\
    T_2&:=\int_\Omega u_0(z)\phi(z)
         \ird \Gamma\Big( c_2 t, z-x \Big) |x-x_0|^{-j} |x-z|^k \d x \d z.
  \end{align*}
  Using now the convolution Lemma~\ref{lem:convolution}, plus the symmetry of $\Gamma$ in the spatial variable,
\begin{align*}
    T_1& \lesssim \int_\Omega u_0(z)|z|^k\phi(z) \ird \Gamma\Big( c_2 t, z-x_0-x \Big) |x|^{-j} \d x \d z\\
        &\lesssim \int_\Omega u_0(z)|z|^k\phi(z) \ird \Gamma( c_2 t, x) |x|^{-j} \d x \d z\lesssim m_{k,\phi} t^{-j/2},\\
    T_2&\lesssim\int_\Omega u_0(z)\phi(z)
    \ird \Gamma\Big( c_2 t, (z-x_0-x)|z-x_0-x|^k \Big) |x|^{-j}  \d x \d z\\
    &\lesssim\int_\Omega u_0(z)\phi(z)
    \ird \Gamma( c_2 t, x)|x|^{k-j}  \d x \d z\lesssim m_{\phi} t^{-(j-k)/2}.\qedhere
\end{align*}
\end{proof}

\subsubsection{Dimension $d=2$}

Bounds of the heat kernel in $d=2$ are more involved, since $\Omega$
is parabolic in this case. It was proved by
\citet[pp. 102--103]{Grigoryan2001} and \citet[Theorem
5.11]{GYRYA2018} that the Dirichlet heat kernel in dimension $d=2$, in
an exterior domain $\Omega$ satisfying
Hypothesis~\eqref{eq:hypOmega1}, satisfies the following for all $t >
0$, $x,y \in \Omega$:
\begin{equation}
  \label{eq:gyrya-SC}
  p_\Omega(t,x,y)\leq
  C\frac{\phi(x)\phi(y)}{\sqrt{V(x,\sqrt{t})V(y,\sqrt{t})}}e^{-\frac{c|x-y|^2}{t}},
\end{equation}
for some constant $C > 0$ depending only on $\Omega$, where
\begin{equation}
  \label{eq:V-def}
  V(x,\sqrt{t}):=\int_{B_{\sqrt{t}}(x)\cap \Omega}\phi^2(z) \d z.
\end{equation}
In order to carry out our estimates we need a more explicit estimate
of the term $V(x,\sqrt{t})$. This estimate is closely related to the
ones given in \citet{GYRYA2018} just before Theorem 5.15, but we have
not been able to find them in the following explicit form:

\begin{lem}
  \label{lem:V-bound}
  Let $\Omega = \R^2 \setminus \overline{U} \subseteq \R^2$ be an
  exterior domain satisfying Hypothesis~\eqref{eq:hypOmega1}, and take
  any $x_0 \in U$ and any $t_0, R > 0$. The quantity $V(x,\sqrt{t})$
  given in \eqref{eq:V-def} satisfies, for all $x \in \Omega$ and all
  $t > 0$,
  \begin{equation*}
    V(x,\sqrt{t}) \gtrsim
    \begin{cases}
      t (\rho(x) + \sqrt{t})^2
      &\qquad \text{if $t \leq t_0$ and $\rho(x) \leq R$,}
      \\
      t \Big(\log(1 + \rho(x) + \sqrt{t})\Big)^2
      &\qquad \text{otherwise,}
    \end{cases}
  \end{equation*}
  where $\rho(x) := \dist(x, \overline{U})$. As a consequence, for
  all $x \in \Omega$ and all $t > 0$,
  \begin{equation*}
    V(x,\sqrt{t}) \gtrsim
    \begin{cases}
      t \max\{\phi(x)^2, t \}
      &\qquad \text{if $t \leq t_0$,}
      \\
      t \max\{ \phi(x)^2, (\log (1+t))^2\}
      &\qquad \text{otherwise,}
    \end{cases}
  \end{equation*}
  or alternatively, if we prefer to write this in a single bound,
  \begin{equation*}
    \label{eq:V-bound-2}
    V(x, \sqrt{t}) \gtrsim t \max \big\{ \phi(x)^2, (\log(1 + \sqrt{t}))^2 \big\}
    \qquad
    \text{for all $x \in \Omega$ and $t > 0$.}
  \end{equation*}
\end{lem}

\begin{proof}
  First, note that in any bounded set where $t \leq C$ and $\rho(x) \leq C$
  we have
  \begin{align*}
    \log(1 + \rho(x) + \sqrt{t})
    &\leq \log(1 + \rho(x) ) + \log(1 + \sqrt{t})\leq \rho(x) + \sqrt{t},
    \\
    \log(1 + \rho(x) + \sqrt{t})
    &\geq \max\{ \log(1 + \rho(x)),\, \log(1 + \sqrt{t}) \}
    \\
    &\gtrsim \max\{ \rho(x),\, \sqrt{t} \} \geq \frac12 (\rho(x) + \sqrt{t}).
  \end{align*}
  From this one easily sees that if the first bound in the lemma holds for
  some given $t_0$, $R$, then it holds for any positive $t_0$ and $R$
  (with a bound which depends on them). We will hence make a choice of
  specific $t_0$ and $R$ to prove the lemma.

  From Lemma \ref{lem:phi} we know that for some $C > 0$,
  \begin{equation*}
    \phi(x) \geq \log |x-x_0| - C.
  \end{equation*}
  Take $R := \max\{ 4 e^C, \diam(\overline{U}) \}.$ This ensures that
  $\log|x-x_0| - C > 0$ for all $x \in \Omega$, and also that
  $\overline{U} \subseteq B_R(x_0)$. We also choose $t_0 = 4R$, and
  divide the proof into several cases.

  \medskip
  \noindent
  \underline{\emph{First case: large $|x-x_0|$.}} Assume $|x-x_0| > R$. Then the
  half of $B_{\sqrt{t}} (x_0)$ defined by the set of the $z$ in
  $B_{\sqrt{t}} (x_0)$ such that $(z-x) \cdot (x-x_0) \geq 0$ is
  contained in $\Omega$.  Call $\mathcal{C}_{\sqrt{t}}$ this
  half-ball. Then the set
  \begin{equation*}
    \mathcal{A}_{\sqrt{t}}
    =
    \{z\in\mathcal{C}_{\sqrt{t}}\, :\, |x-z|>\sqrt{t}/2\},
  \end{equation*}
  satisfies $|\mathcal{A}_{\sqrt{t}}|=\alpha t$ for all $t > 0$ and
  some $\alpha > 0$ independent of $t$. By the Pythagorean inequality,
  all $z\in\mathcal{A}_{\sqrt{t}}$ satisfy
  \[
    |z-x_0|^2\geq |z-x|^2+|x-x_0|^2\geq \frac{t}{4}+|x-x_0|^2.
  \]
  Since $\sqrt{a+b}\geq (\sqrt{a}+\sqrt{b})/2$ for every $a, b > 0$,
  we obtain for all $z\in\mathcal{A}_{\sqrt{t}}$ that
  \[
    \phi(z)>\log(|z-x_0|)-C
    \gtrsim
    \log\left(\frac{|x-x_0|}{2}+\frac{\sqrt{t}}{4}\right)
    \gtrsim
    \log\left(|x-x_0| + \sqrt{t} \right),
  \]
  which gives
  \begin{equation}
    \label{eq:V_1}
    V(x,\sqrt{t}) \geq
    \int_{\mathcal{A}_{\sqrt{t}}}\phi(z)^2 \ {\rm d}z
    \gtrsim
    t \left[
      \log\left(|x-x_0|+\frac{\sqrt{t}}{2}\right)
      \right]^2.
  \end{equation}
  Since in the region where $|x-x_0| > R$ we have $|x-x_0| \gtrsim \rho(x)$, we may substitute
  $|x-x_0|$ by $\rho(x)$ to get the result.

  \medskip
  \noindent
  \underline{\emph{Second case: large $t$}.} Suppose now that $t>4R$. Then the
  point $z_0$ in $B_{\sqrt{t}/2}(x)$ which is furthest away from $x_0$
  satisfies that $B_{\sqrt{t}/4}(z_0) \subseteq B_{\sqrt{t}}(x) \cap
  \Omega$. We can easily repeat a similar argument as in the previous
  case by calling now $\mathcal{A}_{\sqrt{t}} := B_{\sqrt{t}/4}(z_0)$
  and obtain that \eqref{eq:V_1} holds also in this case.

  \medskip
  \noindent
  \underline{\emph{Third case: $t$ small and $|x-x_0| \leq R$}.} Thanks
  to Hopf's Lemma \citep{Friedman1958} we know that $\phi(z) \gtrsim \rho(z)$ for all $z \in \Omega$ with
  $|z-x_0| \leq 2R$. Since $U$ is $\mathcal{C}^2$, its boundary has a
  tubular neighbourhood of a certain width $t_* > 0$ which is
  $C^1$-diffeomorphic to a finite set of copies of
  $S^1 \times (1, 1)$. We define
  \begin{equation*}
    \mathcal{A}_{\sqrt{t}}
    = \left\{z \in B_{\sqrt{t}}(x) \cap \Omega
      \ \big\vert \
    \rho(z) \geq \rho(x) + \frac{\sqrt{t}}{2}\right\}.
  \end{equation*}
  It is then easy to see that, for some $c > 0$,
  $|\mathcal{A}_{\sqrt{t}}| \geq c t$ for all $t \leq t_*$ and $x$
  with $|x-x_0| \leq R$. Repeating our previous argument gives the
  bound in this region.

  \medskip
  \noindent
  \underline{\emph{Fourth case: $t \in [t_0, 4R]$, $|x-x_0| \leq R$}.} In this
  case the bound is just a consequence of the fact that
  $V(x,\sqrt{t})$ is a continuous function of $x$ and $\sqrt{t}$ which
  is strictly positive on this region. Hence it must have a strictly
  positive lower bound, which is enough to show the bound in this
  compact region.
\end{proof}

From \eqref{eq:gyrya-SC} and the previous lemma we immediately get the
following theorem.

\begin{thm}
  \label{thm:kernel_dim 2_best}
  In dimension $d = 2$, assume $\Omega \subseteq \R^d$ satisfies
  \eqref{eq:hypOmega1}. Take any $x_0 \in U$ and $t_0 > 0$. Then there
  exists constants $C, c > 0$ depending on $t_0$ and $\Omega$ such that
  \begin{equation}
    \label{eq:SC-d2-tsmall}
    p_\Omega(t,x,y)\leq C
    \varphi(t,x)\,
    \varphi(t,y)
    \, \Gamma(ct, x-y)
    \qquad \text{ for all }\  0<t<t_0, \ x, y \in \Omega,
\end{equation}
  where $\displaystyle\varphi(t,x):=\min\{1,\phi(x)/\sqrt{t}\}$.

  Regarding times $t \geq t_0$,
  \begin{equation}
    \label{eq:SC-d2-tlarge-2}
    p_\Omega(t,x,y)
    \leq
    C \widetilde\phi(t,x) \, \widetilde\phi(t,y)\,
    \Gamma(ct, x-y)
    \qquad \text{ for all $t \geq t_0$, $x, y \in \Omega$},
  \end{equation}
  where $\displaystyle \widetilde \phi(t,x) := \min\{1,\phi(x)/\log(1+t)\}$.
\end{thm}

\begin{cor}[$L^p$-$L^\infty$ regularisation with weight $\phi$ in
  dimension $2$]
  \label{cor:dim_2_L^p_regularisation}
  In dimension $d = 2$, assume $\Omega \subseteq \R^d$ satisfies
  \eqref{eq:hypOmega1}. Choose $t_0>0$ and $1 \leq p \leq \infty$. For
  $u_0\in L^p(\Omega)$, let $u$ be the standard solution to problem
  \eqref{eq:heat-ext} in $\Omega$ with initial condition $u_0$. Then
  there exists a constant $C=C(t_0,\Omega)$ such that
  \[
    |u(t,x)|
    \leq
    \frac{C\phi(x)}{t^\frac{1}{p}(\log (1+t))^2}
    \ \|\phi u_0\|_{L^p(\Omega)}
    \quad\text{for all }t>t_0.
  \]
  Alternatively we also have, for all $x\in\Omega$,
  \begin{equation*}
    |u(t,x)|
    \leq
    \frac{C}
    {t^{\frac{1}{p}}\log (1+t)}
    \ \|\phi u_0\|_{L^p(\Omega)}
    \quad\text{for all }t>t_0.
  \end{equation*}
  For small times we also have
  \begin{equation*}
    |u(t,x)|
    \leq
    C\frac{\phi(x)}{\sqrt{t}}\| u_0 \|_\infty \quad\text{for all }0<t<t_0.
  \end{equation*}
\end{cor}

\begin{proof}
  Using \eqref{eq:SC-d2-tlarge-2} and choosing the terms
  $\phi/ \log(1+t)$ in the minimum in both $\widetilde{\phi}(t,x)$ and
  $\widetilde{\phi}(t,y)$ we have
  \begin{equation*}
    |u(t,x)|
    \leq
    \int_\Omega |u_0(y)| \, p_\Omega(t,x,y) \d y
    \lesssim
    \frac{1}{t (\log (1+t))^2} \phi(x)
    \int_\Omega |u_0(y)| \, \phi(y) e^{-\frac{c |x-y|^2}{t}} \d y,
  \end{equation*}
  and the latter integral can be estimated by Hölder's
  inequality:
  \begin{equation*}
    \int_\Omega |u_0(y)| \, \phi(y) e^{-\frac{c |x-y|^2}{t}} \d y
    \leq
    \| \phi u_0 \|_{L^p(\Omega)} \
    \| e^{-\frac{c |y|^2}{t}} \|_{L^{p'}(\R^d)}
    \lesssim
    t^{\frac{1}{p'}} \| \phi u_0 \|_{L^p(\Omega)}.
  \end{equation*}
  The second statement is obtained by the same procedure, choosing now $\phi(y) / \log(1+t)$ from the minimum in
  $\widetilde{\phi}(t,y)$, and choosing $1$ from the minimum in
  $\widetilde{\phi}(t,x)$. The small-time estimate is proved by using
  \eqref{eq:SC-d2-tsmall} and following the same calculation, choosing
  $1$ from the minimum in $\varphi(t,y)$, and $\phi(x) / \sqrt{x}$
  from the minimum in $\varphi(t,x)$.
\end{proof}

Finally, we can use these results to give propagation and
regularisation estimates for moments:

\begin{cor}[Moment estimates in dimension $2$]
  \label{cor:2d_moment_regularisation}
  In dimension $d = 2$, assume $\Omega \subseteq \R^d$ satisfies
  \eqref{eq:hypOmega1}, and let $u$ be the standard solution to
  equation \eqref{eq:heat-ext} with a nonnegative initial condition
  $u_0$.
  \begin{enumerate}
  \item (Propagation estimates.) For any $k \geq 0$ there is a
    constant $C > 0$ such that
    \begin{equation*}
      m_k(t) \leq C (1 + t^{\frac{k}{2}}) m_k
      \qquad \text{for all $t \geq 0$}
    \end{equation*}
    We also have, for some constant $C > 0$,
    \[
    m_{2,\phi}(t) \leq C M_{2,\phi}
      \qquad \text{for all $0 \leq t \leq 1$.}
    \]
    As a consequence,
    \[
      M_{2,\phi}(t)\leq C M_{2,\phi}
      \qquad \text{for all $0 \leq t \leq 1$.}
    \]

  \item (Regularisation estimates.) Choose $t_0 > 0$. For any
    $k \geq 0$ there exists a constant $C > 0$ depending only on
    $t_0, k$ and the domain $\Omega$ such that
    \begin{equation*}
      m_k(t)\leq C  (1+t^{\frac{k}{2}}) M_{k,\phi}
      \qquad \text{for all $t \geq t_0$.}
    \end{equation*}
    As a consequence,
    \[
      M_{k}(t) \leq C(1+t^{\frac{k}{2}})M_{k,\phi}
      \qquad \text{for all $t \geq t_0$.}
    \]
  \end{enumerate}
\end{cor}

\begin{proof}
  Let us first prove the regularisation estimates for $m_k$, $M_k$. Using
  \eqref{eq:SC-d2-tlarge-2} one obtains, for all $t \geq t_0$,
  \[
    \begin{aligned}	
      m_k(t)
      &=
      \int_\Omega |x|^ku(t,x)\d x
      \leq
      \int_\Omega u_0(y) \phi(y) \int_\Omega |x|^k \Gamma(ct, x-y) \d
      x \d y
      \\
      &\lesssim
      \int_\Omega u_0(y)\phi(y)\int_\Omega (|x-y|^k + |y|^k)
      \Gamma(ct, x-y) \d x \d y
      \lesssim
      m_\phi t^{\frac{k}{2}} + m_{k,\phi}.
    \end{aligned}
  \]
  This easily implies the stated regularisation estimates for $m_k$
  and $M_k$.

  By a similar procedure, but this time choosing   $1$ in both instances
  of the maximum in \eqref{eq:SC-d2-tlarge-2}, one obtains the
  propagation estimate for $m_k$ (observe that this estimate can be
  deduced from the corresponding estimate for the heat equation in all
  of $\R^d$, which is a supersolution).

  In order to get the propagation estimate on $m_{2,\phi}$ it is
  easier to use the time derivative of the moment: we have
  \begin{align}
    \notag
    \ddt \int_\Omega |x|^2 \phi(x) u(t,x) \d x
    & = \int_\Omega |x|^2 \phi(x) \Delta u(t,x) \d x
    \\
    \label{eq:11}
    &= 4 \int_\Omega \phi(x) u(t,x) \d x
      + 2 \int_\Omega x\cdot \nabla \phi(x) u(t,x) \d x.
  \end{align}
  We use the bound on $\nabla \phi$ from equation
  \eqref{eq:gradphi-d3} in dimension $d=2$, and the fact that
  $|x-x_0| \geq \dist(x_0, \Omega)$ to get
  \begin{equation}
    \label{eq:14}
    \int_\Omega x\cdot\nabla \phi(x) u(t,x) \d x
    \lesssim
      \int_\Omega  \frac{|x|}{|x-x_0|} u(t,x) \d x
      \lesssim
      \int_\Omega  |x| u(t,x) \d x = m_1(t).
  \end{equation}
  Now, for $m_1$ we can use here a short-time bound which we get from
  \eqref{eq:SC-d2-tsmall} in \Cref{thm:kernel_dim 2_best}:
  \begin{align*}
    m_1(t)& = \int_\Omega u(t,x) |x| \d x
    =
    \int_\Omega u_0(y) \int_\Omega p_\Omega(t,x,y) |x| \d x \d y
    \\
    &\lesssim
    \frac{1}{\sqrt{t}} \int_\Omega u_0(y) \phi(y)
    \int_\Omega \Gamma(ct, x-y) |x| \d x \d y
    \\
    &\leq
    \frac{1}{\sqrt{t}} \int_\Omega u_0(y) \phi(y)
    \int_\Omega \Gamma(ct, x-y) (|y| + |x-y|) \d x \d y
    \\
    &\lesssim
    \frac{1}{\sqrt{t}} m_{1,\phi} + m_\phi
    \leq \frac{1}{\sqrt{t}} M_{1,\phi}
  \end{align*}
  for all $0 < t \leq 1$. Using this in \eqref{eq:11} and
  \eqref{eq:14} we get
  \begin{equation*}
    \ddt m_1(t) \lesssim m_\phi + \frac{1}{\sqrt{t}} M_{1,\phi}
    \lesssim \frac{1}{\sqrt{t}} M_{2,\phi}.
  \end{equation*}
  Integrating in time from $0$ to $t$ (and since $1/\sqrt{t}$ is
  integrable) this gives the result.
\end{proof}

Note that we have made no effort to optimise the estimate
$m_{2,\phi}(t)\leq C(1+t^\frac{3}{2}) M_{2,\phi}$, which does not even
give the correct growth rate as $t \to +\infty$. In this paper it will
only be used for bounded times.

\subsubsection{Dimension $d=1$}
\label{sec:d1-prelim}

In dimension $1$ we could follow the same ideas as before, using
\eqref{eq:gyrya-SC} to estimate the kernel. However, in this case the
kernel is explicit and we just carry out the calculations:

\begin{lem}\label{lem:kernel_dim_1_bound}
  In dimension 1, with $\Omega=(x_0,\infty)$, we have
  \[
    p_\Omega(t,x,y)\leq\frac{\phi(x) \phi(y)}{t}\Gamma(t,x-y),
    \qquad x,y > 0, \ t > 0.
  \]
  Alternatively,
  \[
    p_\Omega(t,x,y) \lesssim \frac{\phi(y)}{t},
    \qquad x,y > 0, \ t > 0.
  \]
\end{lem}
\begin{proof}
  It is clearly enough to prove the result on $\Omega =
  (0,+\infty)$. In this case the kernel is
  \begin{equation*}
    p_\Omega(t,x,y) = \Gamma(t,x-y)-\Gamma(t,x+y).
  \end{equation*}
  For the first inequality we write, using
  formula~\eqref{eq:kernel-dim_1_general},
  \[
    \frac{p_\Omega(t,x,y)}{\Gamma(t,x-y)}
    =
    1-e^{\frac{|x-y|^2-|x+y|^2}{4t}} = 1-e^{-\frac{xy}{t}}\leq \frac{xy}{t},
  \]
  the last inequality being true since $1-e^{-z}\leq z$ for all
  $z\in\R$. For the second inequality, use the mean value theorem to
  write, for some $\xi \in (x-y, x+y)$
  \begin{equation*}
    p_\Omega(t,x,y) = 2y \Gamma(t, \xi) \frac{2 \xi}{4t}
    \lesssim \frac{y}{t}. \qedhere
  \end{equation*}
\end{proof}

As a consequence we have the following simple bounds:
\begin{lem}[$L^\infty$ and moment regularisation in dimension $1$]
  \label{lem:regularisation-d1}
  Take with $\Omega = (x_0,+\infty)$ In dimension $1$. For any
  standard solution $u$ to \eqref{eq:heat-ext} it holds that
  \begin{equation*}
    \label{eq:Linfty-d1-x}
    |u(t,x)| \lesssim \frac{\phi(x)}{t^{3/2}} m_\phi,
    \qquad t, x > 0.
  \end{equation*}
  and also that
  \begin{equation*}
    \label{eq:Linfty-d1}
    |u(t,x)| \lesssim \frac{1}{t} m_\phi,
    \qquad t, x > 0.
  \end{equation*}
  Regarding moments, it holds that
  \begin{equation*}
    \label{eq:M2phi-d1}
    M_{2,\phi}(t) \lesssim M_{2,\phi}.
  \end{equation*}
\end{lem}

\begin{proof}
  As before, it is enough to prove the result on
  $\Omega = (0,+\infty)$. For the first estimate, using the first
  bound in \Cref{lem:kernel_dim_1_bound},
  \begin{equation*}
    |u(t,x)|
    \leq
    \int_0^\infty u_0(y) p_\Omega(t,x,y) \d y
    \leq
    \frac{x}{t} \int_0^\infty y u_0(y) \Gamma(t,x-y) \d y
    \lesssim \frac{x}{t^{3/2}} m_\phi.
  \end{equation*}
  The second estimate is obtained in the same way, this time using the
  second bound in \Cref{lem:kernel_dim_1_bound}. The final estimate
  involving $M_{2,\phi}(t)$ is already true for the heat equation on
  $\R$, since $M_{2,\phi} = m_1 + m_2$ in this case; see the proof of
  the propagation estimates of $m_k$ in
  \Cref{cor:2d_moment_regularisation} for an argument which works also
  in dimension $1$.
\end{proof}


\subsection{Estimates on the relative entropy functional}

In order to ensure that the value of the entropy is finite one can
often use the general principle that the entropy of $u$ is bounded by
some $L^p$ norm of $u$ for $p > 1$ and a certain moment $M_k$ for some
$k > 1$. The results in this section make this idea precise.

Our bounds on $\phi$ will also allow us to bound the relative entropy
$H(g \,|\, m_\phi F_\tau)$ at time $\tau = 0$, which is essential in
the main argument of Section \ref{sec:heat-in-exterior-domain}.

We recall a few basic facts. First, for any two nonnegative,
integrable functions $f$, $g$ such that $\int_\Omega f = \int_\Omega g$, we have
$H(f\,|\,g) \geq 0$ and $H(g\,|\,g) = 0$, so the relative entropy functional
attains a minimum at $f=g$. In particular,
\[
  H(\phi u \,|\, m_\phi F_\tau)\geq 0,
\]
whenever $\phi u$ is a nonnegative, integrable function (and with
$m_\phi = \int_\Omega \phi u$). Second, the functional
$u \mapsto H (\phi u \,|\, m_\phi F_\tau )$ is homogeneous of degree
$1$ in the sense that for any $\lambda > 0$
\begin{equation*}
  H (\phi \lambda u \,|\, m_\phi[\lambda u] F_\tau )
  =
  \lambda H (\phi u \,|\, m_\phi[u] F_\tau ),
\end{equation*}
where we write
\begin{equation*}
  m_\phi[u] := \int_\Omega \phi u,
  \qquad
  m_\phi[\lambda u] := \int_\Omega \lambda \phi u
  = \lambda \int_\Omega \phi u.
\end{equation*}
Hence, it is enough to find appropriate bounds for
$H(\phi u \,|\, m_\phi F_\tau)$ assuming that
$m_\phi =  1$. The bound for the general case follows
from this.

\begin{lem}[Bound for the $\phi$-relative entropy functional]
  \label{lem:phiuloguG-bound}
  Let $\Omega$ be an exterior domain satisfying \eqref{eq:hypOmega1},
  and take any $x_0 \in \R^d \setminus \overline{\Omega}$. Let
  $u \: \R^d \to [0,+\infty)$ be a nonnegative measurable function
  with $\int_\Omega \phi u =: m_\phi < +\infty$. Consider
  \begin{equation*}
    H(\phi u \,|\, m_\phi F_0) = H(\phi u \,|\, m_\phi K_0 \phi^2 G)
    = \int_\Omega \phi(x) u(x) \log \frac{u(x)}{m_\phi K_0 \phi(x) G(x)} \d x,
  \end{equation*}
  where $K_0$ is the normalisation constant such that $\int_\Omega K_0
  \phi(x)^2 G(x) \d x = 1$.
  \begin{enumerate}
  \item[\rm (i)] In dimension $d \geq 3$ there exists $C > 0$ such that
    \begin{equation*}
      H(\phi u \,|\, m_\phi K_0 \phi^2 G)
      \leq C (\|u\|_\infty + M_{2,\phi}).
    \end{equation*}

  \item[\rm (ii)] In dimension $d=2$
    there exists $C, c_1, c_2 > 0$ such that
    \begin{equation*}
      H(\phi u \,|\,m_\phi  K_0 \phi^2 G)
      \leq
      C \left(
        \|u\|_\infty
        + M_{2,\phi}
        + \log\log(2 + |x_0|)
      \right).
    \end{equation*}

  \item[\rm (iii)] In dimension $d=1$ we take $\Omega :=
    (0,\infty)$.
    There exists $C > 0$ such that
    \begin{equation*}
      H(\phi u \,|\,m_\phi  K_0 \phi^2 G)
      \leq C (\|u\|_\infty + M_{2,\phi}).
    \end{equation*}
  \end{enumerate}
  All of the previous constants $C$ are invariant by translations of
  $\Omega$.
\end{lem}

\begin{proof}
  We omit the variables for ease of notation, and it is understood
  that all integrals are with respect to $x$. As discussed right
  before this lemma, we will assume $m_\phi=1$ and the general case
  follows by scaling. We write
  \begin{equation*}
    \label{eq:relH-1}
    \begin{aligned}
      \int_{\Omega} \phi v \log \frac{v}{ K_0 \phi G}
      &= \int_{\Omega} \phi v \log \frac{v^2 G}{ K_0 (\phi v) G^2}
      \\[10pt]
      &
      =
      \underbrace{2\int_\Omega \phi v \log v}_{\mathcal{I}}
      \underbrace{ - \int_\Omega \phi v \log \frac{\phi v}{G}}_{+\mathcal{II}}
      \underbrace{-\log K_0 \int_\Omega \phi v}_{+\mathcal{III}}
      \underbrace{-2 \int_\Omega \phi v \log G}_{+\mathcal{IV}}.
    \end{aligned}
  \end{equation*}
  We look at each term separately. First, since $v \log v \leq v^2$,
  \[
    \mathcal{I}\leq 2\int_\Omega \phi v^2\leq 2\|v\|_\infty \int_\Omega\phi v = 2\|v\|_\infty.
  \]
  Next, due to the positivity of the relative entropy,
  \[
    \mathcal{II} \leq 0.
  \]
  The fourth term is easily bounded by
  \[
    \mathcal{IV}=\frac{1}{2}m_{2,\phi} + \log(2\pi).
  \]
  Note that the previous estimates all show that (since $m_\phi = 1$)
  \begin{equation*}
    \mathcal{I} + \mathcal{II} + \mathcal{IV}
    \lesssim 1 + m_{2,\phi} + \|v\|_\infty
    \leq M_{2,\phi} + \|v\|_\infty.
  \end{equation*}
  For the third term, in dimension $d=3$ we may use $\phi \leq 1$, so
  \begin{equation*}
    \mathcal{III} = \log \frac{1}{K_0} = \log \int_\Omega \phi^2 G
    \leq \log \int_\Omega G
    \leq \log \ird G = 0.
  \end{equation*}
  While in dimension $d=1$, since we fix the domain to $(0,+\infty)$,
  \begin{equation*}
    \mathcal{III} = \log \frac{1}{K_0} = \log \int_0^\infty x^2 G
    = \log \frac{1}{2} < 0.
  \end{equation*}
  The only case which contains some subtlety is dimension $d=2$. In
  this case, we show that
  \begin{equation}
    \label{eq:III-bound}
    \mathcal{III} \lesssim \log \log(2 + |x_0|).
  \end{equation}
  In order to show this, let
  $A_{x_0}:=\Omega\cap\{x\in\R^2:|x|<2|x_0|\}$ and
  $B_{x_0}:=\Omega\cap\{x\in\R^2:|x|\geq 2|x_0|\}$. We have, using
  Lemma~\ref{lem:phi},
  \[
    \begin{aligned}	
      \mathcal{III}&=\log\left(\int_\Omega \phi^2 G\right)\leq \log\left(\int_\Omega (\log (C|x-x_0|))^2 G(x)\ \d x\right)\\[10pt]
      &=\log\left(\int_{A_{x_0}} (\log (C|x-x_0|))^2 G(x)\ \d x +\int_{B_{x_0}} (\log (C|x-x_0|))^2 G(x)\ \d x\right)\\[10pt]
      &=\log\left(\int_{A_{x_0}} (\log (3C|x_0|))^2 G(x)\ \d x +\int_{B_{x_0}} (\log (2C|x|))^2 (2\pi)^{-d/2}e^{-|x|^2/4}e^{-|x_0|^2/4}\ \d x\right)\\[10pt]
      &\leq \log\left( (\log (3C|x_0|))^2 \int_{A_{x_0}}G(x)\ \d x +(2\pi)^{-d/2}e^{-|x_0|^2/4}\int_{\R^2} (\log (2C|x|))^2 e^{-|x|^2/4}\ \d x\right)\\[10pt]
      & =\log\left( (\log (3C|x_0|))^2(1-e^{-2|x_0|^2}) + C'e^{-|x_0|^2/4} \right).
    \end{aligned}
  \]
  These computations show that there exist a couple of positive values
  $c_1, c_2$ depending only on $\Omega$ such that
  \[
    \mathcal{III}\leq 2\log\left(\log(c_1|x_0|+c_2)\right),
  \]
  which shows \eqref{eq:III-bound} and finishes the proof.	
\end{proof}

We also have the following ``entropy regularisation estimate''.

\begin{lem}[Heat regularisation for the $\phi$-relative entropy]
  \label{lem:phi-rel-entropy-regularisation}
  Let $\Omega$ be an exterior domain satisfying
  \eqref{eq:hypOmega1}. Let $u$ be the standard solution to the heat
  equation \eqref{eq:heat-ext} in~$\Omega$ with nonnegative initial
  condition $u_0 \in L^1(\Omega)$. Then, calling
  $u_{\frac12}(x) := u(1/2, x)$ we have the following bounds of the
  relative entropy in terms of moments of the initial condition~$u_0$:
  \begin{enumerate}
  \item[\rm (i)]
    In dimension $d \geq 3$,
    $$
    H(\phi u_{\frac12} \,|\,m_\phi  K_0 \phi^2 G)
    \leq C M_{2,\phi}.
    $$

  \item[\rm (ii)] In dimension $d=2$,
    $$
    H(\phi u_{\frac12} \,|\,m_\phi  K_0 \phi^2 G)
    \leq
    C \big( M_{2,\phi}+ \log\log(2+|x_0|) \big).
    $$

  \item[\rm (iii)] In dimension $d=1$, assuming $\Omega = (0,+\infty)$,
    $$
    H(\phi u_{\frac12} \,|\,m_\phi  K_0 \phi^2 G)
    \leq
    C M_{2,\phi}.
    $$
  \end{enumerate}
  The constant $C$ in dimensions $d \geq 2$ is invariant by
  translations of the domain $\Omega$.
\end{lem}

\begin{proof}(i)
  In dimension $d \geq 3$ Lemma \ref{lem:phiuloguG-bound} shows that
  \begin{equation*}
    H(\phi u_{\frac12} \,|\,m_\phi  K_0 \phi^2 G)
    \lesssim
    \|u_{\frac12} \|_\infty + M_{2,\phi}(1/2).
  \end{equation*}
  Also, by Corollaries~\ref{cor:Linfty_regularisation_d3}
  and~\ref{cor:moment-regularisation-d3}, using also that $\phi\le1$,
  \begin{equation}
    \label{eq:27}
    \|u_{\frac12} \|_\infty
    \lesssim m_\phi,
    \qquad
    M_{2,\phi}(1/2) \lesssim  M_{2,\phi},
  \end{equation}
  which shows the result.

  \medskip
  \noindent (ii) In dimension $d = 2$ Lemma \ref{lem:phiuloguG-bound}
  shows that
  \begin{equation*}
    H(\phi u_{\frac12} \,|\,m_\phi  K_0 \phi^2 G)
    \lesssim
    \|u_{\frac12}\|_\infty + \log\log(2+|x_0|)
    +
    M_{2,\phi}(1/2).
  \end{equation*}
  By corollaries~\ref{cor:dim_2_L^p_regularisation}
  and~\ref{cor:2d_moment_regularisation}, the same bounds as in
  equation \eqref{eq:27} work for $\|u_{\frac12}\|_\infty$ and
  $M_{2,\phi}(1/2)$, yielding the result.

  \medskip
  \noindent (iii) Finally, in dimension 1 we have the initial estimate
  from Lemma~\ref{lem:phiuloguG-bound} and the estimate follows
  using formula~\eqref{eq:kernel-dim_1_general}.
\end{proof}

\subsection{Estimates on the normalisation factor $K_\tau$}
\label{sec:Kt}

In the setting of Section \ref{sec:entropy}, we would like to estimate
the quantity $k_t$ for $t \geq 0$ or, equivalently, the quantity
$K_\tau$ defined as the value that satisfies
\begin{equation}
  \label{eq:Kt-def}
  (2 \pi)^{-d/2} K_\tau
  \int_{\Omega_\tau} \phi(e^\tau x)^2 e^{-\frac{|x|^2}{2}} \d x
  = 1.
\end{equation}
We notice that the change between $K_\tau$ from \eqref{eq:Kt-def} and
$k_t$ from \eqref{eq:kt-def-intro} is
\begin{equation}
  \label{eq:Ktau-kt}
  K_\tau = k_{\frac12 e^{2\tau}}
  \quad (\tau \geq 0),
  \qquad
  k_t = K_{\frac12 \log(2t)}
  \quad (t \geq \frac12).
\end{equation}
We only give estimates in $d > 1$, since in $d=1$ we only consider
$\Omega := (0,+\infty)$ and then $k_t = 1/t$ explicitly.
One of our main results for this section is the following:
\begin{prp}
  \label{prp:Kt-final}
  Assume the hypotheses of Theorem \ref{thm:main-L1}. Then there exist
  different constants depending only on the dimension $d$ and the
  domain $\Omega$ such that
  \begin{enumerate}
  \item[\rm (i)] In dimension $d=2$, there exist constants $c_1, c_2>0$ such that
    $$
      K_\tau = \tau^{-2} + O(\tau^{-3})\quad \text{as $\tau \to +\infty$,}
      \qquad \frac{c_1}{1 + \tau^2} \leq K_\tau \leq \frac{c_2}{1 + \tau^2}
      \quad\text{for all $\tau \geq 0$.}
    $$
  \item[\rm (ii)] In dimension $d \geq 3$, there exists a constant
    $c_2 > 0$ such that
    $$
      K_\tau = 1 + O(e^{-(d-2)\tau})
      \quad\text{as $\tau \to +\infty$,}
      \qquad
      1 \leq K_\tau \leq c_2
      \quad\text{for all $\tau \geq 0$.}
    $$
  \end{enumerate}
  The constants $c_2$ in all dimensions and the constant implicit in
  the $O$ notation in dimension $d \geq 3$ are invariant by
  translations of the domain $\Omega$, but not $c_1$.
\end{prp}

This result implies, via the change \eqref{eq:Ktau-kt}, the following asymptotics for $k_t$:
\begin{align}
  \notag
  k_t = \frac{4}{(\log t)^2} + O((\log t)^{-3})
  &\qquad \text{as $t \to +\infty$ in $d = 2$,}
  \\
  \label{eq:kt-d3-changed}
  k_t = 1 + O(t^{-\frac{d-2}{2}} )
  &\qquad \text{as $t \to +\infty$ in $d \geq 3$.}
\end{align}
The implicit constant in the $d \geq 3$ asymptotics in
\eqref{eq:kt-d3-changed} is also invariant by translations. We can
also easily rewrite the bounds on $K_\tau$ as similar bounds on $k_t$.

\medskip The rest of this section is devoted to the proof of
Proposition \ref{prp:Kt-final} and an important estimate on
$|K'_\tau|/{K_\tau}$ which we give in Lemma
\ref{lem:Kt'/Kt}. In order to study $K_\tau$ it is easier to
study the integral which appears in its definition, that is,
\begin{equation}
  \label{eq:It-def}
  I_\tau := (2 \pi)^{-d/2} \int_{\Omega_\tau} \phi(e^\tau x)^2
  e^{-\frac{|x|^2}{2}} \d x
  = \int_{\Omega_\tau} \phi(e^\tau x)^2 G(x) \d x,
  \qquad \tau \geq 0.
\end{equation}
Since $K_\tau = 1/I_\tau$, the following lemma easily
implies Proposition \ref{prp:Kt-final} (as knowing the asymptotics /
bounds for $I_\tau$ yields corresponding asymptotics / bounds for
$K_\tau$):

\begin{lem}[Estimates for $I_\tau$]
  \label{lem:It}
  Assume the hypotheses of Theorem \ref{thm:main-L1} and define
  $I_\tau$ by \eqref{eq:It-def}.
  \begin{enumerate}
  \item[\rm (i)]
    In dimension $d=2$ there exist $0 < c_1 < c_2$ such that
    $$
      I_\tau = \tau^2 + O(\tau)\quad \text{as $\tau \to +\infty$,}\qquad
      c_1(1 + \tau^2) \leq I_\tau \leq c_2(1 + \tau^2)\quad\text{for all $\tau \geq 0$.}
    $$

  \item[\rm (ii)]
    In dimension $d \geq 3$ there exists $0 < c_1$ such that
    \begin{equation}
      \label{eq:It-asymptotic-d3}
      I_\tau = 1 + O(e^{-(d-2) \tau})\quad \text{as $\tau \to +\infty$,}\qquad c_1 \leq I_\tau \leq 1\quad\text{for all $\tau \geq 0$.}
    \end{equation}
  \end{enumerate}

  The constants $c_1$ and $c_2$, and the constants implicit in the $O$
  notation, depend only on the dimension $d$ and the domain
  $\Omega$. Additionally, the constant $c_1$ is invariant by
  translations of the domain $\Omega$, but not the constant $c_2$. The
  constant implicit in the $O$ notation in \eqref{eq:It-asymptotic-d3}
  is also invariant by translations of $\Omega$.
\end{lem}

\begin{proof}
Let us first prove the
  estimates in dimension $d=2$. Choosing
  $x_0 \in \R^2 \setminus \overline{\Omega}$, Lemma \ref{lem:phi}
  gives for some $C = C(d, \Omega) > 0$
  \begin{equation}
    \label{eq:15}
    \log|x-x_0| - C \leq \phi(x) \leq \log|x-x_0| + C,
    \qquad x \in \Omega.
  \end{equation}
  This implies the upper bound
  \begin{equation}
    \label{eq:16}
    \phi(x)^2 \leq (\log|x-x_0| + C)^2,
    \qquad x \in \Omega.
  \end{equation}
  For a lower bound of $\phi(x)^2$ we must take only the set on which
  the lower estimate in~\eqref{eq:15} is nonnegative. Hence we choose
  $R := e^C$ so that $\log R - C = 0$, and $B_R(x_0)$ is the set where
  $\log|x-x_0| - C < 0$. By~\eqref{eq:15} we see that
  $\log|x-x_0| - C \leq 0$ at all $x \in \partial \Omega$, so
  $\partial \Omega \subseteq \overline{B_R(x_0)}$, which implies
  $\R^d \setminus \Omega \subseteq \overline{B_R(x_0)}$. Hence
  \begin{equation}
    \label{eq:17}
    (\phi(x))^2 \geq (\log|x-x_0| - C)^2,
    \qquad x \in \R^d \setminus \overline{B_R(x_0)} =: \Omega^R.
  \end{equation}
  With this definition and the observation above it is clear that
  $\Omega \supseteq \Omega^R$.

  We can now use both bounds \eqref{eq:16}--\eqref{eq:17} to get
  bounds of $I_\tau$. For a lower bound we use \eqref{eq:17}, and call
  $\Omega^R_\tau := e^{-\tau} \Omega^R$ to obtain
  \begin{equation*}
    I_\tau
    \geq
    \int_{\Omega^R_\tau} (\log|e^\tau x - x_0| - C)^2 G(x) \d x
    =
    \int_{\Omega^R_\tau} (\tau + \log|x - e^{-\tau} x_0| - C)^2 G(x) \d x.
  \end{equation*}
  Let us define
  \[
  f_\tau(z):=\begin{cases}
    (\tau + \log |z| - C)^2& \text{if } z\not\in B_{e^{-\tau}R}(0),\\
    0 & \text{if } z\in e^{-\tau}\overline{B_{e^{-\tau}R}(0)}.
  \end{cases}
  \]
  The previous inequality becomes then
  \[
  I_\tau \geq \int_{\R^d} f_\tau(e^{-\tau}x_0-x) G(x) \d x.
  \]
  Since $f$ is radially increasing, the last bound given for $I_\tau$
  is the convolution of this function with $G$, evaluated at the point
  $e^{-\tau} x_0$. Hence by Lemma \ref{lem:convolution} we can say
  that
  \begin{equation}
    \label{eq:18}
    I_\tau \geq
    \int_{\R^d \setminus B_{e^{-\tau}R} (0)}
    (\tau + \log|x| - C)^2 G(x) \d x
    \geq
    \int_{\R^d \setminus B_{R} (0)} (\tau + \log|x| - C)^2 G(x) \d x.
  \end{equation}
  The lower bound in dimension $d=2$ given in the statement is a
  direct consequence of this one.

  In order to obtain an upper bound for $I_\tau$, we observe that
  $\log|x-x_0| + C > 0$ on $\Omega$. Choose $r := e^{-C}$ so that
  $\log r + C = 0$, and
  $\Omega \subseteq \R^d \setminus \overline{B_r(x_0)}$. Call
  $\Omega^r := \R^d \setminus \overline{B_r(x_0)}$ and
  $\Omega^r_\tau := e^{-\tau} \Omega^r$. Using \eqref{eq:16} we have
  \begin{equation}\label{eq:I_tau_upper_bound_e_tau}
    I_\tau \leq
    \int_{\Omega^r_\tau} \big( \log|e^\tau x - x_0| + C \big)^2 G(x) \d x
    =
    \int_{\Omega^r_\tau} \big(\tau + \log|x - e^{-\tau} x_0| + C \big)^2 G(x) \d x.
  \end{equation}
  Similarly as before, the function $z \mapsto (\tau + \log|z| + C)^2$ is
  nondecreasing in $z$, and the integral above is the convolution of
  this function with $G$ evaluated at $e^{-\tau} x_0$. Due to Lemma
  \ref{lem:convolution},
  \begin{equation}
    \label{eq:19}
    I_\tau \leq
    \int_{\R^d\setminus B_r(0)} \big(\tau + \log|x - x_0| + C \big)^2 G(x) \d x.
  \end{equation}
  The upper bound in the lemma is readily obtained from this
  one. Observe that the dependence on $x_0$ seems to be unavoidable
  here.  The asymptotic behavior of $I_\tau$ as $\tau \to +\infty$ can
  also be obtained from \eqref{eq:18} and \eqref{eq:19}.

  For dimensions $d \geq 3$ we can follow a similar reasoning. First,
  since $\phi(x) \leq 1$ on $\Omega$ and $\lim_{|x| \to +\infty}
  \phi(x) = 1$, one directly sees from the expression of $I_\tau$ and the dominated convergence
  theorem that $\lim_{t \to +\infty} I_\tau = 1$. The upper bound in $d
  \geq 3$ is trivial, since $\phi(x) \leq 1$ in $\Omega$ implies $I_\tau
  \leq 1$.

  In order to obtain a lower bound for $I_\tau$ we proceed as in the case
  of dimension $d=2$.  Choosing
  $x_0 \in \R^2 \setminus \overline{\Omega}$, Lemma \ref{lem:phi}
  proves that
  \begin{equation*}
    1 - C_2 |x-x_0|^{2-d}  \leq \phi(x),
    \qquad x \in \Omega.
  \end{equation*}
  As before, we define a domain which will be used for the lower bound:
  \begin{equation*}
    R := C_2^{-\frac{1}{2-d}}, \quad \text{so that } 1 - C_2 R^{2-d} =  0.
  \end{equation*}
  One can check that then, similarly as in the $d=2$ case, $\Omega^R := \R^d \setminus \overline{B_R(x_0)}\subseteq
    \Omega$. We see
  \begin{equation*}
    1 - C_2 |x-x_0|^{2-d}  \leq \phi(x),
    \qquad x \in \Omega^R,
  \end{equation*}
  and since the function on the left-hand side is nonnegative on $\Omega^R$,
  \begin{equation*}
    (1 - C_2 |x-x_0|^{2-d})^2  \leq \phi(x)^2,
    \qquad x \in \Omega^R.
  \end{equation*}
  Calling $\Omega^R_\tau := e^{-\tau} \Omega^R$, and using again Lemma
  \ref{lem:convolution}, this implies that
  \begin{equation}
    \label{eq:28}
    I_\tau \geq \int_{\Omega^R_\tau} (1 - C_2 |e^\tau x - x_0|^{2-d})^2 G(x) \d x
    \geq \int_{\R^d \setminus B_{e^{-\tau} R}(0)} (1 - C_2 |e^\tau x|^{2-d})^2 G(x) \d x.
  \end{equation}
  Since the last expression is increasing in $\tau$, we may set $\tau=0$ and
  obtain
  \begin{equation*}
    I_\tau \geq \int_{\R^d \setminus B_R(0)} (1 - C_2 |x|^{2-d})^2 G(x) \d x
    := c_1,
  \end{equation*}
  which shows the lower bound in the statement (invariant by
  translations of $\Omega$, since $C_2$ and $R$ are). In order to get
  the asymptotics of $I_\tau$ as $\tau \to +\infty$, we may continue from~\eqref{eq:28} and obtain
  \begin{equation}
    \label{eq:29}
    \begin{aligned}
    I_\tau
    &\geq
    \int_{\Omega^R_\tau} G(x) \d x
    - 2 C_2 \int_{\Omega^R_\tau} |e^\tau x|^{2-d} G(x) \d x
    \\
    &\geq
    1 - \int_{B_{Re^{-\tau}}(x_0)} G(x) \d x
    - 2 C_2 e^{-(d-2)\tau} \int_{\R^d} |x|^{2-d} G(x) \d x,
  \end{aligned}
  \end{equation}
  where the inequality in which we removed $x_0$ in the first line is due
  to Lemma \ref{lem:convolution}. The middle term in the inequality
  above can be easily bounded by
  \begin{equation*}
    \int_{B_{Re^{-\tau}}(x_0)} G(x) \d x
    \leq
    (2 \pi)^{-d/2} |B_{Re^{-\tau}}(x_0)|
    \leq
    C e^{-d \tau},
  \end{equation*}
  which implies from \eqref{eq:29} that $I_\tau = 1 - O(e^{-(d-2)\tau})$.
\end{proof}

The next result is used in order to obtain a sharper estimate for the kernel in dimension $d=2$. It measures the distance between $\tau^2$ and $I_\tau$ provided that $x_0$ is small compared with $\tau$. When translated back to the original time variable $t$, it will provide
\[
\left|\frac{(\log t)^2}{4} - I_t\right|=O(\log t)\quad\text{whenever}\quad|x_0|=O(\sqrt{t}).
\]

\begin{lem}
  \label{lem:I_tau_distance_d_2}
  Assume the hypotheses of Theorem \ref{thm:main-L1} and define
  $I_\tau(x_0)$ by \eqref{eq:It-def} this time highlighting its dependence on the variable $x_0$. Suppose also that $|x_0|=O(e^{\tau})$. Then
  \[
    \left|\tau^2 - I_\tau(x_0) \right| = O(\tau).
  \]
\end{lem}

\begin{proof}
  We begin considering equations~\eqref{eq:18}
  and~\eqref{eq:I_tau_upper_bound_e_tau} from
  Lemma~\ref{lem:It}. Equation~\eqref{eq:18} provides
	\[
	\begin{aligned}
	I_\tau &\geq
	\int_{\R^d \setminus B_{e^{-\tau}R} (0)}
	(\tau + \log|x| - C)^2 G(x) \d x \\[10pt]
	& =\tau^2 - \tau^2\int_{B_{e^{-\tau}R} (0)} G(x) \d x + \int_{\R^d \setminus B_{e^{-\tau}R} (0)}
	\big(2\tau(\log|x| - C) + (\log|x| - C)^2\big) G(x) \d x \\[10pt]
	&\geq \tau^2 - R^2e^{-2\tau}\tau^2 - O(\tau)\geq \tau^2-O(\tau).
	\end{aligned}
	\]
	On the other hand, equation~\eqref{eq:I_tau_upper_bound_e_tau} yields
	\[
	\begin{aligned}	
	I_\tau &\leq
	\int_{\Omega^r_\tau} \big(\tau + \log|x - e^{-\tau} x_0| + C \big)^2 G(x) \d x\leq  \int_{\R^2} \big(\tau + \log|x - e^{-\tau} x_0| + C \big)^2 G(x) \d x\\[10pt]
	&= \tau^2 + \int_{\R^2}
	\big(2\tau(\log|x - e^{-\tau} x_0| + C) + (\log|x - e^{-\tau} x_0| + C)^2\big) G(x) \d x.
	\end{aligned}
	\]
	Now we use the estimate $|x|\geq |x - e^{-\tau} x_0| - |e^{-\tau} x_0|$ in order to bound
	\[
	G(x)=C_de^{-\frac{|x-e^{-\tau} x_0+e^{-\tau} x_0|^2}{2}}\leq C_de^{-\frac{|x-e^{-\tau} x_0|^2}{2}}e^{\frac{|e^{-\tau} x_0|^2}{2}} \leq CG(x-e^{-\tau} x_0),
	\]
	since $|x_0|=O(e^\tau)$. Therefore
	\[
	I_\tau \leq \tau^2 +\int_{\R^2} \big(2\tau(\log|x - e^{-\tau} x_0| + C) + (\log|x - e^{-\tau} x_0| + C)^2\big) G(x-e^{-\tau} x_0) \d x = \tau^2 + O(\tau).
	\]
	In total, we have obtained $\tau^2-O(\tau)\leq I_\tau(x_0)\leq \tau^2+O(\tau)$, yielding the desired result.
\end{proof}
Finally, one can ask when the factor $k_t(y)$ can be exchanged by the
quantity $4/(\log t)^2$, which is its large-$t$ asymptotic behaviour
for fixed $y$. Our answer is positive as long as
$|x|\wedge|y|=O(\sqrt{t})$. Suppose for example that $|y|=O(\sqrt{t})$
and define $I_t(y)$ as in~\eqref{eq:It-def} (with the change of
variables $\tau\sim\log(t)/2$ and highlighting its dependence on $y$)
and the function $f(z)=1/z$. Then, by the Mean Value Theorem, for
$t >> 1$ we get
\[
\left|\frac{4}{(\log t)^2} - k_t(y)\right| = \left|f\left(\frac{(\log t)^2}{4}\right) - f\left(I_t(y)\right)\right|\leq \left|\frac{(\log t)^2}{4}- I_t(y)\right|\frac{1}{|\xi|^2},
\]
where $\xi\in[(\log t)^2/4, I_t(y)]$.  After the corresponding change
of variables $\tau\sim\log(t)/2$, Lemma~\ref{lem:It} provides
$\xi\geq c(\log t)^2$, while Lemma~\ref{lem:I_tau_distance_d_2} yields
$\left|\frac{(\log t)^2}{4}- I_t(y)\right|\leq O(\log t)$. In total,
we get
\[
  \left|\frac{4}{(\log t)^2} - k_t(y)\right|
  = O((\log t)^{-3}).
\]

\medskip
The following lemmas involve estimates of the derivative in $\tau$ of
$I_\tau$ and $K_\tau$, and are needed to obtain our final Lemma
\ref{lem:Kt'/Kt}, which will be essential in the next section.

\begin{lem}[Estimates of $I_\tau'$]
  \label{lem:It'}
  Assume the conditions of Lemma \ref{lem:It}. Then $I_\tau$ is a differentiable function, and there exists a constant $C > 0$, invariant by translations of
  the domain $\Omega$, such that for all $t \geq 0$
  $$
   |I'_\tau| \leq C (1 + \tau) \quad\text{in $d = 2$,}\qquad |I'_\tau| \leq C e^{-(d-2) \tau}\quad\text{in $d \geq 3$.}
  $$
\end{lem}

\begin{proof}
  From its expression in \eqref{eq:It-def} and common
  differentiability theorems for parameter-depending integrals we see
  that $I_\tau$ is differentiable with
  \begin{equation*}
     I'_\tau = 2 e^\tau \int_{\Omega_\tau} \phi(e^\tau x) x \cdot \nabla
    \phi(e^\tau x) G(x) \d x.
  \end{equation*}
  (We notice that one can deal with the time-dependent domain by using
  \eqref{eq:moving-domain}, so the boundary term vanishes due to the
  Dirichlet boundary condition on $\phi$.)

  We first show the bound in dimension $d=2$. Using the bounds on
  $\phi$ and $\nabla \phi$ from equations \eqref{eq:phi-bounds-d2} and
  \eqref{eq:gradphi-d3},
  \begin{align*}
    |I'_\tau| &\leq \int_{\Omega_\tau} 2 \phi(e^\tau x) |\nabla \phi(e^\tau x)|\, |e^\tau x|  G(x) \d x
    \lesssim\int_{\Omega_\tau} (C + \log|e^\tau x - x_0|) \frac{|e^\tau x|}{|e^\tau  x - x_0|}    G(x) \d x
    \\
    &=\int_{\Omega_\tau} (C + \tau + \log|x - e^{-\tau} x_0|) \frac{|x|}{|x - e^{-\tau} x_0|}G(x) \d x
    \leq \int_{\Omega_\tau} f_t(|x - e^{-\tau} x_0|) |x| G(x) \d x,
  \end{align*}
  where we define
  $f_\tau (r) := \frac{\tau}{r} + \sup_{s > r} \frac{|C + \log s|}{s}$, a
  decreasing function of $r$. Using that
  $|x| G(x) \lesssim \exp(-|x|^2 / 4)$ we then get by Lemma
  \ref{lem:convolution} that
  \begin{align*}
    |I'_\tau|&\lesssim\int_{\Omega_t} f_t(|x - e^{-t} x_0|) e^{-\frac{|x|^2}{4}} \d x
    \leq \int_{\Omega_t} f_t(|x|) e^{-\frac{|x|^2}{4}} \d x
    \\
    &\leq t \int_{\Omega_t} \frac{1}{|x|} e^{-\frac{|x|^2}{4}} \d x + \int_{\Omega_t} h(|x|) e^{-\frac{|x|^2}{4}} \d x,
  \end{align*}
  with $h(r) := \sup_{s > r} \frac{|C + \log s|}{s}$. This shows the
  estimate in dimension $2$.

  In dimension $d \geq 3$, using $\phi \leq 1$ and the estimate
  \eqref{eq:gradphi-d3} on $|\nabla \phi|$ we have
  \begin{align*}
    |I'_\tau| &\leq     \int_{\Omega_{\tau}} 2 \phi(e^\tau x) |\nabla \phi(e^\tau x)|\, |e^\tau x|  G(x) \d x
    \lesssim \int_{\Omega_\tau} |e^{\tau} x - x_0|^{1-d} |e^{\tau} x| G(x) \d x
    \\
    &= e^{(2-d) \tau} \int_{\Omega_\tau} |x - e^{-\tau}x_0|^{1-d} |x| G(x)\d x
    \leq e^{(2-d) \tau} \int_{\Omega_\tau} |x - e^{-\tau}x_0|^{1-d} e^{-\frac{|x|^2}{4}} \d x.
  \end{align*}
  Now using Lemma \ref{lem:convolution} we get
  \begin{equation*}
    |I'_\tau| \lesssim e^{(2-d) \tau} \int_{\Omega_\tau} |x|^{1-d} e^{-\frac{|x|^2}{4}} \d x \lesssim e^{(2-d) \tau}.\qedhere
  \end{equation*}
\end{proof}

\begin{lem}[Estimates of $K_\tau'$]
  \label{lem:Kt'}
  Assume the hypotheses of Theorem \ref{thm:main-L1}. Then $K_\tau$ is a
  differentiable function, and there exists a constant $C > 0$, invariant by translations of the domain $\Omega$, such that for all
  $\tau \geq 0$
  $$
  |K'_\tau| \leq C (1 + \tau)^{-3}\quad \text{in $d = 2$,}\qquad|K'_\tau| \leq C e^{-(d-2) \tau}
  \quad \text{in $d \geq 3$.}
  $$
\end{lem}

\begin{proof}
  Since $K_\tau = 1/I_\tau$, we have $K'_\tau = - I_\tau^{-2}I'_\tau$.   The estimates of $I_\tau$ from Lemma \ref{lem:It}, and the estimates of
  $I'_\tau$ from Lemma \ref{lem:It'} readily give the result.
\end{proof}

\begin{lem}[Estimates of $K_\tau' / K_\tau$]
  \label{lem:Kt'/Kt}
  Assume the hypotheses of Theorem \ref{thm:main-L1}. There exists a constant
  $C > 0$ such that for all $\tau \geq 0$
  $$
    \frac{ |K'_\tau| } {K_\tau}
    \leq C (1 + \tau)^{-1}
    \quad \text{if $d = 2$,}\qquad
    \frac{ |K'_\tau| } {K_\tau}
    \leq C e^{-(d-2) \tau}
    \quad \text{if $d \geq 3$.}
  $$
  The constant $C$ is invariant by translations of the domain
  $\Omega$.
\end{lem}

\begin{proof}
  Since $K_\tau = 1/I_\tau$ we have $|K'_\tau|/K_\tau=|I'_\tau|/I_\tau$. The estimates of $I_\tau$ from Lemma~\ref{lem:It}, and the estimates of $I'_\tau$ from Lemma~\ref{lem:It'} readily give the result for $d\ge 2$.
\end{proof}

\medskip
We also give an estimate on the difference of two fundamental
solutions to the heat equation on $\R^d$. This estimate is quite easy
to obtain with several methods, and we choose an explicit calculation
for brevity:

\begin{lem}
  \label{lem:Gaussian-L12-difference}
  Take $M > 0$. In all dimensions $d \geq 1$, and for any $t > 0$ and
  $v \in \R^d$ with $|v| \leq M \sqrt{t}$,
  \begin{equation*}
    \ird |x|^2 \big| \Gamma(t,x) - \Gamma(t,x-v) \big| \d x
    \lesssim
    \sqrt{t} |v|.
  \end{equation*}
  The implicit constant in the above inequality depends only on $d$
  and $M$.
\end{lem}

\begin{proof}
  By the mean value theorem one easily sees that for any $x, v \in
  \R^d$ and some $\xi$ in the interval $[x,x-v] = \{
  \theta x +
  (1-\theta)(x-v) \mid \theta \in [0,1]
  \}$,
  \begin{equation*}
    | e^{-|x|^2} -
    e^{-|x-v|^2} |
    \lesssim |v| |\xi| e^{-|\xi|^2}
    \lesssim |v| e^{-\frac{|\xi|^2}{2}}
    \lesssim |v| e^{-\frac{|x|^2}{4}} e^{\frac{|v|^2}{2}},
  \end{equation*}
  since one can see that $|\xi|^2 \geq \frac12 |x|^2 - |v|^2$. This
  implies
  \begin{equation*}
    \big| \Gamma(t,x) - \Gamma(t,x-v) \big|
    \lesssim t^{-\frac{d}{2}} \frac{|v|}{\sqrt{t}}
    e^{-\frac{|x|^2}{8t}} e^{-\frac{|v|^2}{4t}},
  \end{equation*}
  and integrating against $|x|^2$,
  \begin{equation*}
    \ird |x|^2 \big| \Gamma(t,x) - \Gamma(t,x-v) \big| \d x
    \lesssim
    t^{-\frac{d}{2}} \frac{|v|}{\sqrt{t}} e^{-\frac{|v|^2}{4t}}
    \ird |x|^2
    e^{-\frac{|x|^2}{8t}} \d x
    \lesssim
    \sqrt{t} |v| e^{-\frac{|v|^2}{4t}}.
    \qedhere
  \end{equation*}
\end{proof}

\section{Logarithmic Sobolev inequalities}
\label{sec:logsob}

Since all our main results depend on Hypothesis \ref{hyp:logsob}, we
dedicate this section to studying its validity. We will show in
\Cref{prp:logSob-Ft-d2+} that it holds for domains obtained by a
suitable deformation of a ball. It may in fact hold for more general
domains, but investigating this is a separate question from the
results of this paper.

Let us start by gathering some basic results on this type of
inequalities.  We first note that \eqref{eq:log_sob} is equivalent to
the usual form of the logarithmic Sobolev inequality
$$
\frac{\lambda}{4} \int_{\Omega}  f^2 \log f^2 \mu
\leq \int_{\Omega} \left|\nabla f \right|^2 \mu,
$$
holding for all $f \in H^1(\Omega;\mu)$ with $\int_\Omega f^2 \mu = 1$. The
equivalence between the two is seen by setting $\mu=F_{\tau}$ and
$f=\sqrt{g/F_{\tau}}$. We need to show \eqref{eq:log_sob} for all the
functions $F_{\tau}$ with a constant $\lambda$ which does not depend on
$\tau$. In general, for a positive, integrable function
$F \: \Omega \to \R$ defined on an open set $\Omega \subseteq \R^d$,
we call $\lambdaL \equiv \lambdaL(F) \geq 0$ the best constant in the
inequality
\begin{equation*}
  \lambdaL H(g\,|\, F)\leq \int_{\Omega}g\left|\nabla
    \log\frac{g}{F}\right|^2
\end{equation*}
for all positive $g \in L^1(\Omega)$ with
$\int_\Omega g = \int_\Omega F$ (understanding the right-hand side to
be equal to $+\infty$ whenever $\log \frac{g}{F}$ does not have a weak
gradient, or when its weak gradient is defined but the integral on the
right-hand side is infinite). We say that $F$ satisfies a logarithmic
Sobolev inequality when $\lambdaL(F) > 0$. For later use, we also
denote by $\lambdaP \equiv \lambdaP(F) \geq 0$ the best constant in
the Poincaré inequality
\begin{equation*}
  \lambdaP \int_\Omega \left( 1 - \frac{g}{F} \right)^2 F
  \leq \int_{\Omega} \left|\nabla \frac{g}{F}\right|^2 F,
\end{equation*}
for all $g \in L^1(\Omega)$ with $\int_\Omega g = \int_\Omega F$ (with
the same understanding as before: the right-hand side is equal to
$+\infty$ unless $g/F$ has a weak gradient in $L^2(F)$),

There are several results that will be useful for us when estimating
the logarithmic Sobolev constant $\lambdaL$. The first one is a
consequence of the well-known curvature-dimension condition
\citep{Ane2000,Bakry2014}, and this particular statement can be
obtained from the theory presented in \citet[Section 5]{Ane2000} and
the proof of \textit{Corollaire 5.5.2} therein:
\begin{lem}
  \label{lem:curvature-logsob}
  Let $r > 0$,  $F \: (r,+\infty) \to (0,+\infty)$ be a positive,
  integrable function of the form
  \begin{equation*}
    F(x) = C e^{-\Phi(x)},
    \qquad x > r,
  \end{equation*}
  with $C > 0$ and $\Phi \: (r,+\infty) \to \R$ a convex,
  $\mathcal{C}^2$ function with $\Phi''(x) \geq \rho$ for all $x >
  r$. Then the logarithmic Sobolev inequality
  \begin{equation*}
    2 \rho H(g \,|\, F)
    \leq \int_r^\infty g \left| \left( \log\frac{g}{F}\right)' \right|^2
  \end{equation*}
  holds for all nonnegative $g \in L^1(r,+\infty)$ with $\int_\Omega g = \int_\Omega F$.
\end{lem}

We also cite a well known result on perturbation of these
inequalities by \citet{Holley1987}:
\begin{lem}
  \label{lem:HolleyStroock}
  Let $\Omega \subseteq \R^d$ be an open set, $F \: \Omega \to \R$ a
  positive, integrable function which satisfies a logarithmic Sobolev
  inequality with constant $\lambda_L$. Let $A \: \Omega \to \R$ be a
  measurable function such that $|A|$ is bounded. Then the function
  \begin{equation*}
    \widetilde{F}(x) = F(x) e^{-A(x)}, \qquad x \in \Omega,
  \end{equation*}
  also satisfies a logarithmic Sobolev inequality with constant
  $\lambda_L e^{\operatorname{osc}(A)}$, where
  $\operatorname{osc}(A) := \sup A - \inf A$.
\end{lem}
In the case in which the removed domain $U$ is a ball in $\R^d$
centred at the origin, the functions $F_{\tau}$ defined in
\eqref{eq:Ft} are radially symmetric, so the next result regarding
logarithmic Sobolev inequalities for radially symmetric functions will
be useful. It is a particular case of \citet[Theorem
1.1]{Cattiaux2019}:

\begin{lem}[Logarithmic Sobolev for radially symmetric functions]
  \label{lem:cattiaux}
  Let $d \geq 2$ and $F \: \R^d \to \R$ be a positive, integrable,
  radially symmetric function given by
  \begin{equation*}
    F(x) = |x|^{1-d} f(|x|), \qquad x \in \R^d,
  \end{equation*}
  where $f \: [0,+\infty) \to (0,+\infty)$ is a given function (which
  must be integrable, since $F$ is). There exists a constant $c > 0$,
  independent of $F$, such that
  \begin{equation}
    \label{eq:cattiaux}
    \lambdaL(F) \geq c \left( \frac{1}{\lambdaL(f)} + m_1(f)
      \max \left\{ \frac{1}{\lambdaP (f)}, \frac{m_2(f)}{d-1} \right\}^{1/2}
    \right)^{-1},
  \end{equation}
  where the moments $m_k(f)$ are defined by
  \begin{equation*}
    m_k(f) := \int_0^\infty r^k f(r) \d r.
  \end{equation*}
\end{lem}
In other words: if a positive, integrable function $f$ on
$(0,+\infty)$ satisfies a logarithmic Sobolev inequality and its
second moment is finite (which implies that the first moment must also be
finite), then its radial symmetrisation $F(x) = f(|x|)$ also satisfies
a logarithmic Sobolev inequality. We will not use the specific bound
on $\lambdaL(F)$ given in~\eqref{eq:cattiaux}, but just the fact that
it only depends on the logarithmic Sobolev and Poincaré constants of
$f$ and the first and second moments of $f$.

\subsection{Logarithmic Sobolev inequalities for the transient
  equilibrium $F_{\tau}$ outside a ball}
\label{sec:logSob-ball}

In this section we take $\Omega = \R^d \setminus \overline{B_R}$,
where $B_R$ is the open ball in $\R^d$ centred at $0$ with radius
$R > 0$. For $t \geq 0$, the ``transient equilibria'' $F_{\tau}$ are
given by \eqref{eq:Ft},
\begin{equation*}
  F_{\tau}(y) = (2 \pi)^{-d/2}  K_{\tau}\, \phi(e^\tau y)^2 e^{-\frac{|y|^2}{2}},
  \qquad \tau \geq 0, \ y \in e^{-\tau} \Omega =: \Omega_\tau,
\end{equation*}
where $\phi$ is the classical solution to the elliptic
equation~\eqref{eq:phi} singled out by Lemma~\ref{lem:phi}, and
$K_{\tau}$ is a normalisation to ensure that $F_{\tau}$ is a
probability density. We notice that~$\phi$ is explicit in this case:
\begin{subequations}
  \begin{align}
    \label{eq:phid1}
    \phi(x) = x \qquad & \text{in dimension $d=1$,}
    \\
    \label{eq:phid2}
    \phi(x) = \log \frac{|x|}{R} \qquad & \text{in dimension $d=2$,}
    \\
    \label{eq:phid3}
    \phi(x) = 1- \frac{|x|^{2-d}}{R^{2-d}} \qquad & \text{in dimension $d\geq 3$.}
  \end{align}
\end{subequations}
We will show logarithmic Sobolev inequalities for the measures
$F_{\tau}$ in all dimensions, with constants which are bounded below
independently of $\tau$.

As discussed before, in dimension $1$ we consider the domain
$(0,+\infty)$ since $\R \setminus \overline{B_R}$ is disconnected, and
it is enough to consider the evolution of the heat equation on the
half-line. In this case the transient equilibrium is independent of $\tau$:
\begin{equation*}
  F(y) = C_M y^2 e^{-\frac{y^2}{2}}, \qquad y \in (0,+\infty).
\end{equation*}
The measure $F$ satisfies a logarithmic Sobolev inequality with a
constant which can be bounded below by $2$ (the constant in the
logarithmic Sobolev inequality for the usual Gaussian) for all
$\tau \geq 0$.
\begin{lem}[Dimension $d=1$]
  \label{lem:logsob-1d}
  There exists $\lambda > 0$ such that
  \begin{equation*}
    2 H(g\,|\, F)\leq
    \int_0^\infty g\left|\p_y \log\frac{g}{F}\right|^2
  \end{equation*}
  for all nonnegative $g \in L^1(0, +\infty)$ with
  $\int_0^\infty g = 1$.
\end{lem}

\begin{proof}
  We may apply Lemma \ref{lem:curvature-logsob} to $F$, since
  $F = C_M e^{-\Phi}$ with
  $\displaystyle\Phi(y) = \frac{y^2}2  - 2 \log y$, which satisfies
  $\Phi'' \geq 1$.
\end{proof}

Here is our result in dimensions $2$ and higher, showing that the measure $F_{\tau}$ satisfies a logarithmic Sobolev inequality with a constant which can be bounded below uniformly in $\tau$:

\begin{lem}[Dimension $d \geq 2$]
  \label{lem:logsob-2d}
  Take $R > 0$ and $\Omega = \R^d \setminus \overline{B_R}$. In
  dimension $d \geq 2$ there exists $\lambda > 0$ such that
  \begin{equation*}
    \lambda H(g\,|\, F_{\tau})\leq
    \int_{\Omega_\tau} g\left|\nabla \log\frac{g}{F_{\tau}}\right|^2
\end{equation*}
for all $\tau \geq 0$, and for all nonnegative $g \in L^1(\Omega_{\tau})$ with $\int_{\Omega_{\tau}} g = 1$.
\end{lem}

\begin{proof}
  We need to show that the logarithmic Sobolev constants
  $\lambdaL(F_{\tau})$ are bounded below by a positive constant
  $\lambda > 0$. Let us first prove the case $d=2$. We use the radial
  symmetry of $F_{\tau}$ to write $F_{\tau}(x) = |x|^{-1}
  f_\tau(|x|)$, with
  \begin{equation*}
    f_\tau(r) := K_\tau r \phi_\tau(r)^2 e^{-\frac{r^2}{2}}
    =  K_\tau r \Big(
    \tau + \log \frac{r}{R}
    \Big)^2 e^{-\frac{r^2}{2}}, \qquad r \geq R e^{-\tau}.
  \end{equation*}
  In order to use Lemma \ref{lem:curvature-logsob} we write
  $f_\tau(r) = K_\tau e^{-\Phi(r)}$, with
  \begin{equation}
    \label{eq:1}
    \Phi(r) := \frac{r^2}{2} - 2 \log \phi_\tau(r) - \log r, \qquad r
    \geq R e^{-\tau}.
  \end{equation}
  Since the function $\phi_\tau(r) = \tau + \log (r/R)$ is positive for
  $r > Re^{-\tau}$, increasing and concave, one sees that
  \begin{equation*}
    \frac{\mathrm{d}^2}{\mathrm{d}r^2} (\log \phi(r))
    = \frac{\phi''(r)}{\phi(r)} - \frac{(\phi'(r))^2}{(\phi(r))^2} \leq 0,
  \end{equation*}
  so $r \mapsto \log \phi(r)$ is concave. Since $r \mapsto \log r$ is
  also concave, from \eqref{eq:1} we see that $\Phi$ satisfies
  $\Phi''(r) \geq 1/2$ for all $r > Re^{-\tau}$. We may then apply
  Lemma~\ref{lem:curvature-logsob} to obtain that
  \begin{equation*}
    \lambdaL(\Phi) \geq 1.
  \end{equation*}
  As a consequence of Lemma \ref{lem:cattiaux} we can find an explicit
  constant $\lambda > 0$ such that
  \begin{equation*}
    \lambdaL(F_\tau) \geq \lambda \qquad \text{for all $\tau \geq 0$.}
  \end{equation*}
  Notice that the first and second moments of $f_\tau$ can be seen to
  be bounded above by a constant which is uniform in $t$, so the
  quantity on the right-hand side of the bound in Lemma
  \ref{lem:cattiaux} is independent of $t$.

  In dimension $d \geq 3$ the same proof works, since the function
  $\phi(r) = (1 - (r/R)^{2-d})$ is still positive on $(1,+\infty)$,
  increasing and concave.
\end{proof}

\subsection{Logarithmic Sobolev inequalities for the transient equilibrium $F_{\tau}$ outside general domains}
\label{sec:logSob-Omega}

As a consequence of the logarithmic Sobolev inequalities outside a
ball developed in the previous section we can also obtain inequalities
in general exterior domains, as long as they are a suitable
deformation of a ball. To be more precise, let $\Omega$ be a domain in
dimension $d \geq 2$ satisfying \eqref{eq:hypOmega1} and define
$R > 0$ by $|\R^d \setminus \Omega| = |B_R|$, where $B_R$ is the unit ball with
radius $R$. Let us assume the following:
\begin{hyp}
  \label{hyp:diffeomorphism}
  There exists a $\mathcal{C}^2$ diffeomorphism $\Psi \: \R^d \to
  \R^d$ such that for some $R, R_0, C_\Psi > 0$ we have
  \begin{subequations}
    \label{eq:Psi}
    \begin{gather}
      \label{eq:Psi1}
      \Psi(B) = U, \quad \Psi(\partial B_R) = \partial U, \quad \Psi(B_R^c) = \Omega;
      \\
      \label{eq:Psi2}
      \text{$\Psi(x) = x$ for $|x| \geq R_0$};
      \\
      \label{eq:Psi3}
      \text{$J_\Psi(x) = 1$ for all $x \in \R^d$};
      \\
      \label{eq:Psi4}
      \text{$\| (D \Psi(x))^{-1} \|_2 \geq C_\Psi > 0$ for all $x \in \R^d$},
    \end{gather}
  \end{subequations}
  where $D \Psi(x)$ is the Jacobian matrix of $\Psi$ at the point $x$,
  $J_\Psi(x)$ is its determinant, and $B^c$ stands for the complement of the
  unit ball in $\R^d$.
  \end{hyp}
  Following the ideas of Section \ref{sec:entropy}, we define
  for $\tau \geq 0$
  \begin{equation*}
    \label{eq:Omega_t}
    \Omega_\tau := e^{-\tau}\Omega,\quad
    \text{and}\quad B^c_\tau := e^{-\tau} B_R^c.
  \end{equation*}
  By $\phi_{R}$ we will denote the solution to problem~\eqref{eq:phi}
  outside the ball, explicitly given by \eqref{eq:phid2} or
  \eqref{eq:phid3}. As usual, we also denote by $\phi$ the solution
  of problem~\eqref{eq:phi} given by Lemma \ref{lem:phi} with
  $\Omega=\mathbb{R}^d \setminus \overline{U}$.
  Under these assumptions we would like to prove the log-Sobolev
  inequality
  \begin{equation}
    \label{eq:logSob-general}
    \lambda_\Omega \int_{\Omega_\tau} g(y) \log \frac{g(y)}{F_{\tau}(y)}\d y
    \leq \int_{\Omega_\tau} \left|\nabla \log \frac{g}{F_{\tau}} (y)\right|^2g(y) \d y
  \end{equation}
  for all $\tau \geq 0$ and all positive $g$ with
  $\int_{\Omega_\tau} g = 1$.  If we take the change of variables
  $y = e^{-t}\Psi(e^tx)$ in \eqref{eq:logSob-general} and rename
  $\widetilde{g}(x)=g(e^{-t}\Psi(e^tx))$ and
  $\widetilde{F}_{\tau}(x) = F_{\tau}(e^{-\tau}\Psi(e^\tau x))$ we see
  \eqref{eq:logSob-general} is equivalent to
  \begin{equation*}
    \label{eq:49}
    \lambda_\Omega
    \int_{B_\tau^c} \widetilde{g}(x)
    \log \frac{\widetilde{g}(x)}{\widetilde{F}_{\tau}(x)}\d x
    \leq
    \int_{B_\tau^c}
    \left|\left(
        \nabla \log \frac{\widetilde{g}}{\widetilde{F}_{\tau}}(x)
        \right)
      (D \Psi(e^\tau x))^{-1}
    \right|^2
    \widetilde{g}(x) \d x
  \end{equation*}
  for all $\widetilde{g}$ with $\int_{B_\tau^c} \widetilde{g} = 1$,
  where we used that $J_\Psi$ is always $1$ (and hence the change of
  variables we are using also has Jacobian $1$). By our assumption
  \eqref{eq:Psi4}, we also have for any $v, z \in \R^d$,
  \begin{equation*}
    |v|^2 =
    |v (D\Psi (z))^{-1} (D\Psi (z))|^2 \leq
    C_\Psi^2\, | v (D\Psi (z))^{-1} |^2.
  \end{equation*}
  Hence in order to show \eqref{eq:logSob-general} it is enough to prove
  \begin{equation*}
    \label{eq:50}
    C_\Psi^2 \lambda_\Omega
    \int_{B_\tau^c} \widetilde{g}
    \log \frac{\widetilde{g}}{\widetilde{F}_{\tau}}\d x
    \leq
    \int_{B_\tau^c}
    \left|
        \nabla \log \frac{\widetilde{g}}{\widetilde{F}_{\tau}}
    \right|^2
    \widetilde{g} \d x,
  \end{equation*}
  which is precisely a logarithmic Sobolev inequality for the density
  $\widetilde{F}_\tau$. We may write
  \begin{gather*}
    \widetilde{F}_\tau(x)
    = F^B_\tau (x) e^{-A(x)},
\quad\text{where}\\
    A(x) := \log F^B_\tau(x) - \log \widetilde{F}_\tau(x)
    =
    2 \log \frac{\phi_R(e^{\tau}x)}{\phi( \Psi(e^\tau x) )}
    + \log \frac{G(x)}{ G(e^{-\tau} \Psi(e^\tau x))}.
  \end{gather*}
  Using the properties of $\Psi$ from
  \eqref{eq:Psi1}--\eqref{eq:Psi4}, and the behaviour of $\phi$ at the
  boundary given by Lemma \ref{lem:phi2}, we see that there exist $0 <
  c_1 < c_2$ such that
  \begin{equation*}
    c_1
    \leq \frac{\phi_R(e^{\tau}x)}{\phi( \Psi(e^\tau x) )}
    \leq c_2,
    \qquad
    c_1 \leq
    \frac{G(x)}{ G(e^{-\tau} \Psi(e^\tau x))}
    \leq c_2\quad\text{for all $\tau \geq 0$ and all $x \in B_\tau^c$.}
  \end{equation*}
  This shows that $A$  has finite oscillation $\operatorname{osc}(A) := \sup A - \inf
  A$. By the Holley-Stroock perturbation Lemma \ref{lem:HolleyStroock}
  we obtain the following result:

  \begin{prp}[Logarithmic Sobolev inequality for $F_{\tau}$]
    \label{prp:logSob-Ft-d2+}
    Assume Hypothesis \ref{hyp:diffeomorphism} in dimension $d \geq
    2$. There exists $\lambda = \lambda(\Omega) > 0$, independent of
    $\tau$, such that the logarithmic Sobolev inequality
    \begin{equation*}
      \lambda \int_{\Omega_\tau} g \log \frac{g}{F_{\tau}}\d y
      \leq \int_{\Omega_\tau}\left|\nabla \log \frac{g}{F_{\tau}} \right|^2g \d y
    \end{equation*}
    holds for all $\tau \geq 0$ and all positive $g \in L^1(\Omega_\tau)$
    with $\int_{\Omega_\tau} g = 1$.
  \end{prp}

  We actually make the following conjecture, which we have been unable
  to prove or disprove: if $\lambda_\tau$ is the optimal constant in
  the above logarithmic Sobolev inequality, then we expect that
  \begin{equation*}
    \lim_{\tau \to +\infty} \lambda_\tau = 2.
  \end{equation*}
  This seems reasonable, since $2$ is the optimal constant for the
  standard Gaussian in $\R^d$, and $F_{\tau}$ approaches a standard
  Gaussian as $\tau \to +\infty$. However, this approach happens in a
  quite singular way which does not allow for the use of standard
  perturbation results for logarithmic Sobolev inequalities.

\section{$L^1$ estimates with weight $\phi$}
\label{sec:l1}

This section is devoted to the proof of Theorem \ref{thm:main-L1}. We
split it in three parts, according to the spatial dimension.

\subsection{Convergence in dimension $d \geq 3$}
\label{sec:d3}

We start by proving the result in the Fokker-Planck variables introduced in Section \ref{sec:entropy}:
\begin{prp}
  \label{prp:d3}
  Assume the conditions of Theorem \ref{thm:main-L1} in dimension
  $d \geq 3$. There exists a constant $C = C(d, \Omega) > 0$,
  invariant by translations of $\Omega$, such that
  \begin{equation}\label{eq:weighted.convergence.d.ge3}
      \| g(\tau) - m_\phi F_\tau \|_{L^1(\Omega)}\leq C m_\phi^{1/2} \left(h_0 + M_1 \right)^{1/2} e^{-\frac\lambda2 \tau}
      \qquad \text{for all $\tau \geq 0$,}
  \end{equation}
  where $\displaystyle h_0 := \int_\Omega \phi(x) u_0(x) \log \frac{u_0(x)}{m_\phi k_{1/2} \phi(x) G(x)} \d x$.
\end{prp}

\begin{proof}
  We assume that $u_0$ is such that $m_\phi = 1$ (equivalently, $\|g_0\|_1 = 1$); for a general nonnegative (and
  nontrivial) $u_0$, the statement applied to $g / m_\phi$ gives the full result.

Combining~\eqref{eq:derivada_entropia} with the logarithmic Sobolev inequality~\eqref{eq:log_sob-hyp.1} from Hypothesis~\ref{hyp:logsob} we get
  \begin{equation}
    \label{eq:10}
    \ddtau H(g(\tau) \,|\, F_\tau)
    \leq - \lambda H(g(\tau) \,|\, F_\tau)
    - \int_{\Omega_\tau} g(\tau) \frac{\p_\tau F_\tau}{F_\tau},
    \qquad \tau \geq 0.
  \end{equation}
  In order to estimate the second term in the right-hand side of~\eqref{eq:10} we write
  \begin{equation*}
    \frac{\partial_\tau F_\tau}{F_\tau}= A_1(\tau) + A_2(\tau), \quad\text{where }
    A_1(\tau)= \frac{K'_\tau}{K_\tau},\quad
    A_2(\tau)(y)= \frac{2 \nabla \phi(e^\tau y)  \cdot(e^\tau  y)}{\phi(e^\tau y)}.
  \end{equation*}
  We know from Lemma~\ref{lem:Kt'/Kt} that
  $|A_1(\tau)| \lesssim e^{-(d-2) \tau}$, so
  \begin{equation*}
    \Big | \int_{\Omega_\tau} g(\tau)  A_1(\tau)\Big|
    \lesssim \| g_0 \|_1 e^{-(d-2) \tau}
    = m_\phi \,e^{-(d-2) \tau} =e^{-(d-2) \tau}.
  \end{equation*}
  Besides, undoing the change of variables
  ~\eqref{eq:change1},~\eqref{eq:change2}, and using the
  bound~\eqref{eq:gradphi-d3} on~$|\nabla \phi|$ and
  Lemma~\ref{lem:negative-moments}, we have for all $\tau \geq 0$ that
  \begin{align*}
    \Big| \int_{\Omega_t} g(\tau) A_2(\tau)\Big| &\leq 2 \int_{\Omega_t}
    g(\tau, y) \frac{|\nabla \phi(e^\tau y) |\, |e^\tau y|}{\phi(e^\tau y)}\d y
    = 2 \int_\Omega u(t,x) |\nabla \phi(x)| \, |x| \d x
    \\
    &
      \lesssim
      \int_\Omega u(t,x) |x-x_0|^{1-d} |x| \d x \lesssim
      M_1 (1 + t) ^{-(d-2)/2}
      \leq
      M_1 e^{-(d-2) \tau}.
  \end{align*}
  Recalling that we are assuming that
  $m_\phi=1$, so that $1 = m_\phi \leq M_1$, we have then
  \begin{equation*}
    \left| \int_{\Omega_\tau}g(\tau)\frac{\partial_\tau F_\tau}{F_\tau} \right|
    \lesssim
    M_1 e^{-(d-2) \tau}.
  \end{equation*}
  Plugging this in equation \eqref{eq:10} we get
  \begin{equation}
    \label{eq:46}
    \ddt H(g(\tau) \,|\, F_\tau)
    \leq - \lambda H(g(\tau) \,|\, F_\tau)
    + C M_1 e^{-\tau(d-2)},
  \end{equation}
  which can be easily integrated to obtain (recall that we are assuming $\lambda < d-2$)
  \begin{equation*}
    H(g(\tau)\,|\, F_\tau)
    \leq
    \left( h_0+\frac{CM_{1}}{d-2-\lambda}\right)
    e^{-\lambda \tau}
    \lesssim
    \big( h_0 + M_1 \big)\,
    e^{-\lambda \tau}.
  \end{equation*}
  The desired result~\eqref{eq:weighted.convergence.d.ge3} (with $m_\phi=1$) follows then from Csiszár-Kullback's inequality~\eqref{eq:csiszar-kullback} .
\end{proof}

\begin{rem}
  \label{rem:lambda-value}
  The assumption $\lambda<d-2$ is only used when solving the differential inequality~\eqref{eq:46}. If we want to obtain better precision (and assuming we have better information on $\lambda$) we can of course solve the inequality for any $\lambda$. This leads to our conjectured rates of convergence from Remark \ref{rem:optimal_rates}.
\end{rem}

We can undo the change of variables in order to ``translate'' Proposition~\ref{prp:d3} from $g$ to $u$, and then use the entropy regularisation result in Lemma \ref{lem:phiuloguG-bound} to improve the dependence on $h_0$ and $m_{1,\phi}$. We thus obtain a result which is already very close to Theorem \ref{thm:main-L1} in dimensions $d \geq 3$.
\begin{lem}
  Theorem \ref{thm:main-L1} holds in $d \geq 3$ with the slightly weaker estimate
  \begin{equation*}
    \label{eq:L1-weaker}
    \int_\Omega \phi(x) \left|u(t,x) - m_\phi \phi(x)\Gamma(t,x)\right| \d x
    \leq \frac{C m_\phi^{1/2} M_{2,\phi}^{1/2}}{t^{\lambda/4}}
    \quad \text{for all }t \geq 2,
  \end{equation*}
  for some constant $C > 0$ depending only on the dimension $d$ and the domain $\Omega$, and invariant by translations of $\Omega$.
\end{lem}

\begin{proof}
  The estimate~\eqref{eq:weighted.convergence.d.ge3} rewritten under
  the change of variables~\eqref{eq:change1}--\eqref{eq:change-tx}
  reads
  \begin{equation*}
    \int_\Omega\phi(x)\Big|u(t,x)-k_{t+\frac12}m_\phi\phi(x)
    \Gamma\Big(t+\frac12,x\Big)\Big|\d x
    \lesssim m_\phi^{1/2}(h_0+M_1)^{1/2}(2 t +1)^{-\lambda/4},
  \end{equation*}
  valid for all $t > 0$ and all solutions $u$. If we call
  $u_{1/2}(x) := u(1/2,x)$, $x \in \Omega$, and we call $h_{1/2}$ the
  analog to $h_0$ at time $t=1/2$, we may use the above inequality for
  the solution starting at time $t=1/2$ to obtain
  \begin{equation*}
    \int_\Omega\phi(x)|u(t,x)-k_{t}m_\phi\phi(x)\Gamma(t,x)|\d x\lesssim m_\phi^{1/2}(h_{1/2}+M_{1}(1/2))^{1/2}\,t^{-\frac{\lambda}{4}}\qquad\text{for $t>\frac12$.}
  \end{equation*}
  Lemma \ref{lem:phi-rel-entropy-regularisation} yields
  $h_{1/2} \lesssim M_{2,\phi}$, and
  \Cref{cor:moment-regularisation-d3} shows
  $M_1(1/2) \lesssim M_{1,\phi} \lesssim M_{2,\phi}$. Then,
  \begin{equation}\label{eq:weighted.convergence.with.kt}
    \int_\Omega \phi(x)|u(t,x) -   k_{t} m_\phi\phi(x) \Gamma(t, x)| \d x
    \lesssim
    m_\phi^{1/2} M_{2,\phi}^{1/2} \, t^{-\frac{\lambda}{4}}\qquad \text{for $t > \frac12$.}
  \end{equation}
  By Proposition~\ref{prp:Kt-final} (see also equation \eqref{eq:kt-d3-changed}),
  \begin{equation*}
    |k_t - 1|  \int_\Omega \phi^2(x) \Gamma(t, x)\d x
    \lesssim t^{-\frac{d-2}{2}} \int_\Omega \Gamma(t, x)\d x
    = t^{-\frac{d-2}{2}} \lesssim t^{-\frac{\lambda}{4}}
    \qquad \text{for $t > \frac12$.}
  \end{equation*}
  This shows that we may remove $k_t$ from the left-hand side
  of~\eqref{eq:weighted.convergence.with.kt}; that is,
  \begin{equation*}
    \label{eq:31}
    \int_\Omega \phi(x)|u(t,x) - m_\phi\phi(x) \Gamma(t, x)| \d x
    \lesssim m_\phi^{1/2}M_{2,\phi}^{1/2} \, t^{-\frac{\lambda}{4}}
    \qquad \text{for $t > \frac12$.}
  \end{equation*}
  This shows the inequality in the statement (which is written for $t
  \geq 2$, for consistency in other statements, and since the lower
  bound on $t$ is unimportant).
\end{proof}

By taking an initial condition $u_0$ which approximates the delta
function $\delta_y$, using that $\phi\leq 1$ and that $\sqrt{a^2+b^2}\lesssim a+b$ we immediately obtain the following estimate for the heat kernel.

\begin{cor}
  Under the assumptions of \Cref{thm:main-uniform} in dimension $d \geq 3$,
  \begin{equation}
    \label{eq:kernel-L1-estimate-d3}
    \int_\Omega \phi(x) \left|
      p_\Omega(t,x,y) - \phi(x) \phi(y)  \Gamma(t,x)
    \right| \d x\leq\frac{C \phi(y) (1 + |y|)}{t^{\lambda/4}},
    \quad y \in \Omega, \ t \geq 2.
  \end{equation}
\end{cor}
This result is self-improving, and can be used to get the slightly
better bound in Theorem \ref{thm:main-L1}.
\begin{proof}[Proof of Theorem \ref{thm:main-L1} in $d \geq 3$]
  Using \eqref{eq:kernel-L1-estimate-d3}:
  \begin{align*}
    \int_\Omega \phi(x)
    &\left|
      u(t,x) - m_\phi \phi(x) \Gamma(t,x)
      \right| \d x
    \\
    &=
      \int_\Omega \phi(x) \left|
      \int_\Omega u_0(y) \Big(
      p_\Omega(t,x,y) - \phi(x) \phi(y) \Gamma(t,x)
      \Big) \d y
      \right| \d x
    \\
    &\leq
      \int_\Omega u_0(y)
      \int_\Omega \phi(x)  \left|
      p_\Omega(t,x,y) - \phi(x) \phi(y) \Gamma(t,x)
      \right| \d x    \d y
    \\
    &\leq
      C \int_\Omega u_0(y)
      \frac{\phi(y) (1 + |y|)}{t^{\lambda/4}}
      \d y
      \leq
      \frac{C}{t^{\lambda/4}} M_{1,\phi}.
      \qedhere
  \end{align*}
\end{proof}

\subsection{Convergence in dimension $d = 2$}
\label{sec:d2}

In dimension $2$ the proof follows the same strategy, but the estimates of the remainder term $R(\tau)$ are more involved. We start with some preliminary lemmas which give bounds for it.

\begin{lem}
  \label{lem:remainder-x_0}
  Let $\Omega \subseteq \R^2$ satisfying \eqref{eq:hypOmega1}, $x_0 \in \R^2 \setminus \overline{\Omega}$, and $c > 0$. In dimension $d=2$, there is a constant $C>0$ depending only on $c$ and $\dist(x_0, \Omega)$ such that
  \begin{equation*}
    \int_{\Omega} \frac{\Gamma(t, c(x-y))}{|x-x_0|^{2}}  \d x \leq C\,\frac{\log (2+t)}{1+t}\quad\text{for all }t > 0.
  \end{equation*}
\end{lem}

\begin{proof}
  Since the integral on the left-hand side is always bounded by $\dist(x_0, \Omega)^{-1}$, it is enough to prove the given bound for all $t > 1$ (for example). We consider the auxiliary truncated function, for an arbitrary $R>0$,
  \[
    f(x)=\begin{cases}
      \displaystyle \frac{1}{|x|^2},&|x|>R,\\[8pt]
      \displaystyle\frac{1}{R^2},&|x|\leq R.
    \end{cases}
  \]
  Now choose an $R>0$ small enough so that the ball of radius $R$ and center $x_0$ satisfies $B_R(x_0)\subset \R^2\setminus\overline\Omega$. Then we have
  \[
    \begin{aligned}
      \int_{\Omega} \frac{\Gamma(t, c(x-y))}{|x-x_0|^{2}}\d x&= \int_{\Omega} f(x-x_0) \Gamma(t, c(x-y)) \d x\\
      &\leq \int_{\R^2} f(x-x_0) \Gamma(t, c(x-y)) \d x\leq \int_{\R^2} f(x) \Gamma(t, cx) \d x
    \end{aligned}
  \]
  where the last inequality comes from the symmetry of $\Gamma$ in the spatial variable and Lemma~\ref{lem:convolution}. We now split the  last integral,
  \[
    \int_{\R^2} f(x) \Gamma(t, cx) \d x = \int_{B_R(0)}\frac{1}{R^2} \Gamma(t, cx) \d x + \int_{B^c_R(0)}\frac{1}{|x|^2} \Gamma(t, cx) \d x.
  \]
  The first integral on the right-hand side is bounded by $C/t$. For the second one we do the change of variables $x=\xi\sqrt{t}$, then we pass to radial
  coordinates $r=|\xi|$ and finally we split the resulting integral from radius $R/\sqrt{t}$ to 1 and from 1 to $\infty$ (we may assume that $R\le1$), arriving at
  \begin{align*}
    \int_{\Omega} \frac{\Gamma(t, c(x-y))}{|x-x_0|^{2}}  \d x&\lesssim\frac{1}{t} + \frac{1}{t}\int_{\frac{R}{\sqrt{t}}}^1\frac{e^{-cr^2}}{r} \d r + \frac{1}{t}\int_{1}^\infty \frac{e^{-cr^2} }{r} \d r\\
    &\lesssim\frac{1}{t}\Big(1+\int_{\frac{R}{\sqrt{t}}}^1 \frac{1}{r} \d r\Big) \lesssim C\frac{\log t}{t},
  \end{align*}
  for all $t > 1$. This shows the result, since $\log t / t$ is asymptotic to $\log (2+t) / (1 + t)$.
\end{proof}

\begin{lem}[Remainder estimate away from $\tau = 0$]
  \label{lem:remainder-dim2-largetime}
  In dimension $d=2$, if $m_\phi=1$,
  \begin{equation*}
    |R(\tau)|^2 \lesssim \frac{1}{\tau^2} H(g(\tau) \,|\, F_\tau)
    \left( 1 + |x_0|^2 \tau e^{-2 \tau} \right)\quad\text{for all }\tau \geq 2.
  \end{equation*}
  The implicit constant is invariant under translations of $\Omega$.
\end{lem}

\begin{proof}
  We use estimate \eqref{eq:remainder-estimate}, which gives $|R(\tau)|^2 \lesssim H(g(\tau) \,|\, F_\tau)\, Q_g(\tau)$, with
  \begin{align*}
    Q_g(\tau)&:=\int_{\Omega_\tau}\frac{ |\nabla \phi(e^\tau y)|^2 |e^\tau y|^2}{\phi^2(e^\tau y)}(g(\tau,y) + F_\tau(y))\d y\\
    &\lesssim\underbrace{\int_{\Omega_\tau}\frac{ |e^\tau y|^2}{|e^\tau y - x_0|^2 \phi^2(e^\tau y)}g(\tau,y)\d y}_{{\rm I}(\tau)}+\underbrace{\int_{\Omega_\tau}\frac{ |e^\tau y|^2}{|e^\tau y - x_0|^2 \phi^2(e^\tau y)}F_\tau(y)\d y}_{{\rm II}(\tau)},
  \end{align*}
  where we used the estimate~\eqref{eq:gradphi-d3} for
  $|\nabla \phi|$. To estimate ${\rm I}(\tau)$ we use the upper
  estimates on the heat kernel from
  equation \eqref{eq:SC-d2-tlarge-2} and the change of
  variables \eqref{eq:change-g}--\eqref{eq:change-tx},
  \begin{align*}
    {\rm I}(\tau)&=\int_{\Omega}\frac{ |x|^2}{|x-x_0|^2 \phi(x)}u (t,x)\d x\\
    &=\int_{\Omega} \int_{\Omega}\frac{ |x|^2}{|x-x_0|^2 \phi(x)}u_0(y) p_\Omega(t,x,y)\d x \d y
    \\
    &\lesssim\frac{1}{(\log t)^2}\int_{\Omega} u_0(y) \phi(y) \int_{\Omega}\frac{ |x|^2}{|x-x_0|^2} \Gamma(t, c(x-y))\d x \d y
    \\
    &\leq\frac{1}{(\log t)^2}\int_{\Omega}u_0(y)\phi(y)\int_{\Omega}\left(1+\frac{|x_0|^2}{|x-x_0|^2}\right)\Gamma(t, c(x-y))\d x \d y
    \\
    &\lesssim \frac{1}{(\log t)^2} \left( 1 + \frac{|x_0|^2 \log t}{t}\right),
  \end{align*}
  since $m_\phi = 1$, where we also used $|x|^2\leq |x-x_0|^2+|x_0|^2$ and the estimate from Lemma~\ref{lem:remainder-x_0} for the last inequality. Notice that this is valid for $\tau \geq 2$ (so $t$ is also larger than a strictly positive constant). As for ${\rm II}(\tau)$,
  \begin{align*}
    {\rm II}(\tau)
    &=\int_{\Omega}\frac{|x|^2}{|x-x_0|^2}k_{t+\frac12} \Gamma(t+\frac12, x)\d x
    \lesssim\frac{1}{(\log t)^2}\int_{\Omega}\frac{|x|^2}{|x-x_0|^2}\Gamma(t+\frac12, x)\d x
    \\
    &\lesssim\frac{1}{(\log t)^2} \left( 1 + \frac{|x_0|^2 \log t}{t}\right),
  \end{align*}
  where the last inequality is obtained similarly as before. Hence, we finally have
  \begin{equation*}
    Q_g(\tau)
    \lesssim
    \frac{1}{(\log t)^2} \left( 1 + \frac{|x_0|^2 \log t}{t} \right)
    =: q(x_0, t).
  \end{equation*}
  Writing the bound in terms of $\tau$ gives the result.
\end{proof}

\begin{lem}[Remainder estimate for small times]
  \label{lem:remainder-dim2-smalltime}
  In dimension $d=2$,
  \begin{equation*}
    |R(\tau)| \lesssim M_1
    \qquad\text{for all }\tau\in[0,2].
  \end{equation*}
  The implicit constant is invariant under translations of $\Omega$.
\end{lem}

\begin{proof}
  As usual, it is enough to prove this when $m_\phi = 1$. We use the
  expression of $R(\tau)$ from \eqref{eq:remainder-rewritten} and the
  bound $|\nabla \phi(x)| \lesssim |x-x_0|^{-1}$ from Section
  \ref{sec:phi} to get
  \begin{align*}
    |R(\tau)|
    &\leq
    2 \int_{\Omega_\tau} \frac{ |\nabla \phi(e^\tau y)| |e^\tau y|}{\phi(e^\tau y)}
    \big| g(\tau,y) - F_\tau(y) \big| \d y
    \\
    &\lesssim
    \int_{\Omega_\tau} \frac{ |e^\tau y|}{|e^\tau y - x_0| \phi(e^\tau y)}
    \big| g(\tau,y) - F_\tau(y) \big| \d y
    \\
    &=
    \int_{\Omega} \frac{ |x|}{|x - x_0|}
    \Big| u(t, x) - k_{t+\frac12} \Gamma(t+\frac12, x) \Big| \d x
    \\
    &\lesssim (\operatorname{dist}(x_0,\Omega))^{-1}\Big(
    \int_{\Omega} |x| u(t, x) \d x
    + \int_{\Omega} |x| \Gamma(t+\frac12, x) \d x\Big).
  \end{align*}
  Now, for times $t$ in the bounded interval which corresponds to
  $\tau \in [0,2]$,
  \begin{equation*}
    \int_{\Omega} |x| u(t, x) \d x \lesssim M_1,
    \qquad
    \int_{\Omega} |x| \Gamma(t+\frac12, x) \d x \lesssim 1,
  \end{equation*}
  where the first estimate is given by
  \Cref{cor:2d_moment_regularisation} and the second one is a
  straightforward explicit calculation. This gives the estimate in the
  statement.
\end{proof}

\begin{lem}[Differential inequality]
  \label{lem:differential_inequality}
  Let $q \:[0,+\infty)\to[0,+\infty)$ be a continuous function. If $h\:[0,+\infty)\to[0,+\infty)$ is a $\mathcal{C}^1$ function satisfying
  \begin{equation*}
    \ddtau h(\tau)\leq - \lambda h(\tau) + \sqrt{h(\tau)} \sqrt{q(\tau)}\qquad \text{for all $\tau \geq 0$,}
  \end{equation*}
  then
  \begin{equation*}
    h(\tau) \leq 2 e^{-\lambda \tau} h_0+ e^{-\lambda \tau}\left( \int_0^\tau e^{\frac{\lambda}{2} s}\sqrt{q(s)}\d s\right)^2
      \qquad \text{for all $\tau \geq 0$.}
  \end{equation*}
\end{lem}

\begin{proof}
  As can be easily checked, the largest possible solution to the differential \emph{equality}
  \begin{equation*}
    \ddt v(\tau) = - \lambda v(\tau) + \sqrt{v(\tau)} \sqrt{q(\tau)}
  \end{equation*}
  with initial condition $v(0) = v_0 := h(0)$ is
  \begin{equation*}
    v(\tau) = \left( e^{-\frac{\lambda}{2} \tau} \sqrt{v_0}+ \frac12 e^{-\frac{\lambda}{2} \tau}\int_0^\tau e^{\frac{\lambda}{2} s} \sqrt{q(s)} \d s\right)^2,\qquad \tau \geq 0.
  \end{equation*}
  This solution is unique when $v_0 > 0$; and it is the smallest possible when $v_0 = 0$ (solutions which are $0$ on some interval of
  the form $[0,\alpha)$ for $\alpha \in (0, +\infty]$ are also possible in this case). Well known results on differential
  inequalities then show that $h(\tau) \leq v(\tau)$. The inequality $(a+b)^2 \leq 2 a^2 + 2 b^2$ then gives the form in the statement.
\end{proof}

Now we give a result stated in terms of the function $g$, obtained
from $u$ by the change of variables in Section \ref{sec:entropy}:

\begin{prp}
  \label{prp:d2}
  Assume the conditions of Theorem \ref{thm:main-L1}, in dimension $d=2$.  There is a constant $C > 0$, depending only on $\Omega$ and invariant by its translations, such that
  \begin{equation*}
    \|g(\tau)-m_\phi F_\tau\|_{L^1(\Omega)}
    \leq
    C \left(
    \frac{m_\phi}{1+\tau}
    + e^{-\frac{\lambda}{2} \tau}
    m_\phi^{1/2} (h_0 + M_1 + |x_0|^2)^{1/2}
  \right)
  \quad\text{for all $\tau \geq 0$,}
\end{equation*}
where $\displaystyle h_0 := \int_\Omega \phi(x) u_0(x) \log \frac{u_0(x)}{k_{1/2}m_\phi\phi(x) G(x)} \d x$.
\end{prp}


\begin{proof}[Proof of Proposition \ref{prp:d2}]
  Without loss of generality, we assume that $m_\phi = 1$. We start
  from equation \eqref{eq:H-inequality-main2}: for $\tau \geq 0$,
  \begin{equation*}
    \label{eq:10-d2}
    \ddtau H(g(\tau) \,|\, F_\tau)
    \leq - \lambda H(g(\tau) \,|\, F_\tau)
    - R(\tau).
  \end{equation*}
  Now we use our bounds on $R(\tau)$ from Lemmas
  \ref{lem:remainder-dim2-largetime} and \ref{lem:remainder-dim2-smalltime}:
  \begin{equation*}
      \begin{aligned}
        &\ddtau H(g(\tau) \,|\, F_\tau)
        \leq - \lambda H(g(\tau) \,|\, F_\tau)
        + M_1,
        \qquad &&0 \leq \tau \leq 2,
        \\
        &\ddtau H(g(\tau) \,|\, F_\tau)
        \leq - \lambda H(g(\tau) \,|\, F_\tau)
        + \sqrt{H(g(\tau) \,|\, F_\tau)} \sqrt{q(\tau,x_0)}
        \qquad &&\tau \geq 2,
      \end{aligned}
  \end{equation*}
  where $\displaystyle q(\tau,x_0):=\frac{1}{\tau^2}\left( 1 + |x_0|^2 \tau e^{-2 \tau} \right)$.  For convenience, denote
  \begin{equation*}
    h(\tau) := H(g(\tau) \,|\, F_\tau),
    \qquad h_0 := h(0).
  \end{equation*}
  The first inequality shows that
  \begin{equation}
    \label{eq:h-shorttime}
    h(\tau) \leq h_0 + \tau M_1
    \qquad \text{for $0 \leq \tau \leq 2$,}
  \end{equation}
  so in particular $h(2) \leq h_0 + 2 M_{1,\phi}$. On the other hand, Lemma \ref{lem:differential_inequality} applied
  at the starting time $\tau = 2$ shows that
  \begin{align*}
    h(\tau) &\leq
    2 e^{-\lambda (\tau-2)} h(2)
    + e^{-\lambda (\tau-2)}
    \left( \int_2^{\tau} e^{\frac{\lambda}{2} (s-2)} \sqrt{q(s)} \d s
    \right)^2
    \\
    &\lesssim
    e^{-\lambda \tau} (h_0 + M_1)
    + e^{-\lambda \tau}
    \left( \int_2^{\tau} e^{\frac{\lambda}{2} s} \sqrt{q(s)} \d s
    \right)^2
  \end{align*}
  for all $\tau \geq 2$. Now we have, since we assume $\lambda < 2$,
  \begin{equation*}
    \int_2^{\tau} e^{\frac{\lambda}{2} s} \sqrt{q(s)} \d s
    \lesssim
    \int_2^{\tau} \frac{1}{s} e^{\frac{\lambda}{2} s} \d s
    + |x_0| \int_2^{\tau} \sqrt{s} e^{\frac{\lambda-2}{2} s} \d s
    \lesssim
    \frac{1}{\tau} e^{\frac{\lambda}{2} \tau}
    + |x_0|.
  \end{equation*}
  Hence, for all $\tau \geq 2$,
  \begin{equation*}
    h(\tau) \lesssim
    e^{-\lambda \tau} (h_0 + M_1 + |x_0|^2)
    +
    \frac{1}{\tau^2}.
  \end{equation*}
  Csiszár-Kullback's inequality \eqref{eq:csiszar-kullback} then
  implies that
  \begin{equation*}
    \| g(\tau) - F_\tau\|_1 \lesssim \frac{1}{\tau}
    + e^{-\frac{\lambda}{2} \tau} \left(
      h_0 + M_1 + |x_0|^2
    \right)^{1/2}
  \end{equation*}
  for all $\tau \geq 2$. For $0 \leq \tau \leq 2$,
  \eqref{eq:h-shorttime} implies that
  $\| g(\tau)-F_\tau\|_1 \lesssim (h_0 + M_1)^{1/2}$, so we obtain
  the bound in the statement.
\end{proof}

As for dimensions $d\ge3$, this implies a slightly weaker version of Theorem~\ref{thm:main-L1} in $d=2$.

\begin{lem}
  \Cref{thm:main-L1} holds in $d=2$ with the following slightly weaker estimate:
  \begin{equation*}
    \label{eq:main-L1-d2-weaker}
    \int_\Omega \phi(x) \left|
      u(t,x) - k_t m_\phi \phi(x) \Gamma(t,x)
    \right| \d x
    \le
    C \Big(\frac{m_\phi}{\log t}
    + \frac{m_\phi^{1/2} M_{2,\phi}^{1/2}}{t^{\lambda/4}} + \frac{m_\phi |x_0|}{t^{\lambda/4}}\Big)
  \end{equation*}
  for all $t \geq 2$.
\end{lem}

\begin{proof}
Changing variables back to $t$ and $x$ in Proposition~\ref{prp:d2} we obtain
  \begin{equation*}
    \begin{aligned}
      &\int_\Omega \phi(x) \left|u(t,x)- k_{t+1/2} m_\phi\phi(x)
        \Gamma\Big(t + \frac12 ,x\Big)\right| \d x
      \\
      &\hskip1.5cm\lesssim\frac{m_\phi}{1 + \log (2t+1)}
        + m_\phi^{1/2}(h_0 + M_1 + |x_0|^2)^{1/2} (1+t)^{-\lambda/4}
    \end{aligned}
  \end{equation*}
  for all $t \geq 0$. Since $m_\phi$ is invariant over time, we may
  apply this to the solution starting at time $t=1/2$ to get
  \begin{equation}
    \label{eq:38}
    \begin{aligned}
      &\int_\Omega \phi(x) |u(t,x)- k_{t} \phi(x) \Gamma(t,x)| \d x
      \\
      &\hskip1.5cm\lesssim\frac{m_\phi}{1 + \log (2t)}
        + m_\phi^{1/2}(h_{1/2} + M_1(1/2) + |x_0|^2)^{1/2} (1+t)^{-\lambda/4}
    \end{aligned}
  \end{equation}
  for all $t \geq 1/2$, where $h_{1/2}$ denotes the relative entropy $H(\phi u(1/2, \cdot) \,|\, k_{1/2} m_\phi \phi^2 G)$. Now we use
  Lemma~\ref{lem:phi-rel-entropy-regularisation} and~\Cref{cor:2d_moment_regularisation}  to estimate
  \[
    h_{1/2} \lesssim  M_{2,\phi} + \log\log(2+|x_0|),
    \qquad
    M_1(1/2) \lesssim M_{1,\phi} \lesssim M_{2,\phi}.
  \]
  From our last two estimates,
  \[
  h_{1/2} + M_1(1/2) + |x_0|^2\lesssim M_{2,\phi} + |x_0|^2,
  \]
  so equation \eqref{eq:38} in fact gives the estimate in the
  statement for $t \geq 2$ (since for $t \geq 2$ we may substitute
  $1 + \log(2t)$ by the asymptotically equivalent $\log t$, and $1+t$
  by $t$).
\end{proof}

As in the case $d \geq 3$, by approximating $\delta_y$ with a sequence of integrable initial conditions $u_0$ we immediately obtain the following
estimate on the heat kernel.

\begin{cor}
  Under the assumptions of \Cref{thm:main-uniform} in dimension $d=2$,
  \begin{equation}
    \label{eq:kernel-L1-estimate-d2}
    \begin{aligned}
    &\int_\Omega \phi(x)\left|p_\Omega(t,x,y) - k_t \phi(x) \phi(y) \Gamma(t,x)\right| \d x
    \\
    &\hskip1.5cm\leq\frac{C\phi(y)}{\log(2t+1)}+\frac{C\phi(y)(1+|y|^2)^{\frac{1}{2}}}{(1+t)^{\lambda/4}}+\frac{C \phi(y)|x_0|}{(1+t)^{\lambda/4}}
    \end{aligned}
  \end{equation}
  for all $y \in \Omega$ and all $t \geq 2$.
\end{cor}

Now we can complete the proof of Theorem \ref{thm:main-L1} in $d=2$,
improving the moments $M_{2,\phi}$ to $M_{1,\phi}$.

\begin{proof}[Proof of Theorem \ref{thm:main-L1} in $d=2$]
  Using \eqref{eq:kernel-L1-estimate-d2}:
  \begin{align*}
    \int_\Omega &\phi(x) |u(t,x) - k_t m_\phi \phi(x) \Gamma(t,x)| \d x
    \\
    &=\int_\Omega \phi(x) \left|\int_\Omega u_0(y) \Big(p_\Omega(t,x,y) - k_t \phi(x) \phi(y) \Gamma(t,x)\Big) \d y\right| \d x
    \\
    &\leq\int_\Omega u_0(y)\int_\Omega \phi(x)\left|p_\Omega(t,x,y) - k_t \phi(x) \phi(y) \Gamma(t,x)\right| \d x\d y
    \\
    &\lesssim\int_\Omega u_0(y)\frac{\phi(y)}{\log(2t + 1)} \d y+\int_\Omega u_0(y)\frac{\phi(y) (1 + |y|^2)^{\frac{1}{2}}}{(1+t)^{\lambda/4}}\d y +\int_\Omega u_0(y)\frac{\phi(y) |x_0|}{(1+t)^{\lambda/4}}\d y
    \\
    &\lesssim\frac{m_\phi}{\log{2t +1}}+\frac{M_{1,\phi} + m_\phi |x_0|}{(1+t)^{\lambda/4}},
  \end{align*}
  where in the last bound we used $(1 + |y|^2)^{1/2} \leq 1 + |y|$.
\end{proof}

\subsection{Convergence in dimension $d =1$}

In dimension $1$, since the complement of a compact interval is
disconnected, we only need to consider the equation on a half-line. In
contrast to the $d=2$ case, we have chosen to carry our first the
calculations on the domain $(0,+\infty)$, since they are especially
simple and serve as a good illustration of the method. Our results can
then be deduced from this particular case. We point out that one could
do this just as well in dimension $2$: one could write all estimates
in Section \ref{sec:d2} assuming $x_0=0$, and then deduce our final
estimates with a similar argument as the one we will use below. Since
in the $d=2$ case there is not a large advantage by doing so, we have
preferred to keep it this way.
\begin{proof}[Proof of Theorem \ref{thm:main-L1} in $d=1$] \fbox{Step
    1: $x_0=0$.}  In this case $\Omega = (0,+\infty)$, the harmonic
  profile is $\phi(x)=x$, the normalisation function is
  \begin{equation*}
    \label{eq:Ktau-d1}
    K_\tau = 2 e^{-2 \tau},\qquad \text{or equivalently,}\qquad k_t = \frac{1}{t},
  \end{equation*}
  and the ``transient equilibrium'' is
  \begin{equation*}
    F_\tau(y) = K_\tau \phi^2(e^{\tau} y) G(y) =
    K_\tau e^{2\tau} y^2 G(y)
    = \frac{\sqrt{2}}{\sqrt{\pi}}\, y^2 e^{-\frac{y^2}{2}},
  \end{equation*}
  which does not depend on $\tau$, and will hence be denoted by $F$
  instead.  We may apply Lemma \ref{lem:curvature-logsob} to $F$,
  since $F = (2 / \pi)^{\frac{1}{2}} e^{-\Phi}$, with
  \begin{equation*}
    \Phi(y) = \frac12 y^2 - 2 \log y,
    \qquad y > 0,
  \end{equation*}
  which satisfies $\Phi'' \geq 1$. If we assume
  \begin{equation*}
    m_\phi = \int_0^\infty x u_0(x) \d x = 1
  \end{equation*}
  we may use the corresponding logarithmic Sobolev inequality in
  \eqref{eq:derivada_entropia} to deduce
  \begin{equation*}
    \ddt H(g(\tau) \,|\, F)
    \leq -2 H(g(\tau) \,|\, F)
  \end{equation*}
  for $\tau \geq 0$. This differential inequality implies
  \begin{equation*}
    H(g(\tau) \,|\, F) \leq h_0 e^{-2 \tau},
  \end{equation*}
  with $h_0 := H(g_0 \,|\, F)$. By Csiszár-Kullback's inequality
  \eqref{eq:csiszar-kullback},
  $$
  \|g(\tau)- F\|_{L^1(\Omega)}\leq e^{-\tau} \sqrt{2 h_0}.
  $$
  After tracing back the change of variables
  \eqref{eq:change-tauy}--\eqref{eq:change-tx} to the original solution
  $u$ we obtain
  $$
  \int_0^\infty x \left|u(t, x) - 2 D\Big(t + \frac12, x\Big) \right|
  \dx \leq \frac{\sqrt{h_0}}{\sqrt{\frac12 + t}}.
  $$
  where $D = D(t,x)$ is the \emph{dipole function}
  $$
  D(t,x) = -\p_x \Gamma(t,x) = \frac{x}{2t} \Gamma(t,x)
  = \frac{t^{-\frac{3}{2}}}{2\sqrt{\pi}}\,x\,e^{-\frac{x^2}{4 t}},
  \qquad x \geq 0, \ t>0.
  $$
  This is true for all solutions $u$, so we may apply it to the
  solution starting at time $t = 1/2$ (i.e., the solution
  $(t,x) \mapsto u(t+1/2, x)$) and get, for all $t > 1/2$ and
  $x \geq 0$,
  $$
  \int_0^\infty x \left|u(t, x) - 2 D(t,x) \right| \dx
  \leq \frac{\sqrt{h_{\frac12}}}{\sqrt{1 + t}}.
  $$
  We apply now Lemma \ref{lem:phi-rel-entropy-regularisation} in $d=1$,
  which shows that the relative entropy $h_{\frac12}$ is
  ``regularised'': $h_{\frac12} \lesssim M_{2,\phi} = M_3 = 1 + m_3$. Hence
  \begin{equation*}
    \int_0^\infty x \left|u(t, x) - 2 D(t,x) \right|\dx
    \lesssim\frac{\sqrt{M_{2,\phi}}}{\sqrt{1 + t}}
    \qquad \text{for all $t > 1/2$.}
  \end{equation*}
  This shows, as in other dimensions, a slightly weaker form of the
  result: by a simple scaling argument we may remove the condition
  $\int_0^\infty x u_0 = 1$ and we get, for any initial condition $u_0$,
  that
  \begin{equation*}
    \int_0^\infty x \left|u(t, x) - 2 m_\phi D(t,x) \right|\dx
    \lesssim
    \frac{\sqrt{m_\phi}\sqrt{M_{2,\phi}}}{\sqrt{1 + t}}.
  \end{equation*}
  As a consequence, taking a sequence of integrable initial conditions
  approximating $\delta_y$ and passing to the limit,
  \begin{equation*}
    \int_0^\infty x \left|p_\Omega(t,x,y) - 2 y D(t,x) \right|
    \dx
    \lesssim
    \frac{\sqrt{|y|} \sqrt{|y|(1+ |y|^2)}}{\sqrt{1 + t}}
    \lesssim
    \frac{|y|(1 + |y|)}{\sqrt{1 + t}}
  \end{equation*}
  for all $t > 1/2$ and all $y > 0$. Finally, arguing as in dimensions
  $d \geq 2$, we obtain Theorem~\ref{thm:main-L1} on
  $\Omega = (0,+\infty)$. That is (recalling $D =x/(2t) \Gamma$):
  \begin{equation}
    \label{eq:12}
    \int_0^\infty x \left|
      u(t,x) - \frac{m_\phi x}{t} \Gamma(t,x)
    \right|  \d x \lesssim \frac{M_{1,\phi}}{\sqrt{t}}.
  \end{equation}

  \medskip
  \noindent
  \fbox{Step 2: $x_0\ge0$.} For any $z \in \R$ with
  $|z| \leq M \sqrt{t}$, using Lemma~\ref{lem:Gaussian-L12-difference}
  we have
  \begin{equation*}
    \int_0^\infty \frac{x^2}{t} |\Gamma(t,x) - \Gamma(t,x+z)| \d x \lesssim \frac{|z|}{\sqrt{t}}.
  \end{equation*}
  Hence by the triangle inequality, \eqref{eq:12} implies that for any
  solution $u$ on $(0,+\infty)$ and any $t > 1/2$,
  \begin{equation*}
    \label{eq:13}
    \int_0^\infty x \left|
      u(t,x) - \frac{m_{\phi_0} x}{t} \Gamma(t,x+z)
    \right| \d x
    \lesssim
    \frac{M_{1,\phi_0}}{\sqrt{t}} + \frac{m_{\phi_0}|z|}{\sqrt{t}}.
  \end{equation*}
  where $m_{\phi_0}$ and $M_{1,\phi_0}$ denote moments of the initial
  data $u_0$ using the weight $\phi_0(x) = x$. Now, if $u$ is any
  solution on $\Omega = (x_0,+\infty)$, then $v(t,x) := u(t, x+x_0)$
  is a solution on $(0,+\infty)$, so
  \begin{equation*}
    \int_0^\infty x \left|
      u(t,x+x_0) - \frac{m_\phi x}{t} \Gamma(t,x+z)
    \right| \d x
    \lesssim
    \frac{M_{1,\phi_0}[v]}{\sqrt{t}} + \frac{m_{\phi_0}[v]|z|}{\sqrt{t}},
  \end{equation*}
  where now $m_{\phi_0}[v]$ and $M_{1,\phi_0}[v]$ denote moments of the initial data
  $v_0(x) = u_0(x+x_0)$ on $(x_0,+\infty)$ with respect to the weight
  $\phi_0(x) = x$. Notice that
  \begin{equation*}
    m_{\phi_0}[v] = m_\phi[u],
    \qquad
    M_{1,\phi_0}[v] = M_{1,\phi}[u] - x_0 m_\phi[u],
  \end{equation*}
  which gives
  \begin{equation*}
    \int_0^\infty x \left|
      u(t,x+x_0) - \frac{m_\phi x}{t} \Gamma(t,x+z)
    \right| \d x
    \lesssim
    \frac{M_{1,\phi}}{\sqrt{t}} + \frac{m_\phi |x_0|}{\sqrt{t}} + \frac{m_{\phi}|z|}{\sqrt{t}},
  \end{equation*}
  where $m_\phi$ and $M_{1,\phi}$ denote moments of the initial data
  $u_0$ with respect to $\phi(x) := x - x_0$, as usual. Taking $z = x_0$
  and carrying out a change of variables,
  \begin{equation*}
    \int_{x_0}^\infty (x-x_0) \left|
      u(t,x) - \frac{m_\phi (x-x_0)}{t} \Gamma(t,x)
    \right| \d x
    \lesssim
    \frac{1}{\sqrt{t}} \Big( M_{1,\phi} + m_{\phi}|x_0| \Big),
  \end{equation*}
  which is the inequality in the statement.
\end{proof}

\section{$L^1$ estimates}
\label{sec:l1-pure}

In this section we obtain asymptotic estimates of solutions to the heat equation in the $L^1$ norm, with no weight. These results are not too difficult consequences of the basic weighted $L^1$ results from the previous section. They are especially interesting in dimensions $d = 1,2$, where the weight $\phi$ diverges as $|x| \to +\infty$. In these cases we are able to slightly improve the convergence rate of Theorem~\ref{thm:main-L1} when $\phi$ is removed as a factor. The basic argument is a type of interpolation: in compact sets the mass of all solutions decreases quite fast, so the main contribution comes from sets where $\phi$ is large.

\subsection{$L^1$ estimate in dimension $d \geq 3$}
\label{sec:l1_d3_no_weight}

\begin{proof}[Proof of Theorem \ref{thm:global-L1} for $d \geq 3$]
  We need to show that both occurrences of $\phi$ in Theorem
  \ref{thm:main-L1} can be removed without any change in the decay
  rate, which is not too hard by using that $\phi \leq 1$ in $\Omega$.
  We start by observing that, by Lemma~\ref{lem:convolution},
  \begin{equation}
    \label{eq:moment.Gamma}
    \int_\Omega\Gamma(t,x)|x-x_0|^{2-d}\d x\leq\ird\Gamma(t,x)|x-x_0|^{2-d}\d x\l\le \ird \Gamma(t,x) |x|^{2-d} \d x\lesssim t^{-\frac{d-2}{2}}.
  \end{equation}
  Therefore, using also the bound on $1-\phi$ from~\eqref{eq:phi-bounds-d3} and the bound on negative moments of $u$ from Corollary~\ref{cor:moment-regularisation-d3}, for $t\ge 1$ we have
  \begin{align*}
    &\int_\Omega (1 - \phi(x)) \Big| u(t,x) - m_\phi\phi(x)\Gamma\big(t, x\big) \Big| \d x\\
    &\hskip20pt\lesssim \int_\Omega |x-x_0|^{2-d}\, u(t,x) \d x+m_\phi\int_\Omega |x-x_0|^{2-d}\Gamma\big(t, x\big) \d x\\
    &\hskip20pt\lesssim m_\phi (1+t)^{-\frac{d-2}{2}}\leq  M_{1,\phi} (1+t)^{-\frac{d-2}{2}}
    \lesssim M_{1,\phi} (1+t)^{-\frac{\lambda}{4}},
  \end{align*}
  the last inequality due to our assumption that $\lambda<d-2$. Hence,
  we may remove the outer $\phi$ in the integrand of the
  bound~\eqref{eq:main-L1-d3} from Theorem~\ref{thm:main-L1}, thus
  completing the proof of~\eqref{eq:global-L1-d3-phi} in
  Theorem~\ref{thm:global-L1} for $t\ge1$.

  On the other hand, using the bound~\eqref{eq:phi-bounds-d3} on $1-\phi$ and~\eqref{eq:moment.Gamma},
  \begin{align*}
    \int_\Omega\Big|\phi(x)\Gamma\big(t, x\big)-\Gamma(t,x)\Big|\d x&=\int_\Omega\Gamma(t,x)|\phi(x)-1|\d x\\
    &\lesssim\int_\Omega\Gamma(t,x)|x-x_0|^{2-d}\d x\lesssim (1+t)t^{-\frac{d-2}{2}}.
  \end{align*}
  Hence we can remove the appearance of $\phi$
  in~\eqref{eq:global-L1-d3-phi} to
  obtain~\eqref{eq:global-L1-d3-no_phi} for $t\ge1$.
\end{proof}


\subsection{$L^1$ estimate in dimension $d = 2$}
\label{sec:l1_d2_no_weight}

\begin{proof}[Proof of Theorem \ref{thm:global-L1} for $d=2$]
  We choose an $x_0\in\R^d\setminus\overline{\Omega}$ and note that
  from Lemma~\ref{lem:phi} we know there exists $C > 0$ such that
  \begin{equation*}
    \log|x-x_0| - C \leq \phi(x) \leq \log|x-x_0| + C \qquad \text{for all $x \in \Omega$.}
  \end{equation*}
  We partition the domain of integration in two parts: an inner part
  $\Omega_1$ and an outer part $\Omega_2$ defined by
  \begin{equation*}
    \Omega_1(t):=\{x\in\Omega \mid |x-x_0|< R + t^{\frac{1}{4}}\},\qquad\Omega_2(t) := \{x \in \Omega \mid |x-x_0| \geq R + t^{\frac{1}{4}}\},
  \end{equation*}
  where $R > 0$ is chosen so that $\log R - C > 0$ (for example
  $R := \exp(2C)$).

  In order to bound the integral over $\Omega_1$ we use the
  $L^1$--$L^\infty$ bound in
  Corollary~\ref{cor:dim_2_L^p_regularisation} to obtain, for
  $t\ge1$,
  \begin{equation}
    \label{eq:44}
    \int_{\Omega_1(t)} u(t,x) \d x
    \lesssim
    \frac{m_\phi}{t}|\Omega_1(t)| \lesssim  \frac{m_\phi}{\sqrt{t}}.
  \end{equation}
  Since $\phi(x)\lesssim1+\log |x-x_0|$
  (see~\eqref{eq:phi-bounds-d2}), and
  $k_t \lesssim 1/{(1 + \log(1+t))^2}$
  (see~\Cref{prp:Kt-final}), then
  \begin{align*}
    \int_{\Omega_1(t)} k_t \phi(x) \Gamma(t,x) \dx
    &\lesssim
      \frac{1}{(1 + \log(1+t))^2} \int_{\Omega_1(t)} (1 + \log|x-x_0|)
      \Gamma(t,x) \dx
    \\
    &\lesssim \frac{1}{1 + \log(1+t)} \int_{\Omega_1(t)} \Gamma(t,x)
      \dx \lesssim \frac{1}{\sqrt{t} \log t}.
  \end{align*}
  From this and \eqref{eq:44} we see that the integral over
  $\Omega_1(t)$ is bounded by
  \begin{equation}
    \label{eq:45}
    \int_{\Omega_1(t)} \left|u(t,x) - k_t m_\phi \phi(x)\Gamma(t,x)
    \right| \d x
    \lesssim
    \frac{m_\phi}{\sqrt{t}}
    \leq
    \frac{M_{1,\phi}}{\sqrt{t}}
  \end{equation}
  for all $t \geq 0$, which decays faster than the term
  $(\log (2+t))^{-1} M_{1,\phi} (1+t)^{-\lambda /4}$ in the bound we
  intend to prove (since $\lambda < 2$ by assumption\footnote{If we
    had further information on $\lambda$ and we wanted to optimise
    this argument to allow for $\lambda = 2$ one can easily write
    $R+t^{1/8}$ in the definition of $\Omega_1$ and $\Omega_2$.}).

  For the integral over $\Omega_2(t)$ we use the lower bound
  $\phi(x) \geq \log|x-x_0| - C$ and apply Theorem~\ref{thm:main-L1}
  to obtain, for all $t \geq 2$,
  \begin{align*}
    \int_{\Omega_2(t)}|u(t,x)
    &- k_t m_\phi\phi(x)\Gamma(t,x)|\d x
    \\
    &\lesssim \frac{1}{\log(R+t^{1/4}) - C} \int_{\Omega_2(t)}
      \phi(x) \left| u(t,x) - k_t m_\phi\phi(x)\Gamma(t,x) \right|
      \d x
    \\
    &\lesssim \frac{1}{ \log t}
      \left(
      \frac{m_\phi}{\log{t}}
      +
      \frac{M_{1,\phi} + m_\phi |x_0|}{t^{\lambda/4}}
      \right).
  \end{align*}
  Together with \eqref{eq:45}, this shows the result.
\end{proof}


\subsection{$L^1$ estimate in dimension $d = 1$}

\begin{proof}[Proof of Theorem \ref{thm:global-L1} in $d=1$]
  Assume without loss of generality that $m_\phi=1$.  Similarly to the
  $d=2$ case, we divide the domain of integration according to whether
  $x$ is ``large'' or not. We write, for any $R>0$,
  \begin{equation*}
    \int_{x_0}^\infty \left| u - 2D \right| \d x
    \leq \int_{x_0}^{x_0+R}\left| u -  2D \right| \dx +
    \int_{x_0+R}^\infty \left| u - 2D \right| \dx.
  \end{equation*}
  We bound each integral separately. For the first one we use the
  bound for the heat kernel in the interval $(0,\infty)$ given in
  Lemma \ref{lem:kernel_dim_1_bound} (which by translation gives a
  bound for the kernel on $(x_0, +\infty)$), and the fact that $u-2D$ is
  also a solution of the equation in order to find a rough upper bound
  of $u$: we choose a time $0 < t_0 < t$ to write
  \begin{align*}
    \int_{x_0}^{x_0+R}
    \left| u(t,x) -  2D(t,x) \right| \dx
    &= \int_{x_0}^{x_0+R}
      \left|\int_{x_0}^\infty (u(t_0,y) - 2D(t_0,y))p_\Omega(t-t_0,x,y) \d y
      \right| \dx\\[10pt]
    &\lesssim \int_{x_0}^{x_0+R}
      \int_{x_0}^\infty \left|
      (u(t_0,y) - 2D(t_0,y))
      \right|\frac{y-x_0}{t-t_0}
      \d y \dx\\[10pt]
    &\lesssim
      \frac{R (M_{1,\phi} + |x_0|)}{\sqrt{t_0}(t-t_0)}
  \end{align*}
  due to Theorem~\ref{thm:main-L1}, applicable whenever
  $|x_0| \leq M \sqrt{t_0}$.
  For the second integral we use again  Theorem~\ref{thm:main-L1} to obtain
  \[
    \int_{R+|x_0|}^\infty\left| u -  2D \right| \dx \leq
    \frac{1}{R}\int_{R+|x_0|}^\infty (x-x_0)\left| u -  2D \right|
    \dx
    \lesssim
    \frac{M_{1,\phi} + |x_0|}{R\sqrt{t}}.
  \]
  Again, this application of Theorem~\ref{thm:main-L1} is fine as long
  as $|x_0| \leq M \sqrt{t}$. Choosing $R=\sqrt{t_0}$ and $t_0=t/2$
  yields
  \begin{equation*}
    \int_{x_0}^\infty \left| u(t,x) - 2D(t,x) \right| \d x
    \lesssim \frac{M_{1,\phi} + |x_0|}{t}.
  \end{equation*}
  By a scaling argument, this shows the result also when  $m_\phi \neq 1$.
\end{proof}


\subsection{Asymptotic behaviour of the total mass in all dimensions}
\label{sec:total_mass}

As a consequence of Theorem \ref{thm:global-L1} we can give an
asymptotic expansion of the mass of solutions to
equation~\eqref{eq:heat-ext} up to the first nonconstant term, with
explicit error estimates.

\begin{cor}\label{cor:mass}
  Assume the hypotheses \Cref{thm:main-uniform}, and also that
  $0 \in U$ in the case $d \geq 2$. There exists a constant $C > 0$
  depending only on $\Omega$ such that the total mass of the standard
  solution $u$ of equation \eqref{eq:heat-ext} satisfies the
  following:
  \begin{enumerate}
  \item[\rm (i)] In dimension $d \geq 3$,
    \begin{equation*}
    \left|
      \int_\Omega u(t,x) \d x - m_\phi - K m_\phi t^{1-\frac{d}{2}}
    \right|
    \leq C M_{1,\phi}\, t^{1-\frac{d}{2} - \frac{\lambda}{2d}}
    \qquad \text{for all $t \geq 2$},
  \end{equation*}
  where $K = C^* \int_{\mathbb{R}^N} G(y)|y|^{2-d} \d y$ and
  $\displaystyle C^*=\lim_{|x|\to\infty}(1-\phi(x))|x|^{d-2}$.

\item[\rm (ii)] In dimension $d=2$, for all $t \geq 2$,
  \begin{equation*}
    \label{eq:mass-decay-d2}
    \begin{aligned}
      \displaystyle
      \left |\int_\Omega u(t,x) \d x - \frac{2 m_\phi}{\log t} \right|
      \leq
      &    \frac{1}{ \log t}
      \left(
        \frac{m_\phi}{\log{t}}
        +
        \frac{M_{1,\phi}}{t^{\lambda/4}}
      \right).
    \end{aligned}
  \end{equation*}

\item[\rm (i)] In dimension $d=1$, assuming $\Omega = (0,+\infty)$,
  there exists a constant $C$ such that
  \begin{equation*}
    \label{eq:mass-decay-d1}
    \left | \int_0^\infty u(t,x) \d x - \frac{m_\phi \sqrt{\pi}}{\sqrt{t}} \right|
    \leq \frac{C M_{1,\phi}}{t}
    \qquad \text{for all $t \geq 2$.}
  \end{equation*}

  \end{enumerate}
\end{cor}

\begin{rem}
  In the previous statement we assume $0 \in U$ for simplicity; the
  statement can easily be applied to any $\Omega$ by using the
  translation invariance of solutions, hence writing
  \begin{equation*}
    M_{1,\phi}^{x_0} = \int_\Omega \phi(x) (1+|x-x_0|) u_0(x) \d x
  \end{equation*}
  instead of $M_{1,\phi}$.
\end{rem}

\begin{proof}[Proof of Corollary~\ref{cor:mass}]
  Since the statement is invariant when multiplying $u$ by a factor,
  we will assume that $m_\phi = 1$.

  \medskip
  \noindent
  \fbox{$d \geq 3$.} We follow an idea from~\cite{CEQW2012} which uses
  our $L^1$ convergence result in \Cref{thm:global-L1}. Using the
  conservation law~\eqref{eq:conservation} we can write
  \begin{equation}
    \label{eq:3}
    \begin{array}{c}
    \displaystyle\int_\Omega u(t,x) \d x - m_\phi
    =
    \int_\Omega u(t,x)(1-\phi(x)) \d x= I_1 + I_2,
    \\[10pt]
    \displaystyle I_1=
    \int_\Omega \big( u(t,x) - \Gamma(t,x)\big)
    (1-\phi(x)) \d x,
    \qquad I_2=
    \int_\Omega \Gamma(t,x)(1-\phi(x)) \d x.
    \end{array}
  \end{equation}
  We now estimate the terms $I_1$ and $I_2$. For $I_1$ we may use first the bound~\eqref{eq:phi-bounds-d3} on $\phi$, the standard regularisation property $\|u\|_\infty \lesssim t^{-d/2} \|u_0\|_1$, and then \Cref{thm:global-L1} to get, for any $R>0$ sufficiently large so that $B_R^{\mathrm{c}} = \R^d \setminus B_R\subseteq \Omega$,
  \begin{align*}
    |I_1| &\lesssim
    \int_\Omega \big| u(t,x) - \Gamma(t,x)\big|
    \, |x|^{2-d} \d x
    \\
    &\leq
    \int_{\Omega\cap B_R} \big| u(t,x) - \Gamma(t,x)\big|
    \, |x|^{2-d} \d x
    +
    R^{2-d} \int_{B_R^{\mathrm{c}}} \big| u(t,x) - \Gamma(t,x)\big|
    \d x
    \\
    &\lesssim \Big(
    (1+\|u_0\|_1) t^{-\frac{d}{2}}
    \int_{\Omega \cap B_R} \, |x|^{2-d} \d x
    +
    R^{2-d} M_1 t^{-\frac{\lambda}{4}}
    \Big)
    \\
    &\lesssim
    M_{1,\phi} \left(
    R^2 t^{-\frac{d}{2}}
    +
    R^{2-d} t^{-\frac{\lambda}{4}}\right),
  \end{align*}
  where we used $1 = m_\phi \leq M_{1,\phi}$ and $\|u_0\|_1 = m_0 \leq M_{1,\phi}$. By choosing $R := t^{\frac{1}{2} - \frac{\lambda}{4d}}$ we obtain the following for sufficiently large times $t$:
  \begin{equation}
    \label{eq:8}
    |I_1| \lesssim M_{1,\phi}\, t^{1 - \frac{d}{2} - \frac{\lambda}{2d}}.
  \end{equation}
  For the term $I_2$ in \eqref{eq:3} we write
  \begin{align*}
    I_2 &= C^* \ird \Gamma(t,x) |x|^{2-d} \d x-C^* \int_{\mathbb{R}^d\setminus\Omega} \Gamma(t,x) |x|^{2-d} \d x \\
    &\quad+\int_\Omega \Gamma(t,x)|x|^{2-d} \Big(|x|^{d-2} (1-\phi(x)) - C^*\Big) \d x.
  \end{align*}
  Making the change of variables $x = y \sqrt{2t}$ we see that the first term is just $K t^{1-\frac{d}{2}}$, so we have, using Lemma~\ref{lem:d3-phi-limit} with $x_0=0$,
  \begin{align*}
    |I_2 - K t^{1-\frac{d}{2}}|
    &\leq C^* \int_{\mathbb{R}^d\setminus\Omega} \Gamma(t,x) |x|^{2-d} \d x+
    \ird \Gamma(t,x)|x|^{2-d} \Big|
    |x|^{d-2} (1-\phi(x)) - C^*
    \Big| \d x
    \\
    &\lesssim t^{-\frac d2}+
    \ird \Gamma(t,x)|x|^{1-d} \d x
    \lesssim t^{-\frac d2}+t^{\frac{1}{2}-\frac{d}{2}}
      \lesssim t^{\frac{1}{2}-\frac{d}{2}}
      \lesssim t^{1 - \frac{d}{2} - \frac{\lambda}{2d}}.
  \end{align*}
  (The last inequality holds since we always assume $\lambda < d-2$, so $\lambda < d$). Together with~\eqref{eq:8} and \eqref{eq:3}, this
  shows the result.

  \medskip

  \noindent \fbox{$d = 2$.} A straightforward consequence of Theorem \ref{thm:global-L1} is that
  \begin{align*}
    &\frac{1}{ \log (2+t)}\left(\frac{m_\phi}{\log{(2t +1)}}+\frac{M_{1,\phi}}{(1+t)^{\lambda/4}}\right)
    \\
    &\qquad\gtrsim\int_\Omega \left|u(t,x) -k_t m_\phi \phi(x) \Gamma(t,x) \right| \d x
     \geq\left|\int_\Omega u(t,x) \dx-k_t m_\phi  \int_\Omega\phi(x) \Gamma(t,x) \d x\right|.
  \end{align*}
  From our estimates on $k_t$ in Proposition \ref{prp:Kt-final} and similar estimates on the integral $\int_\Omega \phi(x) \Gamma(t,x)
  \d x$, it is not hard to see that
  \begin{equation*}
    k_t = \frac{4}{(\log t)^2}+O((\log t)^{-3}),\quad\int_\Omega \phi(x) \Gamma(t,x) \d x= \frac12 \log t + O( 1 )\quad\text{as }t \to +\infty,
  \end{equation*}
  which implies that
  \begin{equation*}
    k_t \int_\Omega \phi(x) \Gamma(t,x) \d x= \frac{2}{\log t} + O((\log t)^{-2})\quad\text{as }t \to +\infty.
  \end{equation*}
  Together with our previous estimate, this leads to the estimate in
  the statement.

  \medskip
  \noindent\fbox{$d = 1$.}  The statement from Theorem \ref{thm:global-L1} gives
  \begin{align*}
      \frac{C M_{1,\phi}}{(1+t)}
      &\geq
      \int_0^\infty
      \left|
        u(t,x) - m_1 \frac{x}{t} \Gamma(t,x)
      \right| \d x
      \\
      &\geq
      \left|
        \int_0^\infty
        u(t,x) \dx -
        m_1 \int_0^\infty \frac{x}{t} \Gamma(t,x) \dx
      \right|
      =
      \left|
        \int_0^\infty
        u(t,x) \dx -
        m_1 (\pi t)^{-\frac12}
      \right|.\qedhere
    \end{align*}
  \end{proof}

  The estimates in Corollary \ref{cor:mass} are comparable to results
  in \citet{DominguezTena-RodrigueBernal-2023} in the case of
  Dirichlet boundary conditions. One important difference is that we
  always require a certain finite moment of the initial condition
  $u_0$ ($M_1$ in $d \geq 3$; $M_{1,\log}$ in dimension $2$; and $M_2$
  in dimension $d=1$), which leads to sharper but less general
  results. Results in \citet{DominguezTena-RodrigueBernal-2023} are
  valid for any integrable initial condition, and in particular show
  that there are initial conditions in dimensions $1$ and $2$ for
  which the decay of mass can be very slow. Hence, some conditions on
  the initial data $u_0$ (as the finiteness of a suitable moment,
  which we require) are needed in order to give quantitative estimates
  for the decay.


\section{Uniform estimates}\label{sec:uniform_convergence}

This last section is devoted to the proof of \Cref{thm:main-uniform},
which gives uniform estimates in the whole exterior domain~$\Omega$,
including uniform estimates in relative error if we restrict ourselves
to \emph{inner regions}, for which $|x|\lesssim t^{1/2}$.

\subsection{Uniform convergence in dimension $d \geq 3$}
\label{sec:linfinity_d3}

The idea is to use the $L^1$--$L^\infty$ regularisation property of the heat equation in $\Omega$ in order to transform the $L^1$ estimates in Section~\ref{sec:l1} into $L^\infty$ estimates.

\begin{proof}[Proof of Theorem \ref{thm:main-uniform} for $d \geq 3$]
  Calling $w := u - \phi m_\phi \Gamma$, and using $\Delta \phi = 0$,
  $\p_t \Gamma = \Delta \Gamma$, one readily sees that
  \begin{equation*}
    \left\{
      \
      \begin{aligned}
        \p_t w = \Delta w + 2m_\phi\nabla \phi\cdot \nabla \Gamma,
        \qquad
        &\text{$x\in \Omega$, $t>0$,}
        \\
        w(t,x) = 0,
        \qquad
        &\text{$x \in \p \Omega$, $t > 0$.}
      \end{aligned}
    \right.
   \end{equation*}
  Applying Duhamel's formula from a starting time $t_0 > 1$ gives, for any $t \geq t_0$,
  \begin{equation}
    \label{eq:20}
    w(t,\cdot)
    = S_{t-t_0} w(t_0,\cdot) + 2m_\phi \int_{t_0}^t S_{t-s} (\nabla \phi\cdot
    \nabla \Gamma(s, \cdot)) \d s,
  \end{equation}
  where $S_t$ is the semigroup of the Dirichlet heat equation in
  $\Omega$. Choose $t > 4$ and $t_0 := t/2$. Using the bound from
  \Cref{cor:Linfty_regularisation_d3} in the case $p=1$ and the
  $\phi$-weighted $L^1$ bound from Theorem \ref{thm:main-L1} we can
  estimate the first term:
  \begin{equation}
    \label{eq:39}
    |S_{t-t_0} w (t_0, x)|=|S_{t/2}w(t/2,0)|
    \lesssim
    t^{-\frac{d}{2}} \phi(x)\|\phi w (t/2, \cdot) \|_1\lesssim
    t^{-\frac{d}{2}-\frac\lambda4}\phi(x)  M_{1,\phi}.
  \end{equation}
  In order to bound the second term in \eqref{eq:20} we split the integral in it into two regions:
  \begin{equation*}
    I_1 := \int_{t/2}^{t-1} S_{t-s} (\nabla \phi\cdot
    \nabla \Gamma(s, \cdot)) \d s,
    \qquad
    I_2 := \int_{t-1}^{t} S_{t-s} (\nabla \phi\cdot
    \nabla \Gamma(s, \cdot)) \d s.
  \end{equation*}
For the first one we use the upper bound for the heat kernel in~\eqref{eq:Zhang2} to get
  \begin{align*}
    |I_1|&\leq\int_{t/2}^{t-1}| S_{t-s} (\nabla \phi\cdot\nabla \Gamma(s, \cdot))|\d s\\
    &\leq \int_{t/2}^{t-1}  \int_\Omega p_\Omega(t-s,x,y) |\nabla \phi (y)|\, |\nabla \Gamma(s, y) | \d y \d s\\
    &\lesssim\phi(x) \int_{t/2}^{t-1}  \int_\Omega \phi(y) \Gamma(t-s,(x-y)/c_2)|\nabla \phi (y) |\, |\nabla \Gamma(s, y) | \d y \d s.
  \end{align*}
  We now use the bound~\eqref{eq:gradphi-d3} for $|\nabla\phi|$ plus the estimate
  \begin{equation}
    \label{eq:estimate.gradient.Gamma}
    |\nabla \Gamma(s,y) |\lesssim s^{-(d+1)/2}\frac{|y|}{\sqrt{s}} e^{-\frac{|y|^2}{4 s}}\lesssim s^{-(d+1)/2},
  \end{equation}
  and the fact that $0\le\phi\leq 1$, to obtain
  \begin{equation*}
    |I_1|\lesssim\phi(x) \int_{t/2}^{t-1}  s^{-(d+1)/2}\int_\Omega \Gamma(t-s, (x-y)/c_2)|y-x_0|^{1-d}\d y \d s.
  \end{equation*}
  The symmetry of $\Gamma$ in the spatial variable and the convolution Lemma \ref{lem:convolution} yield
  \begin{align*}
    \int_\Omega \Gamma(t-s, (x-y)/c_2)|y-x_0|^{1-d}\d y&\leq\ird \Gamma(t-s,(x-y)/c_2)|y-x_0|^{1-d}\d y\\
    &\leq\ird \Gamma(t-s,y/c_2)\,|y|^{1-d}\d y\lesssim(t-s)^{(1-d)/2}.
  \end{align*}
  Hence,
  \begin{equation}
  \label{eq:estimate.I1}
    \left | I_1 \right|
    \lesssim
    \phi(x)\,
    t^{-(d+1)/2}
    \int_{t/2}^{t-1} (t-s)^{(1-d)/2} \d s
    \lesssim
    \phi(x)
    \begin{cases}
      t^{1-d},& d>3,\\
      t^{-2}\log t,&d=3.
    \end{cases}
    \end{equation}

    To estimate $I_2$ we use the short-time bound
    \eqref{eq:Zhang2-tsmall} for $p_\Omega$, which implies in
    particular that
    \begin{equation*}
      p_\Omega(t,x,y)\lesssim\frac{\phi(x)}{\sqrt{t}}\,\Gamma( c t,  x-y).
    \end{equation*}
    Combining this with the estimates $0\le\phi\le1$,
    $|\nabla \phi(y)| \lesssim 1$
    and~\eqref{eq:estimate.gradient.Gamma},
    \begin{equation}
      \label{eq:34}
      \begin{split}
        |I_2|&\leq\int_{t-1}^{t}\int_\Omega
               p_\Omega(t-s,x,y)|\nabla\phi(y)|\,|\nabla\Gamma(s,
               y)|\d y\d s
        \\
             &\lesssim\phi(x)\int_{t-1}^t(t-s)^{-\frac{1}{2}}s^{-(d+1)/2}
               \int_\Omega\Gamma(c(t-s),x-y)\d
               y \d s
        \\
             &\lesssim\phi(x) t^{-(d+1)/2} \int_{t-1}^t (t-s)^{-\frac{1}{2}}\d s\lesssim\phi(x) t^{-(d+1)/2}.
      \end{split}
    \end{equation}
  From~\eqref{eq:20}--\eqref{eq:34} we obtain immediately
  \begin{equation*}
    |w(t, x)|\lesssim\phi(x)\big( M_{1,\phi}t^{-\frac{d}{2}-\frac\lambda4}+m_\phi t^{1-d}+m_\phi t^{-(d+1)/2}\big).
  \end{equation*}
  Since $\lambda<d-2$, then $\frac d2+\frac\lambda4<d-1$. If moreover $\lambda\le2$, then $\frac d2+\frac\lambda4 \leq \frac{d+1}2$, and we
  finally get
  \begin{equation*}
        \| w(t, \cdot) \|_\infty
        \lesssim \phi(x)t^{-\frac{d}{2}-\frac\lambda4}M_{1,\phi}
  \end{equation*}
  if $t>4$, since $m_\phi\le M_{1,\phi}$. This inequality is clearly
  also true for $2 \leq t \leq 4$, as can be seen by separately
  estimating the two terms in the expression
  $w = u - \phi m_\phi \Gamma$ (use
  \Cref{cor:Linfty_regularisation_d3}, case $p=1$, for a suitable
  estimate of $u$).
 \end{proof}


 \subsection{Uniform convergence in dimension $d = 2$}
\label{sec:linfinity_d2}

Let us first prove an auxiliary lemma.
\begin{lem}\label{lem:bound_integral_section7.2}
  There exists a constant $C>0$ such that
  \[
    \int_{t/2}^{t-1} (t-s)^{-\frac{1}{2}}(\log(1+t-s))^{-1} \d s \leq Ct^{\frac{1}{2}}\left(\log\left(1+\frac{t}{2}\right)\right)^{-1}\quad\text{  for all $t \geq 4$.}
  \]
\end{lem}
\begin{proof}
	We compute
	\begin{align*}
	 I(t)&=\int_{t/2}^{t-1} (t-s)^{-\frac{1}{2}}(\log(1+t-s))^{-1} \d s = \int_{1}^{t/2} s^{-\frac{1}{2}}(\log(1+s))^{-1} \d s, \\
	  I'(t)&=\frac{1}{2}t^{-\frac{1}{2}}\left(\log\left(1+\frac{t}{2}\right)\right)^{-1}.
	\end{align*}
	Define on the other hand, for some $c > 0$ to be fixed later,
	\[
	H(t):=ct^{\frac{1}{2}}\left(\log\left(1+\frac{t}{2}\right)\right)^{-1}.
	\]
	Then
	\[
	H'(t)=\frac{c}{2}t^{-\frac{1}{2}}\left(\log\left(1+\frac{t}{2}\right)\right)^{-1} - c\frac{t^{\frac{1}{2}}}{2+t}\left(\log\left(1+\frac{t}{2}\right)\right)^{-2}.
	\]
	Since $\dfrac{t^{\frac{1}{2}}}{2+t}\leq t^{-\frac{1}{2}}$, taking $c > 4$ there must exist a time $t_1>0$ such that
	\[
	H'(t)\geq \frac{c}{4}t^{-\frac{1}{2}}\left(\log\left(1+\frac{t}{2}\right)\right)^{-1}= \frac{c}{4}I'(t)>I'(t)\quad\text{for all}\quad t>t_1.
	\]
	Now, since both $I(t)$ and $H(t)$ are uniformly bounded above and below for all $t\in[4,t_1]$, we can choose $c > 4$ and large enough so that
	\[
	I(t)<H(t)\text{ for all }t\in[4, t_1]\quad\text{and}\quad I'(t)<H'(t)\text{ for all }t\geq t_1,
	\]
	implying that $I(t)<H(t)$ for all $t\geq 4$ and	proving our claim.
\end{proof}

\medskip
We treat the case $d=2$ in a similar way as the case of $d \geq 3$,
using the weighted $L^1$ convergence result from Theorem
\ref{thm:main-L1}.

\begin{proof}[Proof of Theorem \ref{thm:main-uniform} for $d=2$]
  The function $ w(t,x) := u(t,x) - k_t \phi(x)m_\phi\Gamma(t,x)$ satisfies
  \[
  \begin{gathered}
      \begin{cases}
      \p_t w(t,x) = \Delta w(t,x)+ m_\phi F(t,x), \quad &x \in \Omega$, $t > 0,
      \\
      w(t,x) = 0, &x \in \p \Omega,\ t > 0,
    \end{cases}
    \\[10pt]
  \text{where }
    F(t,x) :=  2 k_t  \nabla\phi(x) \cdot\nabla\Gamma(t,x)
    - k_t'  \phi(x) \Gamma(t,x).
    \end{gathered}
  \]
  We denote by $S_t$ the heat equation semigroup in $\Omega$ with
  Dirichlet boundary conditions (so that $S_t u_0$ is the solution with initial condition $u_0$ at time $t$). We apply Duhamel's formula
  from a starting time $t_0 > 1$ and we get, for any $t\geq t_0$,
  \begin{equation}
    \label{eq:43}
    w(t,\cdot)
    = S_{t-t_0} w(t_0,\cdot) + m_\phi \int_{t_0}^t S_{t-s} (F(s,\cdot)) \d s.
  \end{equation}
  Take $t \geq 4$ and $t_0 := t/2$. Using the bound from Corollary
  \ref{cor:dim_2_L^p_regularisation} in the case $p=1$ and the $L^1$
  bound from Theorem \ref{thm:main-L1} in dimension $d=2$ we can bound
  the first term:
  \begin{equation}
    \label{eq:39-d2}
    \begin{aligned}
       |S_{t-t_0} w(t_0,\cdot)(x)|&= | S_{t/2} w (t/2, \cdot)(x)|\lesssim\frac{\phi(x)}{t(\log t)^2} \|\phi w (t/2, \cdot) \|_1
        \\
        &\lesssim\frac{\phi(x)}{t(\log t)^2}\left(\frac{m_\phi}{\log{(t+1)}}+\frac{M_{1,\phi} + m_\phi |x_0|}{(1+t)^{\lambda/4}}
        \right).
    \end{aligned}
  \end{equation}
  for all $t \geq 4$. To estimate the second term in~\eqref{eq:43}, we separate $F$ into the two terms
  \begin{equation*}
    F_1(t,x) :=  2 k_t \nabla\phi(x) \cdot \nabla\Gamma(t,x),\qquad F_2(t,x) := k_t'  \phi(x) \Gamma(t,x).
  \end{equation*}
  \noindent\underline{Estimate for $F_1$.} To estimate the term with $F_1$ we divide the integral in two parts:
  \begin{gather*}
    \left | \int_{t/2}^t S_{t-s} (k_s\ \nabla \phi\cdot
      \nabla \Gamma(s,\cdot)) \d s \right|\le I_1 + I_2,\quad\text{where}
    \\
    I_1:=
    \int_{t/2}^{t-1} | S_{t-s} (k_s\ \nabla \phi\cdot
    \nabla \Gamma(s,\cdot))| \d s,\quad I_2= \int_{t-1}^{t} | S_{t-s} (k_s\ \nabla \phi\cdot
    \nabla \Gamma(s,\cdot))| \d s.
  \end{gather*}
  As for $I_1$, there both $s$ and $t-s$ are away from 0, so we may
  use again the kernel bound in \eqref{eq:SC-d2-tlarge-2}, and our
  bound $k_s \lesssim (\log(2+s))^{-2} \lesssim (\log s)^{-2}$ (for
  $s \geq 2$):
  \begin{align*}
    I_1&\leq\int_{t/2}^{t-1} k_s | S_{t-s} (\nabla \phi\cdot\nabla \Gamma(s, \cdot)) | \d s
    \\
    &\lesssim \int_{t/2}^{t-1} \frac{1}{(\log s)^2} \int_\Omega p_\Omega(t-s,x,y)|\nabla\phi(y)|\, |\nabla \Gamma(s, y) | \d y \d s
    \\
    &\lesssim\frac{\phi(x)}{(\log t)^2}\int_{t/2}^{t-1}\int_\Omega\frac{\phi(y)}{(\log (1+ t-s))^2}\Gamma\Big(t-s, \frac{x-y}{c_2}\Big)
    |\nabla\phi(y)|\,|\nabla\Gamma(s,y)|\d y \d s.
  \end{align*}
  Since $|\nabla \phi (y) | \lesssim|y-x_0|^{-1}$, $\phi(y) \lesssim \log (2 + |y-x_0|)$ for all $y \in \Omega$, using also  \eqref{eq:estimate.gradient.Gamma} we get
  \begin{equation}
    \label{eq:32}
    \begin{aligned}
        I_1&\lesssim\frac{\phi(x)}{(\log t)^2}\int_{t/2}^{t-1}\int_\Omega\frac{1}{(\log(1+t-s))^2}\Gamma\Big(t-s,\frac{x-y}{c_2}\Big)
        \frac{\log(2 + |y-x_0|)}{|y-x_0|}s^{-\frac{3}{2}}\d y \d s\\
        &\leq\frac{\phi(x) t^{-\frac{3}{2}}}{(\log t)^2}\int_{t/2}^{t-1}\frac{1}{(\log (1+ t-s))^2}\int_\Omega
        \Gamma\Big(t-s,\frac{x-y}{c_2}\Big)\frac{\log(2 + |y-x_0|)}{|y-x_0|}\d y \d s.
    \end{aligned}
  \end{equation}
  The inner integral can be estimated as follows: for $a, b \geq 0$ we have
  \begin{equation*}
    \log(2 + ab) \leq \log ((2+a)(1+b)) = \log(2+a) + \log (1+b),
  \end{equation*}
  so
  \begin{equation}\label{eq:log_2}
    \log (2 + |y-x_0|) \leq\log (1 + \sqrt{t-s}) + \log \Big(2 + \frac{|y-x_0|}{\sqrt{t-s}}\Big)
  \end{equation}
  and we have, using the convolution \Cref{lem:convolution},
  \begin{align*}
    \int_\Omega&\Gamma\Big(t-s, \frac{x-y}{c_2}\Big)\frac{\log(2 + |y-x_0|)}{|y-x_0|}\d y\\[10pt]
    &\lesssim\int_\Omega\Gamma\Big(t-s, \frac{x-y}{c_2}\Big)\frac{\log(1+\sqrt{t-s})}{|y-x_0|}\d y
    +\int_\Omega\Gamma\Big(t-s, \frac{x-y}{c_2}\Big)\frac{\log\Big(2 + \frac{|y-x_0|}{\sqrt{t-s}}\Big)}{|y-x_0|}\d y\\[10pt]
    &\leq\int_{\R^2}\Gamma\Big(t-s, \frac{y}{c_2}\Big)\frac{\log(1 + \sqrt{t-s}) }{|y|}\d y
    +\int_{\R^2}\Gamma\Big(t-s, \frac{y}{c_2}\Big)\frac{\log\Big(2 + \frac{|y|}{\sqrt{t-s}}\Big)}{|y|}\d y\\[10pt]
    &\lesssim\int_{\R^2}\Gamma\Big(t-s, \frac{y}{c_2}\Big)\frac{\log(1 + \sqrt{t-s}) }{|y|}\d y
    +\int_{\R^2}\Gamma\Big(t-s, \frac{y}{2c_2}\Big)\frac{1}{|y|}\d y\\[10pt]
    &\lesssim(t-s)^{-\frac{1}{2}} (\log(1 + \sqrt{t-s}) + 1)\lesssim(t-s)^{-\frac{1}{2}} \log(1 + t-s).
  \end{align*}
  Hence, from \eqref{eq:32} and Lemma~\ref{lem:bound_integral_section7.2},
  \begin{align*}
    I_1&\lesssim\frac{\phi(x) t^{-\frac{3}{2}}}{(\log t)^2}\int_{t/2}^{t-1}\frac{(t-s)^{-\frac{1}{2}}}{\log (1+ t-s)}\d s
    \lesssim\frac{t^{-\frac{3}{2}}\phi(x)}{(\log t)^2}\frac{t^{\frac{1}{2}}}{\log\Big(1+\frac{t}{2}\Big)}\lesssim \frac{\phi(x)}{t (\log t)^3}
  \end{align*}
  for all $t \geq 4$. For the integral $I_2$, we use the small-time
  bound in Corollary~\ref{cor:dim_2_L^p_regularisation},
  $|\nabla\phi|\leq C$,~\eqref{eq:estimate.gradient.Gamma} and the
  estimate for $k_t$, and compute
    \begin{align*}
    	\int_{t-1}^{t} | S_{t-s} (k_s\ \nabla \phi\cdot\nabla \Gamma(s,\cdot))| \d s&\lesssim\phi(x)\int_{t-1}^t |\log(1+\sqrt{t-s})|^{-1}
    	\|k_s\ \nabla \phi\cdot \nabla \Gamma(s,\cdot) \|_\infty \d s
    	\\
    	&\lesssim\phi(x)k_t\ t^{-\frac{3}{2}}\int_{t-1}^t (\log(1+\sqrt{t-s}))^{-1}\d s
    	\\
    	&=\phi(x)k_t\  t^{-\frac{3}{2}}\int_{0}^1 (\log(1+\sqrt{s}))^{-1}\d s
    	\\
    	&\lesssim\phi(x)(\log(1+t))^{-2}t^{-\frac{3}{2}}.
    \end{align*}
    So the term with $F_1$ can be bounded, for all $t \geq 4$, by
    \begin{equation}\label{eq:21}
        \left | \int_{t_0}^t S_{t-s} (k_s \nabla \phi\cdot
        \nabla \Gamma(s,\cdot)) \d s \right|
        \leq \frac{\phi(x)}{t (1+ \log(1+t))^3}.
    \end{equation}

 \medskip
 \noindent
 \underline{Estimate for $F_2$.} Regarding $F_2$, we again split the integral in two terms and compute
   \[
   \left | \int_{t_0}^t S_{t-s} (k'_s\ \phi\ \Gamma(s,\cdot)) \d s \right|
   \leq
   \underbrace{\int_{t_0}^{t-1} | S_{t-s} (k'_s\ \phi\ \Gamma(s,\cdot))| \d s}_{I_1} + \underbrace{\int_{t-1}^{t} | S_{t-s} (k'_s\phi\Gamma(s,\cdot))| \d s}_{I_2}.
   \]
   Regarding $I_1$, from Lemma \ref{lem:Kt'/Kt} we obtain
   \begin{equation*}
     k_t' = \ddt K_{\frac12 \log(2t)}
     = \frac{1}{2t}
     K'_\tau \big\vert_{\tau = \frac12 \log(2t)}
     \lesssim
     \frac{k_t}{t (1+ \log(1+t))}\leq \frac{1}{t(1+\log(1+t))^3},
   \end{equation*}
   and so, using this estimate,
   \begin{equation}\label{eq:integral_bound_I_1}
     \begin{aligned}
     |I_1|&\lesssim
     \int_{t_0}^{t-1}\int_\Omega \frac{\phi(x)}{s(1+\log(1+s))^3}\Gamma(s,y)\ \frac{\phi^2(y)\Gamma(t-s,(x-y)/c_2)}{(\log(1+t-s))^2} \d y \d s
     \\[10pt]
     &\lesssim
     \frac{\phi(x)}{t^2(1+\log(1+t))^3} \int_{t_0}^{t-1}\int_\Omega \frac{\phi^2(y)\Gamma(t-s,(x-y)/c_2)}{(\log(1+t-s))^2} \d y \d s.
   \end{aligned}
   \end{equation}
	In order to bound the interior integral, similarly to the computations in~\eqref{eq:log_2}, we get
	\[
	\phi^2(y)\lesssim (\log(1+\sqrt{t-s}))^2 + \Big(\log\Big(2+\frac{y-x_0}{\sqrt{t-s}}\Big)\Big)^2.
	\]
	So
	\begin{equation*}
    	\int_\Omega \frac{\phi^2(y)\Gamma(t-s,(x-y)/c_2)}{(\log(1+t-s))^2}\d y\lesssim 1+\int_\Omega\frac{\Big(\log\Big(2+\frac{y-x_0}{\sqrt{t-s}}\Big)\Big)^2}{(\log(1+t-s))^2}\Gamma(t-s,(x-y)/c_2) \d y
	\end{equation*}
	and  using again the convolution \Cref{lem:convolution},
	\begin{align*}
		\int_\Omega\frac{\phi^2(y)\Gamma(t-s,(x-y)/c_2)}{(\log(1+t-s))^2}
        \d  y&\lesssim1+\int_\Omega\frac{\left(\log\left(2+\frac{y}{\sqrt{t-s}}\right)\right)^2}{(\log(1+t-s))^2}
        \Gamma(t-s,y/c_2)\d y\\[10pt]
		&\lesssim 1+\frac{1}{(\log(1+t-s))^2}.
	\end{align*}
	Going back to~\eqref{eq:integral_bound_I_1},
	\[
	   |I_1|\lesssim \frac{\phi(x)}{t^2(1+\log(1+t))^3} \int_{t/2}^{t-1}\Big( 1+\frac{1}{(\log(1+t-s))^2}\Big)\d s\lesssim\frac{\phi(x)}{t(1+\log(1+t))^3}.
	\]
	
	For the second integral $|I_2|$ we use again the small-time
        estimate in
        Corollary~\ref{cor:dim_2_L^p_regularisation}
        and compute
	\begin{align*}
		|I_2|&\lesssim\phi(x)\int_{t-1}^{t} |\log(1+\sqrt{t-s})|^{-1}\| k'_s\ \phi \Gamma(s, \cdot) \|_\infty \d s\\
		&\lesssim\phi(x)k'_t\ t^{-\frac{1}{2}}\int_{t-1}^{t} |\log(1+\sqrt{t-s})|^{-1}\d s\\
		& \lesssim\phi(x)k'_t\  t^{-\frac{1}{2}}\lesssim\frac{\phi(x)}{t^{\frac{3}{2}} (1+ \log(1+t))^3}.
	\end{align*}
	So the term with $F_2$ can be bounded by
	\begin{equation}\label{eq:22}
		\left| \int_{t/2}^t S_{t-s} (k'_s\  \phi\ \Gamma(s,\cdot)) \d s \right|\leq \frac{\phi(x)}{t (1+ \log(1+t))^3}.
	\end{equation}

   We can now use
   \eqref{eq:39-d2},~\eqref{eq:21} and \eqref{eq:22} to write
   \begin{equation*}
     |w(t,x)|\lesssim\frac{\phi(x)}{t(1+\log(1+t))^2}\left(\frac{m_{\phi}}{1+\log(1+t)}+\frac{M_{1,\phi}+m_\phi|x_0|}{(1+t)^{\lambda/4}}\right).
   \end{equation*}
   This shows the result for all $t\geq 4$. It is also clearly true for $0 < t \leq 4$.
 \end{proof}


 \subsection{Uniform convergence in dimension $d = 1$}
\label{sec:linfinity_d1}

We write now the proof of Theorem \ref{thm:main-uniform} in $d=1$.

\begin{proof}[Proof of Theorem \ref{thm:main-uniform} in $d=1$]
  We assume without loss of generality that $m_\phi= 1$.  Again, the
  idea is similar to the cases $d \geq 3$ and $d = 2$, using the $L^1$
  behaviour of the difference $w(t,x):=u(t,x)- 2D(t,x)$. This case is
  even simpler since $w$ is itself a solution of the heat equation, so
  the $L^1-L^\infty$ regularisation of the heat equation directly
  applies. Using Lemma~\ref{lem:regularisation-d1} and
  Theorem~\ref{thm:main-L1}, we can also compute
  \[
    \|w\|_{L^\infty(\Omega)}
    \lesssim
    t^{-\frac{3}{2}} \ \phi(x)\ \|\phi(\cdot)\ w(t/2,
  \cdot)\|_{L^1(\Omega)}
  \lesssim
  \frac{\phi(x)}{t^{2}} (M_{1,\phi} + m_{\phi}|x_0|) .\qedhere
  \]
\end{proof}


\subsection*{Acknowledgments}

J.~Cañizo and A.~Gárriz were supported by grants PID2020-117846GB-I00,
PID2023-151625NB-I00, RED2022-134784-T, and the IMAG María de Maeztu
grant CEX2020-001105-M, all of them funded by MICIU/AEI/10.13039/501100011033.

A.~Gárriz was also supported by the Juan de la Cierva grant
JDC2022-048653-I, funded by MICIU/AEI/10.13039/501100011033 and by the
European Union ``NextGenerationE''/PRTR, and by a post-doctoral
fellowship of the ANR project DEEV ANR-20-CE40-0011-01.

F.\,Quir\'os was supported by grants PID2020-116949GB-I00,
PID2023-146931NB-I00, RED2022-134784-T, and the ICMAT Severo Ochoa
grant CEX2023-001347-S, all of them funded by
MICIU/AEI/10.13039/501100011033, and also by the Madrid Government
(Comunidad de Madrid – Spain) under the multiannual Agreement with UAM
in the line for the Excellence of the University Research Staff in the
context of the V PRICIT (Regional Programme of Research and
Technological Innovation).

We thank Arnaud Guillin for comments on the logarithmic Sobolev
inequalities we are using in this paper.

\bibliography{bibliography}

\end{document}